\documentclass{amsart}

\usepackage[draft=false,pagebackref]{hyperref}
\usepackage{esint,amssymb,bm}
\newtheorem{thm}[equation]{Theorem}
\newtheorem{lem}[equation]{Lemma}

\theoremstyle{definition}
\newtheorem{dfn}[equation]{Definition}
\newtheorem{rmk}[equation]{Remark}

\numberwithin{equation}{section}
\numberwithin{figure}{section}

\newcommand\abs[2][empty]{\csname#1\endcsname \lvert{#2}\csname#1\endcsname\rvert}
\newcommand\doublebar[2][empty]{\csname#1\endcsname \lVert{#2}\csname#1\endcsname\rVert}

\newcommand\myeqref[1]{\texorpdfstring{\eqref{#1}}{(\ref*{#1})}}
\newcommand\mat[1]{\bm{#1}}
\newcommand\arr[1]{\bm{\dot{#1}}}
\newcommand\dist{\mathop{\mathrm{dist}}\nolimits}
\newcommand\Div{\mathop{\mathrm{div}}\nolimits}
\newcommand\Tr{\mathop{\smash{\arr{\mathrm{Tr}}}\vphantom{T}}\nolimits}
\newcommand\Trace{\mathop{\mathrm{Tr}}\nolimits}
\newcommand\M{\mathop{\smash{\arr{\mathrm{M}}}\vphantom{M}}\nolimits}

\newcommand\supp{\mathop{\mathrm{supp}}\nolimits}

\newcommand\re{\mathop{\mathrm{Re}}\nolimits}

\newcommand\RR{\mathbb{R}} \let\R\RR
 \let\C\CC
 
 \let\N\NN
\newcommand\1{\mathbf{1}}
\newcommand\D{\mathcal{D}}
\newcommand\s{\mathcal{S}}
\newcommand\e{\vec{e}}

\newcommand\pureH{\parallel}
\newcommand\semiH{\vert}

\newcommand\dmn{{n+1}}
\newcommand\pdmn{{(n+1)}}
\newcommand\dmnMinusOne{n}

\begin{document}

\title[Square function estimates on layer potentials]{Square function estimates on layer potentials for higher-order elliptic equations}

\author{Ariel Barton}
\address{Ariel Barton, 202 Math Sciences Bldg., University of Missouri, Columbia, MO 65211}
\email{bartonae@missouri.edu}
%\thanks{}

\author{Steve Hofmann}
\address{Steve Hofmann, 202 Math Sciences Bldg., University of Missouri, Columbia, MO 65211}
\email{hofmanns@missouri.edu}
\thanks{Steve Hofmann is partially supported by the NSF grant DMS-1361701.}

\author{Svitlana Mayboroda}
\address{Svitlana Mayboroda, Department of Mathematics, University of Minnesota, Minneapolis, Minnesota 55455}
\email{svitlana@math.umn.edu}
\thanks{Svitlana Mayboroda is partially supported by the Alfred P. Sloan Fellowship, the NSF CAREER Award DMS 1056004,  the NSF INSPIRE Award DMS 1344235, and the NSF Materials Research Science and Engineering Center Seed Grant.}

\subjclass[2010]{Primary
35J30, %  	Higher-order elliptic equations
Secondary
31B10, %   	Integral representations, integral operators, integral equations methods
35C15% Integral representations of solutions
}

\begin{abstract}
In this paper we establish square-function estimates on the double and single layer potentials for divergence-form elliptic operators, of arbitrary even order $2m$, with variable $t$-independent coefficients in the upper half-space. This generalizes known results for variable-coefficient second-order operators, and also for constant-coefficient higher-order operators.
\end{abstract}

\keywords{Elliptic equation, higher-order differential equation, square function estimates, layer potentials}

\maketitle 
% amsart wants \maketitle after the subject class and the abstract
% jems wants it before

\tableofcontents

\section{Introduction}
%\numberwithin{equation}{section}

With this paper we continue towards the grand goal of resolving the Dirichlet and Neumann problems for general divergence-form higher order elliptic operators with $L^p$ data. The investigation of the second-order case has essentially spanned the past three decades in the subject, drawing from and giving back many cutting-edge tools of harmonic analysis, and by now the real coefficient case is relatively well understood. However, even the simplest higher order operators, such as the bilaplacian, presented outstanding difficulties: for instance, the sharp range of $p$ for which the Dirichlet problem for the bilaplacian is well-posed in $L^p$ is still not known in high dimensions.
And, to the best of our knowledge, no $L^p$ well-posedness results are currently available in the general variable coefficient case.

In this work we aim to develop the method of layer potentials for general divergence form higher order elliptic operators. The main results of the present paper are square function estimates for single and double layer potentials in $L^2$ and the corresponding Sobolev spaces. We remark that one of the key difficulties in such a general context lies in the definition of suitable layer potentials and, more generally, of Dirichlet and Neumann boundary data, as in the higher order case there is considerable ambiguity, some choices leading to ill-posed problems.  Our approach is new even in the constant coefficient context, but is carefully crafted to handle the general case.

Let us discuss the background and the results in more detail.

In this project we study elliptic differential operators of the form
\begin{equation}\label{eqn:divergence}
Lu = (-1)^m \sum_{\abs{\alpha}=\abs{\beta}=m} \partial^\alpha (A_{\alpha\beta} \partial^\beta u),\end{equation}
for $m\geq 2$, with general bounded measurable coefficients. As mentioned above, contrary to the second order case, most of the known well-posedness results for higher order boundary value problems have been established only in the case of constant coefficients (see, for example, \cite{Ver90,PipV95A,Ver96,She06A,She06B,KilS11b,MitM13A}, or the survey paper \cite{BarM15p}),  or concern boundary-value problems such as the Dirichlet problem
\begin{equation}\label{eqn:introduction:Dirichlet}
Lu = (-1)^m \sum_{\abs{\alpha}=\abs{\beta}=m} \partial^\alpha (A_{\alpha\beta} \partial^\beta u)=0 \text{ in }\Omega,
\qquad
\nabla^{m-1} u= \arr f\text{ on }\partial\Omega\end{equation}
where the boundary data $\arr f$ lies in a fractional smoothness space. See \cite{Agr07,MazMS10,BreM13}. We are interested in the Dirichlet problem~\eqref{eqn:introduction:Dirichlet}, with variable coefficients, in the classical case where the boundary data $\arr f$ lies in $L^2(\partial\Omega)$.

\typeout{}\typeout{Correct the reference to \cite{BarM15p} once we know how the journal is to be published}\typeout{}

\subsection{The method of layer potentials, general framework}

Classic tools for solving second-order  boundary-value problems are the double and single layer potentials, given by
\begin{align}
\label{eqn:introduction:D}
\D^{\mat A}_\Omega f(X) &= \int_{\partial\Omega} \overline{\nu\cdot \mat A^*(Y)\nabla_{Y} E^{L^*}(Y,X)} \, f(Y)\,d\sigma(Y)
,\\
\label{eqn:introduction:S}
\s^{\mat A}_\Omega g(X) &= \int_{\partial\Omega}E^{L}(X,Y) \, g(Y)\,d\sigma(Y)
\end{align}
where $\nu$ is the unit outward normal to~$\Omega$ and where $E^L(X,Y)$ is the fundamental solution for the operator~$L$. Making sense of formulas \eqref{eqn:introduction:D} and~\eqref{eqn:introduction:S} in our general context is one of the key tasks in its own right, and we will return to it below.

It may be shown that, for any nice functions $f$ and $g$ defined on~$\partial\Omega$, the functions $u=\D^{\mat A}_\Omega f$ or $u=\s^{\mat A}_\Omega g$ satisfy $Lu=0$ away from~$\partial\Omega$.
The classic method of layer potentials for solving boundary-value problems is to show that $\s^{\mat A}_\Omega$ (or $\D^{\mat A}_\Omega$) satisfies some estimate, for example,
\begin{equation*}\int_\Omega \abs{\nabla \partial_\dmn \s^{\mat A}_\Omega g(X)}^2\dist(X,\partial\Omega)\,dX\leq C \doublebar{g}_{L^2(\partial\Omega)},\end{equation*}
and to show that the operator $g\mapsto \s^{\mat A}_\Omega g\big\vert_{\partial\Omega}$ is invertible from some space to another, for example, $L^2(\partial\Omega)\mapsto \dot W^2_1(\partial\Omega)$. Then the function $u= \s^{\mat A}_\Omega \bigl((\s^{\mat A}_\Omega \big\vert_{\partial\Omega})^{-1} f\bigr)$ is a solution to the Dirichlet regularity problem
\begin{equation*}Lu=0\text{ in }\Omega,\quad u=f \text{ on }\partial\Omega, \quad \int_\Omega \abs{\nabla \partial_\dmn u(X)}^2 \dist(X,\partial\Omega)\,dX\leq C \doublebar{f}_{\dot W^2_1 (\partial\Omega)}.\end{equation*}
This method has been used in
\cite{Ver84,DahK87,FabMM98,Zan00} in the case of harmonic functions (that is, the case $\mat A=\mat I$ and $L=-\Delta$). This method has also been used to study more general second order problems in
\cite{AlfAAHK11, Bar13, BarM13p,Ros13, HofKMP15B, HofMayMou15,HofMitMor} under various assumptions on the coefficients~$\mat A$. Layer potentials have been used in other ways in \cite{KenR09,Rul07,Mit08,Agr09,MitM11,AusM14}. In particular, the second-order double and single layer potentials have been used to study higher-order differential equations in \cite{PipV92,BarM13}.

\subsection{Outline of the main results}
In this paper we begin to generalize this method to the case of higher-order equations by defining the double and single layer potentials $\D^{\mat A}$ and $\s^{\mat A}$ for higher-order equations (see Section~\ref{sec:potentials}), and then by establishing some bounds on these potentials under certain conditions on the coefficients~$\mat A$. We plan to establish invertibility of layer potentials for some~$\mat A$, and thus well-posedness of the corresponding boundary value problems, in a forthcoming paper.

Even in the case of second-order equations, some regularity assumption must be imposed on the coefficients $\mat A$ in order to ensure well-posedness of boundary-value problems. See the classic example of Caffarelli, Fabes, and Kenig \cite{CafFK81}, in which real, symmetric, bounded, continuous, elliptic coefficients $\mat A$ are constructed for which the Dirichlet problem with $L^2$ boundary data is not well-posed in the unit disk.
A common starting regularity condition is $t$-independence, that is,
\begin{equation}\label{eqn:t-independent}\mat A(x,t)=\mat A(x,s)=\mat A(x) \quad\text{for all $x\in\R^n$ and all $s$, $t\in\R$}.\end{equation}
Boundary value problems for such coefficients have been investigated extensively in domains $\Omega$ where the distinguished $t$-direction is always transverse to the boundary, that is, $\Omega=\{(x,t):t>\varphi(x)\}$ for some Lipschitz function~$\varphi$. See, for example, \cite{JerK81A,FabJK84,KenP93,AlfAAHK11,AusAH08,AusAM10,AusA11,AusR12,AusM14,HofKMP15A,HofKMP15B,BarM13p}. %A variant is to consider coefficients constant in a radial direction ($\mat A(X)=\mat A(rX)$ for any $r>0$ and any $X\in\R^\dmn$) and starlike Lipschitz domains; see \cite{KenP93}.
(In two dimensions some well-posedness results are available even if the distinguished direction is not transverse to the boundary; see \cite{KenKPT00,Rul07,Bar13}.)

The main result of this paper is the following theorem.
\begin{thm}\label{thm:square}
Suppose that $L$ is an elliptic operator associated with coefficients $\mat A$ that are $t$-independent in the sense of formula~\eqref{eqn:t-independent} and satisfy the ellipticity conditions \eqref{eqn:elliptic} and~\eqref{eqn:elliptic:bounded}.

Then the single and double layer potentials $\s^{\mat A}$ and $\widetilde\D^{\mat A}$ in the half-space, %initially defined on the dense subspaces $\mathfrak{N}$ and $\mathfrak{D}$ of Section~\ref{sec:whitney},
as defined by formulas \eqref{dfn:D:tilde} and~\eqref{eqn:S:fundamental},
extend to operators that satisfy the bounds
\begin{align}
\label{eqn:S:square}
\int_{\R^n}\int_{-\infty}^\infty \abs{\nabla^m \partial_t\s^{\mat A} \arr g(x,t)}^2\,\abs{t}\,dt\,dx
	& \leq C \doublebar{\arr g}_{L^2(\R^n)}^2
,\\
\label{eqn:D:tilde:square}
\int_{\R^n}\int_{-\infty}^\infty \abs{\nabla^m \partial_t \widetilde \D^{\mat A} \arr f(x,t)}^2\,\abs{t}\,dt\,dx
	& \leq C \doublebar{\arr f}_{\dot W\!A^2_{m,\semiH}(\R^n)}^2
	= C \doublebar{\arr f}_{L^2(\R^n)}^2
\end{align}
for all $\arr g\in {L^2(\R^n)}$ and all $\arr f \in {\dot W\!A^2_{m,\semiH}(\R^n)}$, where $C$ depends only on the dimension $\dmn$ and the ellipticity constants $\lambda$ and $\Lambda$ in the bounds \eqref{eqn:elliptic} and~\eqref{eqn:elliptic:bounded}.
\end{thm}

The space ${\dot W\!A^2_{m,\semiH}(\R^n)}\subset L^2(\R^n)$ will be defined in Section~\ref{sec:whitney}.

We conjecture that this theorem may be generalized from the half-space to Lipschitz graph domains, but the method of proof at the moment requires the extra structure of~$\R^\dmn_+$.
%In a forthcoming paper, we hope to establish invertibility of layer potentials for at least some classes of coefficients~$\mat A$ (in particular, self-adjoint coefficients), and thus to establish well-posedness of at least some boundary value problems in the half-space~$\R^\dmn_+$.
In the case of second-order operators (the case $m=1$), bounds in the upper half-space may be immediately extended to bounds in domains above Lipschitz graphs via a change of variables, and so extra arguments are not necessary. In the higher-order case, this is not true, as the divergence form~\eqref{eqn:divergence} is not preserved under changes of variables. (A different form of higher-order operator is preserved under changes of variables; such operators were investigated in \cite{BarM13}.)

\subsection{Boundedness of layer potentials for second order elliptic operators}
We now turn to the history of this problem. A reader familiar with the second order case may skip this subsection.
As discussed above, layer potentials have been used extensively in the theory of second-order and constant coefficient boundary value problems. Thus, boundedness results for layer potentials have long been of interest. The celebrated result of Coifmann, McIntosh and Meyer \cite{CoiMM82} established boundedness of the Cauchy integral on a Lipschitz curve; this implies that the operators $f\mapsto \D_\Omega^{\mat I} f\big\vert_{\partial\Omega}$ and $g\mapsto \nu\cdot\nabla\s_\Omega^{\mat I} g\big\vert_{\partial\Omega}$ are bounded $L^2(\partial\Omega)\mapsto L^2(\partial\Omega)$, where $\Omega$ is a Lipschitz domain and where $\mat A=\mat I$ is the identity matrix (that is, where $L=-\Delta$ is the Laplace operator). From there many other bounds on harmonic layer potentials may be derived. For example, bounds on $L^p(\partial\Omega)$, $1<p<\infty$, follow from classical Calder\'on-Zygmund theory, and bounds on layer potentials in the domain $\Omega$ may be established using bounds on~$\partial\Omega$; see, for example, \cite{DahV90}.

In the case of second-order equations with variable $t$-independent coefficients, a number of boundedness results have been established. In \cite{KenR09}, Kenig and Rule established that in dimension $\dmn=2$, layer potentials for operators with real-valued coefficients are bounded on $L^p(\partial\Omega)$ for $\Omega$ the domain above a Lipschitz graph, and in \cite{Rul07} this result was extended to bounded Lipschitz domains and Lipschitz graph domains with arbitrary orientation. In \cite{AlfAAHK11}, boundedness of layer potentials on $L^2(\partial\Omega)$ was established in arbitrary dimensions,
in the domain above a Lipschitz graph, for coefficients that are real-valued and symmetric. A stability result was also established; that is, if solutions and layer potentials for some operator $L_0$ have certain boundedness and invertibility properties on $L^2(\partial\Omega)$, then so do layer potentials for any operator $L_1$ whose coefficients $\mat A_1$ are $t$-independent and near (in $L^\infty$) to those of~$L_0$. (This result required a local H\"older continuity estimate for solutions to $L_0u=0$; this estimate is always valid if $\mat A_0$ is real-valued but may not be valid for complex~$\mat A_0$.)

More generally, in \cite{Ros13} Ros\'en showed that layer potentials are always bounded on $L^2(\partial\Omega)$, for $\Omega$ the domain above a Lipschitz graph, provided that the coefficients of the associated operator are $t$-independent, and also that solutions to $Lu=0$ are continuous and satisfy the local bound
\begin{equation*}\abs{u(X)}\leq C\biggl(\fint_{B(X,r)}\abs{u}^2\biggr)^{1/2}\end{equation*}
whenever $Lu=0$ in $B(X,r)$. (The local H\"older continuity requirement, used in \cite{AlfAAHK11} and in many other papers, is a stronger requirement than this local bound. The local boundedness estimate is necessary for Ros\'en's construction of the fundamental solution $E^L(X,Y)$, and thus for the formulas \eqref{eqn:introduction:D} and~\eqref{eqn:introduction:S} to be meaningful; he also showed that, even without local boundedness, the double and single layer potentials may be continued analytically to bounded operators for $t$-independent coefficients~$\mat A$.) Ros\'en's results built on an alternative approach to boundary-value problems involving semigroups \cite{AusAH08, AusAM10}; essentially he established that layer potentials are equal to certain operators studied in \cite{AusAM10}, and thus the boundedness results therein apply. The results of \cite{AusAM10,Ros13} extend to the case of elliptic systems.

In the case of two dimensions, or of smooth coefficients, standard Calder\'on-Zygmund theory allows for straightforward generalization of $L^2$ bounds to $L^p$ bounds, $1<p<\infty$. In the case of rough coefficients in higher dimensions, new arguments are necessary to bound the layer potentials \eqref{eqn:introduction:D} and~\eqref{eqn:introduction:S} on $L^p(\partial\Omega)$. Some such arguments are presented in various papers, in particular in \cite{HofKMP15A,HofMitMor,HofMayMou15}.

In the case of scalar equations, Ros\'en's $L^2$ boundedness result was later established another way, without semigroups, by Grau de la Herran and Hofmann in \cite{GraH14p}.
As in \cite{AlfAAHK11}, they required that solutions to $Lu=0$ be locally H\"older continuous, and in particular locally bounded.
In this paper we will closely follow their approach. We will need to confront a number of additional difficulties that arise in the case of higher-order equations. However, one significant advantage of the higher-order setting is that local H\"older continuity is automatic in the case of operators of very high order, and there are established techniques to generalize to operators of low or moderate order (see \cite{AusHMT01,Bar14p} or Section~\ref{sec:high}); thus, our Theorem~\ref{thm:square} is valid without any assumptions on local boundedness or H\"older continuity of solutions.

\subsection{Layer potentials for higher order operators: known approaches and new ideas}
\label{sec:introduction:higher}

Turning to the history of higher order problems, we recall that an interesting first step lies in even defining layer potentials in the higher order case. In particular, the prototypical higher order operator, the bilaplacian $\Delta^2$, can be viewed in two ways: either as an operator in the divergence form~\eqref{eqn:divergence}, $\Delta^2=\sum_{j,k}\partial_{jk}(\partial_{jk})$, or as a composition of two second order operators (two copies of $-\Delta$). Many papers have used potentials based on a formulation of higher order operators as compositions; see \cite{DahKV86,Ver90,PipV92,BarM13}. A somewhat different approach is necessary if we view $\Delta^2$ as a divergence form operator, or if we seek to generalize to other such operators.

We begin with Neumann boundary values.
Notice that the Neumann boundary values $\nu\cdot \mat A^*\nabla E^{L^*}$ of the fundamental solution appear in the definition of the second order double layer potential. In fact, layer potentials are deeply connected to Dirichlet and Neumann boundary values of solutions in other ways; for example, if $Lu=0$ in~$\Omega$ for some second-order operator~$L$, then $u$ satisfies the Green's formula
\begin{equation}\label{eqn:green:introduction}\1_\Omega \, u = -\D^{\mat A}_\Omega (u\big\vert_{\partial\Omega}) + \s^{\mat A}_\Omega (\nu\cdot \mat A\nabla u).\end{equation}
That is, such a function $u$ is equal to the double layer potential of its Dirichlet boundary values added to the single layer potential of its Neumann boundary values.

The formulation of Neumann boundary data for higher-order equations is an interesting and very difficult question in its own right. It has often been based on an integration by parts: for sufficiently nice domains~$\Omega$, operators~$L$, and test functions $w$ and~$v$,
\begin{equation}
\label{eqn:Neumann:traditional}
\int_\Omega w\,Lv = \sum_{\abs\alpha=\abs\beta=m}\int_\Omega \partial^\alpha w\, A_{\alpha\beta} \, \partial^\beta v +\sum_{j=0}^{m-1} \int_{\partial\Omega} \partial_\nu^j w\, B_j^{\mat A} v \,d\sigma \end{equation}
for some functions $B_j^{\mat A} v$,
where $\partial_\nu^j$ is the $j$th derivative in the normal direction. (If desired, exact formulas for the functions $B_j^{\mat A}v$ in terms of the higher derivatives of~$v$ may be computed; if $L=\Delta^2$ then formulas for $B_j^{\mat A} v$ may be found in \cite{CohG85,Ver05} or in Section~\ref{sec:data:history} below, and in the case of general constant-coefficient operators an explicit formula may be found in \cite[Proposition~4.3]{MitM13A}.)

It is very natural to regard the array $\{\partial_\nu^j w\}_{j=0}^{m-1}$ as the Dirichlet boundary values of~$w$. Then the array of functions $\{B_j^{\mat A} v\}_{j=0}^{m-1}$ may be regarded as the Neumann boundary values of~$v$. The Neumann problem for the biharmonic function $\Delta^2$, with this formulation of boundary data, was studied in \cite{CohG85,Ver05,She07A,MitM13B}. The Neumann problem for more general constant-coefficient operators was studied in \cite{Ver10,MitM13A}, and for some classes of variable coefficient operators in \cite{Agr07}. We remark that a given operator $L$ may be associated to more than one coefficient matrix~$\mat A$, and that each choice of $\mat A$ will give rise to different boundary operators $B_j^{\mat A}$. We will provide more details and a specific example of these different boundary operators (for $L=\Delta^2$) in Section~\ref{sec:data:history};
several of them are physically relevant (in different contexts!) and some even lead to ill-posed problems.

Going further, if $Lu=0$ in $\Omega$ and $E^{L^*}(X,Y)$ is the fundamental solution to $L^*$ (so that $L^*_Y E^{L^*}(Y,X)=\delta_X(Y)$), then
\begin{align*}
\1_\Omega(X)\,u(X)
&= \int_\Omega u\, \overline{L^* E^{L^*}(\,\cdot\,,X)}
\\&=
	\sum_{j=0}^{m-1} \int_{\partial\Omega} \partial_\nu^j u \,\overline{ B_j^{\mat A^*} E^{L^*}(\,\cdot\,,X)} \, d\sigma
	- \int_{\partial\Omega} B_j^{\mat A} u\, \overline{\partial_\nu^j E^{L^*}(\,\cdot\,,X)}\,d\sigma
.\end{align*}
This naturally suggests the two multiple layer potentials
\begin{align*}
\D^{\mat A}_\Omega \arr f(X)
&=
	\sum_{j=0}^{m-1} \int_{\partial\Omega} \overline{ B_j^{\mat A^*} E^{L^*}(\,\cdot\,,X)} \, f_j \,d\sigma,
\\
\s^{\mat A}_\Omega \arr g(X) &=
	\sum_{j=0}^{m-1} \int_{\partial\Omega} \overline{\partial_\nu^j E^{L^*}(\,\cdot\,,X)}\,g_j\, d\sigma.\end{align*}
Notice that in the higher-order case, layer potentials take as input an array of several functions. Also, this formulation of layer potentials generalizes the Green's formula~\eqref{eqn:green:introduction}.
Layer potentials constructed in this way, from an integration by parts against the fundamental solution, have been used in \cite{CohG83,CohG85,Ver05,MitM13B} to study the biharmonic operator $\Delta^2$ (and in particular the associated Neumann problem), and in \cite{Agm57,MitM13A} to study more general constant-coefficient operators; therein $L^p$ boundedness results for layer potentials have been established. (Some boundedness results for inputs in fractional smoothness spaces were established in \cite{MitM13A}.)

% Layer potentials don't seem to show up in \cite{MazMS10}.

Our formulation of the Neumann boundary values of a solution, and thus layer potentials, is \emph{different}. Specifically, observe that the different terms $\partial_\nu^j u$ exhibit different degrees of smoothness; if $\nabla^{m-1} u\in L^2(\partial\Omega)$, for example, and $\partial\Omega$ is sufficiently smooth, then we expect $\partial_\nu^j u$ to lie in the Sobolev space $\dot W^2_{m-1-j}(\partial\Omega)$ of functions with gradients of order $m-1-j$. Furthermore, if $\nabla^m u\in L^2(\partial\Omega)$, then we generally expect the Neumann boundary terms $B_j^{\mat A} u$ to lie in \emph{negative} smoothness spaces (specifically, we expect $B_j^{\mat A} u\in \dot W^2_{j+1-m}(\partial\Omega)$, and so only $B_{m-1}^{\mat A} u$ lies in~$L^2(\partial\Omega)$). See Section~\ref{sec:data:history} for an example.

This is somewhat problematic in the case of Lipschitz and other non-smooth domains, as higher smoothness spaces and negative smoothness spaces are difficult to formulate. Furthermore, dealing with mixed orders of smoothness is difficult even in smooth domains. To avoid these difficulties, we will prefer to regard $\nabla^{m-1} u$, rather than $\{\partial_\nu^j u\}_{j=0}^{m-1}$, as the Dirichlet boundary values of~$u$; this will allow us to formulate a similarly homogeneous notion of Neumann boundary data. The latter has the advantage of working with elements of the same degree of smoothness and being naturally adaptable to our general context. However, explicit formulas for Neumann boundary data can only rarely be obtained; we treat the entire package of Neumann data as a linear functional on a suitable Sobolev space. See a detailed discussion and an example in Section~\ref{sec:data:history}. We have also formulated layer potentials based on this notion of boundary data; see Section~\ref{sec:potentials}. Our potentials thus take as input arrays of functions in homogeneous spaces; notice the $L^2$ norms on the right-hand sides of the bounds \eqref{eqn:S:square} and~\eqref{eqn:D:tilde:square}.

A Green's formula involving homogeneous boundary data has been used in \cite{PipV95B,Ver96}. However, this Green's formula was formulated in terms of derivatives of order $2m-1$, and as such does not lend itself to formulation of Neumann boundary data or the natural division into double and single layer potentials. Furthermore, their construction used some delicate integrations by parts not available in the variable coefficient case, and so our formulation of layer potentials is of necessity somewhat different and more abstract.

\subsection{Our method and outline of the paper}
The remainder of this paper will be devoted to a proof of Theorem~\ref{thm:square}. Specifically, we will define our terminology in Section~\ref{sec:dfn}. We will provide a few preliminary arguments, mainly involving the theory of solutions to higher-order equations, in Section~\ref{sec:preliminary}. We will show that the bounds \eqref{eqn:S:square} and~\eqref{eqn:D:tilde:square} follow from more convenient bounds (specifically, bounds involving derivatives in the $t$-direction only) in Section~\ref{sec:Theta}; we will also define new operators $\Theta_t^D$ and $\Theta_t^S$ that are somewhat easier to work with.
The proof of Theorem~\ref{thm:square} will make extensive use of $T1$ and $Tb$ theorems; we will state the theorems we will need (taken from \cite{ChrJ87} and \cite{GraH14p}) in Section~\ref{sec:Tb}. The remaining sections of the paper will be devoted to showing that $\Theta_t^D$ and $\Theta_t^S$ satisfy the conditions of Theorems~\ref{thm:grau:hofmann:1} and~\ref{thm:grau:hofmann}, and thus satisfy appropriate estimates; a more detailed outline of Sections~\ref{sec:decay}--\ref{sec:high} is provided in Section~\ref{sec:Tb}.

\subsection*{Acknowledgements}
We would like to thank the American Institute of Mathematics for hosting the SQuaRE workshop on ``Singular integral operators and solvability of boundary problems for elliptic equations with rough coefficients,'' at which many of the results and techniques of this paper were discussed.

\section{Definitions}
\label{sec:dfn}
%\numberwithin{equation}{subsection}

Throughout we work with a divergence-form elliptic operator $L$ of order~$2m$ acting on functions defined in $\R^\dmn$.

We will reserve the letters $\alpha$, $\beta$, $\gamma$, $\zeta$ and~$\xi$ to denote multiindices in $\N^\dmn$. If $\zeta=(\zeta_1,\zeta_2,\dots,\zeta_\dmn)$ is a multiindex, then we define $\abs{\zeta}$, $\partial^\zeta$ and $\zeta!$ in the usual ways, as $\abs{\zeta}=\zeta_1+\zeta_2+\dots+\zeta_\dmn$, $\partial^\zeta=\partial_{x_1}^{\zeta_1}\partial_{x_2}^{\zeta_2} \cdots\partial_{x_\dmn}^{\zeta_\dmn}$,
and $\zeta!=\zeta_1!\,\zeta_2!\cdots\zeta_\dmn!$. % if $\vec X\in \C^\dmn$ then $\vec X^\zeta=X_1^{\zeta_1}\dots X_\dmn^{\zeta_\dmn}$. Observe that if $\varphi$ is a smooth function then
%\begin{equation*}\sum_{\abs{\zeta}=k} \frac{k!}{\zeta!} \partial^\zeta \varphi = \sum_{j_1=1}^\dmn\sum_{j_2=1}^\dmn\cdots\sum_{j_k=1}^\dmn \partial_{j_1}\partial_{j_2}\dots\partial_{j_k}\varphi.\end{equation*}
If $\zeta$ and $\xi$ are two multiindices, then we say that $\xi\leq \zeta$ if $\xi_i\leq \zeta_i$ for all $1\leq i\leq\dmn$, and we say that $\xi<\zeta$ if in addition the strict inequality $\xi_i< \zeta_i$ holds for at least one such~$i$.

We will routinely deal with arrays $\arr F=\begin{pmatrix}F_{\zeta}\end{pmatrix}$ of numbers or functions indexed by multiindices~$\zeta$ with $\abs{\zeta}=k$ for some~$k$.
In particular, if $\varphi$ is a function with weak derivatives of order up to~$k$, then we view $\nabla^k\varphi$ as such an array.

The inner product of two such arrays of numbers $\arr F$ and $\arr G$ is given by
\begin{equation*}\bigl\langle \arr F,\arr G\bigr\rangle =
\sum_{\abs{\zeta}=k}
\overline{F_{\zeta}}\, G_{\zeta}.\end{equation*}
If $\arr F$ and $\arr G$ are two arrays of $L^2$ functions defined in an open set $\Omega$ or on its boundary, then the inner product of $\arr F$ and $\arr G$ is given by
\begin{equation*}\bigl\langle \arr F,\arr G\bigr\rangle_\Omega =
\sum_{\abs{\zeta}=k}
\int_{\Omega} \overline{F_{\zeta}}\, G_{\zeta}
\qquad\text{or}\qquad
\bigl\langle \arr F,\arr G\bigr\rangle_{\partial\Omega} =
\sum_{\abs{\zeta}=k}
\int_{\partial\Omega} \overline{F_{\zeta}}\, G_{\zeta}\,d\sigma\end{equation*}
where $\sigma$ denotes surface measure.

If $\arr G$ is an array of functions defined in~$\Omega$ and indexed by multiindices~$\alpha$ with $\abs{\alpha}=k$, then $\Div_k\arr G$ is the distribution given by
\begin{equation}
\label{eqn:Div}
\bigl\langle \varphi, \Div_k\arr G\bigr\rangle_\Omega
=
(-1)^m\bigl\langle\nabla^k\varphi, \arr G\bigr\rangle_\Omega
\end{equation}
for all smooth test functions~$\varphi$ supported in~$\Omega$.
In particular, if the right-hand side is zero for all such $\varphi$ then we say that $\Div_k\arr G=0$.

We let $\vec e_k$ be the unit vector in $\R^\dmn$ in the $k$th direction; notice that $\vec e_k$ is a multiindex with $\abs{\vec e_k}=1$. We let $\arr e_{\zeta}$ be the ``unit array'' corresponding to the multiindex~$\zeta$; thus, $\langle \arr e_{\zeta},\arr F\rangle = F_{\zeta}$.
We will often distinguish the $\dmn$th direction; we let $\gamma_\perp = (m-1)\vec e_\dmn = (0,\dots,0,m-1)$ and $\alpha_\perp = m\vec e_\dmn$, and let the array $\arr e_\perp$ denote either $\arr e_{\gamma_\perp}$ or $\arr e_{\alpha_\perp}$. Which is meant should be clear from context.

We let $L^p(U)$ and $L^\infty(U)$ denote the standard Lebesgue spaces with respect to either Lebesgue measure (if $U$ is a domain) or surface measure (if $U$ is a subset of the boundary of a domain).
We denote the homogeneous Sobolev space $\dot W^p_k(U)$ by
\begin{equation*}\dot W^p_k(U)=\{u:\nabla^k u\in L^p(U)\}\end{equation*}
with the norm $\doublebar{u}_{\dot W^p_k(U)}=\doublebar{\nabla^k u}_{L^p(U)}$. (Elements of $\dot W^p_k(U)$ are then defined only up to adding polynomials of degree $k-1$.) We say that $u\in L^p_{loc}(U)$ or $u\in W^p_{k,loc}(U)$ if $u\in L^p(V)$ or $u\in \dot W^p_k(V)$ for every bounded set $V$ with $\overline V\subset U$. In particular, if $U$ is a set and $\overline U$ is its closure, then functions in $L^p_{loc}(\overline U)$ are required to be locally integrable even near the boundary~$\partial U$; if $U$ is open this is not true of $L^p_{loc}(U)$.

%If $B$ is a Banach space we will let $B^*$ denote its dual space. %; if $1\leq p\leq \infty$ then we will let $p'$ be the extended real number that satisfies $1/p+1/p'=1$. Thus, if $1\leq p<\infty$, then $(L^p(U))^*=L^{p'}(U)$.

If $\mu$ is a measure and $E$ is a $\mu$-measurable set, with $\mu(E)<\infty$, we let $\fint_E f\,d\mu=\frac{1}{\mu(E)}\int_E f\,d\mu$.
If $E\subset\R^\dmn$ is a set, we let $\1_E$ denote the characteristic function of~$E$; in particular, we will let $\1_\pm$ denote the characteristic function of the half-space $\R^\dmn_\pm$. If $f$ is a function defined on~$E$, we will often let $\1_E f$ denote the extension of $f$ to $\R^\dmn$ by zero.

If $Q\subset\R^n$ is a cube, we let $\ell(Q)$ be its side-length. We let $rQ$ be the concentric cube of side-length $r\ell(Q)$. We will make frequent use of ``dyadic annuli'' defined as follows.
We let
\begin{equation}\label{eqn:annuli}
A_0(Q) = 2Q, \qquad
A_j(Q) = 2^{j+1}Q\setminus 2^{j}Q\quad\text{for all $j\geq 1$}
.\end{equation}
If $i\geq 0$, let
\begin{equation}
\label{eqn:annuli:wide}
A_{j,i}(Q) = \bigcup_{\ell=j-i}^{j+i}A_\ell(Q)\end{equation}
where $A_\ell(Q)=\emptyset$ whenever $\ell<0$.

% We will let
%\begin{equation*}A_0(Q)=2Q,\qquad A_j(Q)=2^{j+1}Q\setminus 2^jQ\text{ for all $j\geq 1$.}\end{equation*}

Throughout the paper we will work mainly in the domain $\R^\dmn_+=\{(x,t):x\in\R^n,t>0\}$. We will also need to consider $\R^\dmn_-=\{(x,t):x\in\R^n,t<0\}$. We will often identify $\R^n$ with~$\partial\R^\dmn_\pm$.

If $\varphi$ is a function defined on an open subset of $\R^\dmn$, we will let $\nabla_\pureH\varphi$ denote the gradient only in the first $n$ variables; we will also use $\nabla_\pureH f$ to denote the gradient of a function~$f$ defined on $\R^n=\partial\R^\dmn_\pm$.
We will view $\nabla_\pureH^k\varphi$ as an array of functions indexed by multiindices $\zeta\in \N^\dmn$ with $\abs{\zeta}=k$ and $\zeta_\dmn=0$; equivalently we may view $\nabla_\pureH^k\varphi$ as an array of functions indexed by multiindices $\zeta\in \N^n$.

Similarly, if $\arr F$ is an array of functions indexed by multiindices $\zeta\in \N^n$ with $\abs{\zeta}=k$, we will define $\Div_{k,\pureH}\arr F$ formally as $\Div_{k,\pureH}\arr F=\sum_{\abs{\zeta}=k,\>\zeta\in\N^n} \partial^\zeta F_\zeta$; the weak definition is precisely analogous to the definition~\eqref{eqn:Div} of~$\Div_k \arr F$.

%Finally, we define the nontangential maximal functions
%\begin{align}
%\label{dfn:NTM}
% N_{\pm} F(x) &= \sup\bigl\{\abs{F(y,s)}:\abs{x-y}<s \bigr\},
%\\
%\label{dfn:NTM:modified}
%\widetilde N_{\pm} F(x) &= \sup
%\biggl\{\biggl(\fint_{B((y,s),\abs{s}/2)} \abs{F}^2\biggr)^{1/2}:
%\abs{x-y}<s
%\biggr\}.
%\end{align}
%

\subsection{Elliptic operators}
\label{sec:elliptic}

Let $\mat A = \begin{pmatrix} A_{\alpha\beta} \end{pmatrix}$ be measurable coefficients defined on $\R^\dmn$, indexed by multtiindices $\alpha$, $\beta$ with $\abs{\alpha}=\abs{\beta}=m$. If $\arr F$ is an array, then $\mat A\arr F$ is the array given by
\begin{equation*}(\mat A\arr F)_{\alpha} =
\sum_{\abs{\beta}=m}
A_{\alpha\beta} F_{\beta}.\end{equation*}

Throughout we consider coefficients that satisfy the G\r{a}rding inequality
\begin{align}
\label{eqn:elliptic}
\re {\bigl\langle\nabla^m \varphi,\mat A\nabla^m \varphi\bigr\rangle_{\R^\dmn}}
&\geq
	\lambda\doublebar{\nabla^m\varphi}_{L^2(\R^\dmn)}^2
	\quad\text{for all $\varphi\in\dot W^2_m(\R^\dmn)$}
\end{align}
and the bound
%\begin{align}
%\label{eqn:elliptic:bounded:1}
%\abs[big]{\bigl\langle\nabla^m\vec g,\mat A\nabla^m\vec f\bigr\rangle_{\R^\dmn}}
%&\leq
%	\Lambda \doublebar{\nabla^m\vec f}_{L^2(\R^\dmn)} \doublebar{\nabla^m\vec g}_{L^2(\R^\dmn)}
%\end{align}
%for some $\Lambda>\lambda>0$, and for all $\vec f$, $\vec g$, $\varphi\in \dot W^2_m(\R^\dmn)$. %Occasionally we will need to replace the bound~\eqref{eqn:elliptic:bounded:1} with the stronger bound
\begin{align}
\label{eqn:elliptic:bounded}
\doublebar{\mat A}_{L^\infty(\R^\dmn)}
&\leq
	\Lambda
\end{align}
for some $\Lambda>\lambda>0$.

We let $L$ be the $2m$th-order divergence-form operator associated with~$\mat A$. That is, we say that $L u=\Div_m \arr F$ in~$\Omega$ in the weak sense if, for every $\varphi$ smooth and compactly supported in~$\Omega$, we have that
\begin{equation}
\label{eqn:L}
\bigl\langle\nabla^m\varphi, \mat A\nabla^m u\bigr\rangle_\Omega
=
\bigl\langle\nabla^m\varphi, \arr F\bigr\rangle_\Omega,
\end{equation}
that is, we have that
\begin{equation}\label{eqn:L:integral}
\sum_{\abs{\alpha}=\abs{\beta}=m}
\int_{\Omega}\partial^\alpha \bar \varphi\, A_{\alpha\beta}\,\partial^\beta u
=
\sum_{\abs{\alpha}=m}
\int_{\Omega} \partial^\alpha \bar \varphi \, F_{\alpha}
.
\end{equation}
In particular, if the left-hand side is zero for all such~$\varphi$ then we say that $L u=0$.

We let  $\mat A^*$ be the adjoint matrix, that is, $A^*_{\alpha\beta}=\overline{A_{\beta\alpha}}$. We let $L^*$ be the associated elliptic operator.

In this paper we will focus exclusively on operators $L$ that are $t$-independent, that is, whose coefficients satisfy formula~\eqref{eqn:t-independent}.

Throughout the paper we will let $C$ denote a constant whose value may change from line to line, but which depends only on the dimension $\dmn$, the ellipticity constants $\lambda$ and $\Lambda$ in the bounds \eqref{eqn:elliptic} and~\eqref{eqn:elliptic:bounded}, and the order~$2m$ of our elliptic operators. Any other dependencies will be indicated explicitly. We say that $ A\approx B$ if, for some such constant $C$, $ A\leq CB$ and $B\leq CA$.

\subsection{Dirichlet and Neumann boundary data}
\label{sec:data}

Our goal in the present paper is to bound the double and single layer potentials; in future work we hope to use the results of this paper to solve the Dirichlet and Neumann boundary value problems. Thus, in this section, we will define higher-order Dirichlet and Neumann boundary data.

We define higher-order Dirichlet boundary data as follows.
Suppose that $\nabla^m F\in L^1_{loc}(\overline{\R^\dmn_+})$ or $\nabla^m F\in L^1_{loc}(\overline{\R^\dmn_-})$. Then $\partial^\gamma u\in \dot W^1_{1,loc}(\overline{\R^\dmn_\pm})$ for any $\gamma$ with $\abs{\gamma}=m-1$. We define $\Tr_{m-1} u$ as the array given by
\begin{equation}
\label{eqn:Dirichlet}
\begin{pmatrix}\Tr_{m-1}  u\end{pmatrix}_{\gamma}
=\Trace \partial^\gamma u
\quad\text{for all $\abs{\gamma}=m-1$}\end{equation}
where $\Trace$ denotes the standard trace operator on Sobolev spaces. In some cases we will use $\Tr_{m-1}^+ u$ and $\Tr_{m-1}^- u$ to specify that the trace is taken in $\R^\dmn_+$ or~$\R^\dmn_-$. Notice that if $\nabla^m u$ is locally integrable in all of $\R^\dmn$, then $\Tr_{m-1}^+ u=\Tr_{m+1}^- u$.

With some care we may define boundary values of certain higher-order derivatives. If $u\in \dot W^1_{m,loc}(\R^\dmn)$ is such that $\Tr_{m-1}u\in \dot W^1_{1,loc}(\R^n)$, then for each $\beta$ with $\beta_\dmn<\abs{\beta}=m$, we define
\begin{equation}
\label{eqn:Dirichlet:tangential}
(\Tr_{m,\semiH} u)_\beta=\partial_{x_j} \Trace \partial^{\beta-\vec e_j} u\quad \text{for all $j$ with $1\leq j\leq n$ and $\beta_j>0$}.\end{equation}
Note the requirement that $j\neq \dmn$.
This is well-defined; that is, if $\beta_j>0$ and $\beta_k>0$ for some $j\neq\dmn$ and $k\neq\dmn$, then it does not matter whether we choose $x_j$ or $x_k$ as our distinguished representative.

We define higher-order Neumann boundary data as follows. Let $\dot W^2_{m,0}({\R^\dmn_\pm})$ be the closure in~$\dot W^2_m({\R^\dmn_\pm})$ of the set of all smooth functions supported in~${\R^\dmn_\pm}$. Then \[\dot W^2_{m,0}({\R^\dmn_\pm})=\{ v\in \dot W^2_m({\R^\dmn_\pm}):\Tr_{m-1}^\pm v=0\}.\]
Observe that if $ v \in\dot W^2_{m,0}({\R^\dmn_\pm})$, $ u\in \dot W^2_m({\R^\dmn_\pm})$ and $L u=0$ in ${\R^\dmn_\pm}$, then 
\[\langle \nabla^m v, \mat A\nabla^m u\rangle_{\R^\dmn_\pm}=0.\]
Thus, if $ \varphi\in \dot W^2_m({\R^\dmn_\pm})$, then the inner product $\langle \nabla^m \varphi, A\nabla^m u\rangle_{\R^\dmn_\pm}$ depends only on $\Tr_{m-1}^\pm \varphi$; if $\Tr_{m-1}^\pm \varphi=\Tr_{m-1}^\pm \eta$ then $\langle \nabla^m \varphi-\nabla^m \eta, A\nabla^m u\rangle_{\R^\dmn_\pm}=0$. We define the Neumann boundary values $\M_{\mat A}^\pm  u$ by
\begin{equation}
\label{eqn:Neumann}
\bigl\langle \Tr_{m-1}^\pm\varphi,\M_{\mat A}^\pm  u\bigr\rangle_{\partial{\R^\dmn_\pm}}
=
\bigl\langle \nabla^m\varphi, \mat A\nabla^m u\bigr\rangle_{\R^\dmn_\pm}
\qquad\text{for all }\varphi\in \dot W^2_m({\R^\dmn_\pm}).
\end{equation}
$\M_{\mat A}^\pm u$ is then a linear operator on the space of traces of $\dot W^2_m({\R^\dmn_\pm})$-functions.

\subsubsection{Historical remarks and context}\label{sec:data:history}

We now provide some further discussion and history of Dirichlet and Neumann boundary data.

We remark that we have three ways to refer to derivatives of order~$m$ at $\R^n=\partial\R^\dmn_\pm$, namely $\nabla^m_\pureH u(x,0)$, $\Tr_{m,\semiH} u(x)$, and $\nabla^m u(x,0)$. All three are arrays indexed by multiindices $\beta$ with $\abs{\beta}=m$. To give the reader some intuition, let us discuss the simplest case, that of smooth functions in $\RR^2_+=\{(x,t): x\in \RR,\> t>0\}$ with $m=2$. In this case,  formally, $\Tr_1 u$ is the array containing $\partial_x u(x,0)$ and $\partial_t u(x,0)$.
As for the higher order traces, $\nabla^m_\pureH u(x,0)=\partial^2_{xx}u(x,0)$, $\Tr_{m,\semiH} u(x)$ is the array $(\partial^2_{xx} u,\partial_{xt}^2 u)$ containing $\partial^2_{xx}u=\partial_x \Trace \partial_x u$ and $\partial_{xt}^2 u=\partial_x\Trace \partial_t u$ on $\partial\RR^2_+$, while $\nabla^m u(x,0)$ is the array $(\partial_{xx}^2 u,\partial_{xt}^2 u,\partial_{tt}^2u)$ of all second derivatives. %(In fact, for the bilaplacian in $\RR^2_+$ the solutions are smooth up to the boundary, and so all above derivatives have continuous restriction to the boundary).

The reader should compare our choice of representation of the Dirichlet data $\Tr_1 u=(\partial_x u(x,0), \partial_t u(x,0))$ to the traditional choice $(u(x,0), \partial_t u(x,0))$. This is, of course, a question of representation and what matters is the function spaces for the data. Working with $\Tr_1$ in place of $(u(x,0), \partial_t u(x,0))$ brings considerable advantage and clarity mainly because both parts of the trace array belong to a function space of the same level of smoothness. For example, $\Tr_1$ is a vector with both components in $L^2(\RR)$ when  $(u(x,0), \partial_t u(x,0))$ lies in $\dot W^{1,2}(\RR)\times L^2(\RR)$. This makes things much clearer when dealing with divergence form operators of arbitrary higher order. This also allows us to properly define Neumann data.

Recall that the latter is quite tricky even for the simple case of the bilaplacian and that some choices can make the Neumann problem ill-posed. Let us discuss this in some detail, starting with the bilaplacian on a Lipschitz domain $\Omega\subset \RR^n$. We will translate to the half-space below.

The Neumann boundary values of a solution are traditionally given by an integration by parts (formula~\eqref{eqn:Neumann:traditional}) or less explicitly as an inner product (formula~\eqref{eqn:Neumann}). In the case of the biharmonic equation, Neumann boundary values also have applications in the theory of elasticity.
Recall that the principal physical motivation for the inhomogeneous biharmonic equation $\Delta^2 u=h$ is that it describes the equilibrium position of a thin elastic plate subject to a vertical force~$h$. The Dirichlet problem $u\big\vert_{\partial\Omega}=f$, $\nabla u\big\vert_{\partial\Omega}=\vec g$ describes an elastic plate whose edges are clamped, that is, held at a fixed position in a fixed orientation. The Neumann problem, on the other hand, corresponds to the case of a free boundary.
Guido Sweers has written an excellent short paper \cite{Swe09} discussing the boundary conditions that correspond to these and other physical situations.

More precisely, if a thin two-dimensional plate is subject to a force $h$ and the edges are free to move, then its displacement $u$ satisfies the boundary value problem
\begin{equation*}
\left\{
\begin{aligned}
\Delta^2 u &= h && \text{in } \Omega,\\
\rho \Delta u+(1-\rho)\partial_\nu^2 u&=0 &&\text{on }\partial\Omega,\\
\partial_\nu \Delta u+(1-\rho) \partial_{\tau\tau\nu}u&=0 &&\text{on }\partial\Omega.\\
\end{aligned}\right.
\end{equation*}
Here $\rho$ is a physical constant, called the Poisson ratio, and $\nu$ and $\tau$ are the unit outward normal and unit tangent vectors to the boundary. This formulation goes back to Kirchoff and is well known in the theory of elasticity; see, for example, Section~3.1 and Chapter~8 of the classic engineering text \cite{Nad63}.
We remark that by \cite[formula \hbox{(8-10)}]{Nad63},
\begin{equation*}
\partial_\nu \Delta u + (1-\rho) \partial_{\tau\tau\nu}u=
\partial_\nu \Delta u + (1-\rho) \partial_{\tau}
\left(\partial_{\nu\tau} u\right).\end{equation*}

This suggests the following homogeneous boundary value problem in a Lipschitz domain $\Omega\subset\R^\dmn$ of arbitrary dimension. We say that the $L^p$-Neumann problem is well-posed if there exists a constant $C>0$ such that, for every $f_0\in L^p(\partial\Omega)$ and $\Lambda_0\in W^p_{-1}(\partial\Omega)$, there exists a function $u$ such that
\begin{equation}
\label{eqn:neumann}
\left\{
\begin{aligned}
\Delta^2 u ={}& 0 && \text{in } \Omega,\\
M_\rho u:={}&\rho \Delta u + (1-\rho) \partial_\nu^2 u = f_0&& \text{on }\partial\Omega,\\
K_{\rho} u:={}&\partial_\nu \Delta u + (1-\rho) \frac{1}{2}
\partial_{\tau_{jk}}
\left(\partial_{\nu\tau_{jk}} u \right) = \Lambda_0&& \text{on }\partial\Omega,\\
\doublebar{N(\nabla^2 u)}_{L^p(\partial\Omega)}
\leq{}& C
\doublebar{f_0}_{W^p_1(\partial\Omega)}
+C\doublebar{\Lambda_0}_{W^p_{-1}(\partial\Omega)}
&&{\text{on }\partial\Omega.}
\end{aligned}\right.
\end{equation}
Here $\tau_{ij} = \nu_i {\e}_j-\nu_j {\e}_i$ is a vector orthogonal to the outward normal $\nu$ and lying in the $x_ix_j$-plane, and $N(\nabla^2 u)$ denotes the nontangential maximal function common in the theory of elliptic boundary value problems. We apply the convention that the repeated indices $j$ and $k$ are summed from $1$ to $\dmn$.

The boundary operators $M_\rho$ and $K_\rho$, derived from the theory of elasticity, are the same as the Neumann boundary operators discussed in Section~\ref{sec:introduction:higher}. Specifically, for any $\rho\in\R$, the equation
\begin{align}
\label{eqn:biharmonicDN}
\int_\Omega w\,\Delta^2 v
&=
	%\sum_{j,k=1}^\dmn
	\int_\Omega \left(\rho\,\Delta w\, \Delta v
	+(1-\rho)\,\partial_{jk}w\,\partial_{jk} v\right)
	%\\&\qquad\nonumber
	+\int_{\partial\Omega} w\, K_\rho v-\partial_\nu w\,M_\rho v\,d\sigma
\end{align}
is valid for arbitrary smooth functions. Comparing to formula~\eqref{eqn:Neumann:traditional}, we see that $B_0^{\mat A} = K_\rho$ and $B_1^{\mat A} = -M_\rho$, where $\mat A=\mat A_\rho$ is an appropriate choice of coefficients associated with $L=\Delta^2$.

Observe that, contrary to the Laplacian or more general second order operators, there is a \emph{family} of relevant Neumann data for the biharmonic equation. Moreover, different values (or, rather,   ranges) of $\rho$ correspond to different natural physical situations. We refer the reader to  \cite{Ver05} for a detailed discussion.

Recall that our formulation of Neumann data is somewhat different; we use the array $\M_{\mat A}^\pm u$ of formula~\eqref{eqn:Neumann} rather than the functions $B_j^{\mat A}u$ of formula~\eqref{eqn:Neumann:traditional}. As an example, in the case $L=\Delta^2$ and $\Omega=\R^\dmn_+$, we will provide an explicit formula for one representative of $\M_{\mat A}^+ u$.
In the half-space,
\begin{equation*}K_\rho u = -\partial_\dmn \Delta u - (1-\rho) \partial_j (\partial_{j\pdmn}^2 u).\end{equation*}
We are now summing from $j=1$ to~$n$.
If $\Delta u$ is harmonic in $\R^\dmn_+$, then $\partial_\dmn \Delta u = \partial_j R_j (\Delta u)$, where $R_j$ denotes the $j$th Riesz transform. Thus, if $\vec\varphi:\R^n\mapsto\C^n$ is any divergence-free vector field, then
\begin{align}
\label{eqn:biharmonicDN:weak}
\int_{\R^\dmn_+} w\,\Delta^2 u
&=
	%\sum_{j,k=1}^\dmn
	\int_{\R^\dmn_+} \left(\rho\,\Delta w\, \Delta u
	+(1-\rho)\,\partial_{jk}w\,\partial_{jk} u\right)
	\\&\qquad\nonumber
	+\int_{\partial{\R^\dmn_+}} \partial_\dmn w\,M_\rho u
	+  \partial_j w \bigl(  R_j (\Delta u) + (1-\rho) \partial_{j\pdmn}^2 u+\varphi_j\bigr)
	\,d\sigma
\end{align}
from which we may recover an explicit formula for $\M_{\mat A}^+ u$.
%\begin{equation*}\M_{\mat A}^+ v = M_\rho v \,\vec e_\dmn + (  R_j (\Delta v) + (1-\rho) (\partial_{j\pdmn}^2 v))\vec e_j + (\vec\varphi,0).\end{equation*}

We comment on several aspects of this formula. First, observe that we still have a family of Neumann boundary data indexed by the parameter~$\rho$. Next, observe that we did use the fact that $\Delta^2 u=0$; this formula, unlike formula~\eqref{eqn:biharmonicDN}, is not valid for arbitrary smooth functions. Furthermore, observe the presence of the vector field $\vec\varphi$ in $\M_{\mat A}^+u$; our explicit representation gives a natural normalization $\vec\varphi=0$, but for more general operators the divergence-free vector field cannot be neglected. Finally, observe that our formula for $\M_{\mat A}^+ u$ is not a local one: it involves the Riesz transforms of derivatives of~$u$ rather than simply linear combinations.

However, notice one significant advantage of our formulation~\eqref{eqn:biharmonicDN:weak} over the operators~$K_\rho$ and $M_\rho$. The term $M_\rho u$ involves second derivatives of~$u$, while the term $K_\rho u$ involves third derivatives; we have expressed all components of $\M_{\mat A}^+ u$ using the second derivatives of~$u$. As discussed in Section~\ref{sec:introduction:higher}, this means that we expect the different components of the Neumann boundary data to lie in a single smoothness space; furthermore, using boundedness of the Riesz transform, we may control $\doublebar{\M_{\mat A}^+ u}_{L^p(\R^n)}$ by $\doublebar{\nabla^2 u}_{L^p(\R^n)}$, for $1<p<\infty$.

Our formulation of Neumann boundary data for general operators will display most of these issues. The existence of a family of Neumann data may be eliminated by specifying the matrix of coefficients $\mat A$ in formula~\eqref{eqn:Neumann}, but our formulation of $\M_{\mat A}^+ u$ does require that $u$ be a solution, is well-defined only up to adding divergence-free terms, and need not have a local representation. Indeed, at this point, the estimate $\doublebar{\M_{\mat A}^+ u}_{L^p(\R^n)}\leq C\doublebar{\nabla^m u}_{L^p(\R^n)}$, while apparently plausible, is still only a conjecture.

We now discuss the history of the $L^p$-Neumann problem~\eqref{eqn:neumann}.
In \cite{CohG85}, Cohen and Gosselin showed that this problem was well-posed in $C^1$ domains contained in $\RR^2$ for for $1<p<\infty$, provided in addition that $\rho=-1$. In \cite{Ver05}, Verchota investigated the Neumann problem~\eqref{eqn:neumann} in full generality. He considered Lipschitz domains with compact, connected boundary contained in $\RR^{n+1}$, $n+1\geq 2$.
He showed that if $-1/n\leq \rho<1$, then the Neumann problem is well-posed provided $2-\varepsilon<p<2+\varepsilon$. That is, the solutions exist, satisfy the desired estimates, and are unique either modulo functions of an appropriate class, or (in the case where $\Omega$ is unbounded) when subject to an appropriate growth condition. See \cite[Theorems 13.2 and~15.4]{Ver05}. The Neumann problem is ill-posed for $\rho\geq 1$ and $\rho<1/n$; see \cite[Section~21]{Ver05}. More recently, in \cite{She07B}, Shen improved upon Verchota's results by extending the range on $p$ (in bounded simply connected Lipschitz domains) to $2n/(n+2)-\varepsilon<p<2+\varepsilon$ if $n+1\geq 4$, and $1<p<2+\varepsilon$ if $n+1=2$ or $n+1=3$. All of the aforementioned results rely on the method of layer potentials.
Finally, in \cite[Section~6.5]{MitM13A}, I. Mitrea and M. Mitrea showed that if $\Omega\subset\R^\dmn$ is a simply connected domain whose unit outward normal $\nu$ lies in~$VMO(\partial\Omega)$ (for example, if $\Omega$ is a $C^1$ domain), then the acceptable range of~$p$ is $1<p<\infty$; this may be seen as a generalization of the result of Cohen and Gosselin to higher dimensions, to other values of~$\rho$, and to slightly rougher domains.
The question of the sharp range of $p$ for which the $L^p$-Neumann problem is well-posed in a Lipschitz domain is still open.

It turns out that extending well-posedness results for the Neumann problem beyond the case of the bilaplacian is an excruciatingly difficult problem. Even defining Neumann boundary values for more general operators is a difficult problem; compare the traditional definition of Neumann boundary data of~\eqref{eqn:Neumann:traditional} and our formulation~\eqref{eqn:Neumann}. While some progress has been made (see \cite{Agr07,MitM13A,Bar15p}), at present there are no well-posedness results for the Neumann problem with $L^p$ boundary data.

Recall that \cite{Agr07,MitM13A} have investigated the Neumann problem for boundary data formulated as in formula~\eqref{eqn:Neumann:traditional}. Specifically,
Agranovich has established some well-posedness results for the inhomogeneous problem $Lu=h$ with homogeneous Neumann boundary data, and has provided some brief discussion of the conditions needed to resolve the Neumann problem with inhomogeneous boundary data; see  \cite[Section~5.2]{Agr07}. The book \cite{MitM13A} considers the case of constant-coefficient operators at length; therein they establish well-posedness results for the Neumann problem, with boundary data in certain fractional smoothness spaces, for elliptic constant-coefficient operators.

\subsection{The Newton potential and the fundamental solution}

The main purpose of the present paper is to define and bound the double and single layer potentials for higher-order elliptic operators of the form specified in Section~\ref{sec:elliptic}. Recall from formulas \eqref{eqn:introduction:D} and~\eqref{eqn:introduction:S} that the second-order layer potentials are built from the second-order fundamental solution.

The main result of the paper \cite{Bar14p} was a construction of the fundamental solution $E^L$ in the case of higher-order operators. $E^L$ was constructed as (an order-$m$ antiderivative of) the kernel to the operator $\Pi^L$, the Newton potential for~$L$, defined as follows.
For any $\arr H\in L^2(\R^\dmn)$, by the Lax-Milgram lemma there is a unique function $\Pi^L\arr H$ in $\dot W^2_m(\R^\dmn)$ that satisfies
\begin{equation}\label{eqn:newton}
\langle \nabla^m\varphi, \mat A\nabla^m \Pi^L\arr H\rangle_{\R^\dmn}=\langle \nabla^m\varphi, \arr H\rangle_{\R^\dmn}\end{equation}
for all $\varphi\in \dot W^2_m(\R^\dmn)$.
The Newton potential is a bounded operator on~$L^2(\R^\dmn)$ and satisfies the bound
\begin{equation}\label{eqn:newton:bound}
\doublebar{\nabla^m \Pi^L\arr H}_{L^2(\R^\dmn)}\leq C
\doublebar{\arr H}_{L^2(\R^\dmn)}.
\end{equation}
We will need two additional properties of the Newton potential. First, we will need the symmetry relation
\begin{equation}\label{eqn:newton:adjoint}
\langle \arr G, \nabla^m \Pi^L\arr H\rangle_{\R^\dmn} = \langle \nabla^m\Pi^{L^*}\arr G, \arr H\rangle_{\R^\dmn}
\end{equation}
for all $\arr H\in L^2(\R^n)$ and all $\arr G\in L^2(\R^n)$.
Second, we will need the identity
\begin{equation}\label{eqn:newton:identity}
\nabla^m \Pi^L (\mat A\nabla^m F) = \nabla^m F
\end{equation}
for all $F\in \dot W^2_m(\R^\dmn)$; this identity follows by uniqueness of the Newton potential as the solution operator in formula~\eqref{eqn:newton}.

We remark that this Newton potential $\Pi^L$ is somewhat different in smoothness from the traditional Newton potential. This potential, which we will denote $N^L$, is the unique solution of the equation $Lu=f$, $u\in \dot W^2_m(\R^\dmn)$, or, more precisely, of
\begin{equation}\label{eqn:newton-bis}
\langle \nabla^m\varphi, \mat A\nabla^m N^L f\rangle_{\R^\dmn}=\langle \varphi, f\rangle_{\R^\dmn}\end{equation}
for all $\varphi\in \dot W^2_m(\R^\dmn)$. The input $f$ should thus be taken in $\dot W^2_{-m}(\R^\dmn)$, the dual space to $\dot W^2_m(\R^\dmn)$. One can write $f=\Div_m \arr H$ for some $\arr H\in L^2(\R^\dmn)$. The above formulas then become $Lu=\Div_m \arr H$, $u\in \dot W^2_m(\R^\dmn)$, or, more precisely,
\begin{equation}\label{eqn:newton-bis2}
\langle \nabla^m\varphi, \mat A\nabla^m N^L\Div_m \arr H\rangle_{\R^\dmn}=\langle \nabla^m \varphi, \arr H\rangle_{\R^\dmn}.\end{equation}
In other words, $\Pi^L=N^L\Div_m$. We shall be working exclusively with $\Pi^L$, but perhaps this analogy is useful to keep in mind.

The main result of \cite{Bar14p} may be stated as follows.

\begin{thm}[{\cite[Theorem~62 and Lemma~69]{Bar14p}}]\label{thm:fundamental}
Let $L$ be an operator of order~$2m$ that satisfies the bounds \eqref{eqn:elliptic} and~\eqref{eqn:elliptic:bounded}.
Then there exists a function $E^L(X,Y)$ with the following properties.

Let $q$ and $s$ be two integers that satisfy $q+s<\dmn$ and the bounds $0\leq q\leq \min(m,\pdmn/2)$, $0\leq s\leq \min(m,\pdmn/2)$.

Then we have the symmetry property
\begin{equation}
\label{eqn:fundamental:symmetric}
\partial_X^\zeta\partial_Y^\xi E^L(X,Y) = \overline{\partial_X^\zeta\partial_Y^\xi E^{L^*}(Y,X)}
\end{equation}
as locally $L^2$ functions, for all multiindices $\zeta$, $\xi$ with $\abs{\zeta}=m-q$ and $\abs{\xi}=m-s$.

There is some $\varepsilon>0$ such that if $X_0$, $Y_0\in\R^\dmn$, if $0<4r<R<\abs{X_0-Y_0}/3$, and if $q<\pdmn/2$ then
\begin{equation}
\label{eqn:fundamental:far}
\int_{B(Y_0,r)}\int_{B(X_0,R)} \abs{\nabla^{m-s}_X \nabla^{m-q}_Y E^L(X,Y)}^2\,dX\,dY \leq C r^{2q} R^{2s} \biggl(\frac{r}{R}\biggr)^\varepsilon
.\end{equation}
If $q=\pdmn/2$ then we instead have the bound
\begin{equation}
\label{eqn:fundamental:far:lowest:2}
\int_{B(Y_0,r)}\int_{B(X_0,R)} \abs{\nabla^{m-s}_X \nabla^{m-q}_Y E^L(X,Y)}^2\,dX\,dY \leq C(\delta)\, r^{2q} R^{2s} \biggl(\frac{R}{r}\biggr)^\delta
\end{equation}
for all $\delta>0$ and some constant $C(\delta)$ depending on~$\delta$.

%If in addition $q+s>0$, then for all $p$ with $1\leq p\leq 2$ and $p<\pdmn/(\dmn-(q+s))$, we have that
%\begin{equation}
%\label{eqn:fundamental:near}
%\int_{B(X_0,r)}\int_{B(X_0,r)} \abs{\nabla^{m-s}_X \nabla^{m-q}_Y E^L(X,Y)}^p\,dX\,dY \leq C(p) r^{(2-p)\pdmn+p(s+q)}\end{equation}
%for all $X_0\in\R^\dmn$ and all $r>0$.

Furthermore,
%there is some $\varepsilon>0$ such that if $2-\varepsilon<p<2+\varepsilon$ then $\nabla^m\Pi^L$ extends to a bounded operator $L^p(\R^\dmn)\mapsto L^p(\R^\dmn)$. If $\zeta$ satisfies $m-\pdmn/p<\abs{\zeta}\leq m-1$ for some such~$p$, then
%\begin{equation}
%\label{eqn:fundamental:low}
%\partial_x^\zeta
%\Pi^L_j\arr H(X)
%	= \sum_{k=1}^N \sum_{\abs{\beta}=m} \int_{\R^\dmn} 	\partial_x^\zeta\partial_y^\beta E^L_{j,k}(X,Y)\,F_{k,\beta}(Y)\,dY
%\end{equation}
%for almost every $X\in\R^\dmn$, and for all $\arr H\in L^p(\R^\dmn)$ that are also locally in $L^{P}(\R^\dmn)$, for some $P>\pdmn/(m-\abs{\zeta})$. In the case of $\abs{\alpha}=m$, we still have that
if $\abs{\alpha}=m$ then
\begin{equation}
\label{eqn:fundamental:2}
\partial^\alpha\Pi^L\arr H(X)
	= \sum_{\abs{\beta}=m} \int_{\R^\dmn} 	\partial_X^\alpha\partial_Y^\beta E^L(X,Y)\,F_{\beta}(Y)\,dY
	\quad\text{}
\end{equation}
for almost every $X\notin\supp \arr H$, and for all $\arr H\in L^2(\R^\dmn)$ whose support is not all of $\R^\dmn$.

Finally, if $\widetilde E^L$ is any other  function that satisfies the bounds~\eqref{eqn:fundamental:far},~\eqref{eqn:fundamental:far:lowest:2} and formula~\eqref{eqn:fundamental:2}, then
\begin{align}
\label{eqn:fundamental:unique}
\widetilde E^L(X,Y)
=
	E^L(X,Y)&+\sum_{\abs{\zeta}<m-\pdmn/2} f_\zeta(X)\,Y^\zeta+\sum_{\abs{\xi}<m-\pdmn/2} g_\xi(Y)\,X^\xi
	\\&\nonumber
	+\sum_{\abs{\zeta}=\abs{\xi}=m-\pdmn/2}c_{\zeta,\xi} \,X^\zeta\,Y^\xi
\end{align}
for some functions~$f_\zeta$ and~$g_\zeta$ and some constants $c_{\zeta,\xi}$. Thus, $\nabla_X^{m-q}\nabla_Y^{m-s}E^L(X,Y)$ is a well-defined, locally $L^2$ function provided $q$ and $s$ satisfy the conditions specified above.
\end{thm}

Note that formula~\eqref{eqn:fundamental:2} and the definition of $\Pi^L$ assures that $E$ is indeed analogous to the traditional fundamental solution, which, roughly speaking, solves $LE=\delta$. That is, $E$ is formally the kernel of the potential $N$ defined above.

We record one further property of the fundamental solution for $t$-independent operators.
By the uniqueness property for the fundamental solution, if $\mat A$ is $t$-independent, then we have that
\begin{equation*}\partial_{x,t}^\zeta \partial_{y,s}^\xi E^L(x,t,y,s)= \partial_{x,t}^\zeta \partial_{y,s}^\xi E^L(x,t+r,y,s+r)\end{equation*}
for almost every $x\in\R^n$, $y\in \R^n$ and almost every $t$, $s$, $r\in \R$, and all multiindices~$\zeta$, $\xi$ as in formula~\eqref{eqn:fundamental:symmetric}. In particular, for such $\zeta$ and $\xi$ we have that
\begin{equation}
\label{eqn:fundamental:vertical}
\partial_{x,t}^\zeta \partial_{y,s}^\xi \partial_t E^L(x,t,y,s)= -\partial_{x,t}^\zeta \partial_{y,s}^\xi \partial_s E^L(x,t,y,s)
.\end{equation}

\begin{rmk}
We comment on the additional terms in formula~\eqref{eqn:fundamental:unique}. Notice that $E^L$ is defined essentially by the relation~\eqref{eqn:fundamental:2}. But this relation involves only derivatives of order~$2m$; in other words, it is only $\nabla_X^m\nabla_Y^m E^L(X,Y)$ that is well-defined. The lower-order derivatives are defined only up to adding polynomials. (The $\partial_X^\alpha$ derivative is included in formula~\eqref{eqn:fundamental:2} because $\vec\Pi^L\arr H\in \dot W^2_m(\R^\dmn)$, and so $\vec\Pi^L\arr H$ is also defined only up to adding polynomials.)

If $q$ and $s$ are small enough, then there is a unique normalization of the derivatives $\nabla_X^{m-q}\nabla_Y^{m-s}E^L(X,Y)$ that satisfies the bound \eqref{eqn:fundamental:far} or~\eqref{eqn:fundamental:far:lowest:2}; in \cite{Bar14p} this normalization was found using the Gagliardo-Nirenberg-Sobolev inequality. However, if $2m\geq\dmn$ then $E^L$ itself (and possibly some of its derivatives) are still not well-defined. The extra terms on the right-hand side of formula~\eqref{eqn:fundamental:unique} are precisely the terms compatible with the requirement $\nabla_X^{m-q}\nabla_Y^{m-s}\widetilde E^L(X,Y)=\nabla_X^{m-q}\nabla_Y^{m-s}E^L(X,Y)$ for $q$, $s$ small enough.
%(Notice that $f_\zeta(X)$ and $g_\xi(Y)$ may be polynomials of arbitrary order, and so we may add any polynomial in $X$ and~$Y$ of order at most $2m-\pdmn$. We must allow more general $f_\zeta$, $g_\xi$ because only some components of $\nabla^{2m-\pdmn+1}_{X,Y}E^L(X,Y)$ are well-defined, namely those where at least $m-\pdmn/2$ derivatives are taken in each of the two variables $X$ and~$Y$.)

Consequently, throughout this paper we will be careful to use only derivatives of $E^L$ of sufficiently high order; in fact, we will use only derivatives of the form $\partial_X^\zeta\partial_Y^\xi E^L(X,Y)$ for $\abs{\xi}\geq m-1$, $\abs{\zeta}\geq m-1$ and $\abs{\xi}+\abs{\zeta}\geq 2m-1$.

In some very special cases, there are natural normalization conditions for the fundamental solution even if $2m>\dmn$; for example, if $L=(-\Delta)^m$ and $\dmn\leq 2m$ is even, then we may take $E^L(X,Y)=C_{m,n} \abs{X-Y}^{2m-\pdmn}\log\abs{X-Y}$. Notice the presence of logarithmic growth in the fundamental solution. However, if we take $2m-\pdmn+1$ derivatives (in either $X$ or~$Y$), then the logarithm vanishes; this is the lowest order of derivative that Theorem~\ref{thm:fundamental} guarantees is well-defined.
\end{rmk}

\subsection{The double and single layer potentials}
\label{sec:potentials}

In this paper we seek to formulate a notion of layer potentials for higher-order elliptic operators of the form specified in Section~\ref{sec:elliptic}.
The goal of this paper is to produce bounds on layer potentials in the domain $\Omega=\R^\dmn_\pm$; thus, we will define boundary values and layer potentials only for the half-spaces.

We begin by recalling the second-order Green's formula~\eqref{eqn:green:introduction}. %: if $L$ is a second-order operator and $Lu=0$ in $\R^\dmn_+$, and if $u$ is reasonably well behaved, then
%\begin{equation*}\1_+ u=-\D^{\mat A}(\Trace u)+\s^{\mat A}(\nu\cdot A\nabla u)\end{equation*}
%where $\D$ is the double layer potential, $\s$ is the single layer potential, and where $\Trace u$ represents the trace of $u$, that is, the Dirichlet boundary values of~$u$, and $\nu\cdot A\nabla u$ represents the Neumann boundary values of~$u$.
%
To generalize this formula to higher order, notice that, for any function~$u\in \dot W^2_m(\R^\dmn_+)$,
\begin{equation*}\1_+ u = (\1_+ u - \Pi^L(\1_+ \mat A\nabla^m u)) + \Pi^L(\1_+ \mat A\nabla^m u)\end{equation*}
as $\dot W^2_m(\R^\dmn_\pm)$-functions.
We claim that the quantity
\begin{equation}
\label{dfn:D:newton}
\D^{\mat A}\arr f = -\1_+ F + \Pi^L(\1_+ \mat A\nabla^m F)
\qquad\text{if } \arr f=\Tr_{m-1}^+ F
\end{equation}
is well-defined; that is, the right-hand side depends only on $\Tr_{m-1}^+ F$. We will define $\s^{\mat A}\arr g$ in such a way that $\s^{\mat A}(\M_{\mat A}^+ u) = \Pi^L(\1_+ \mat A\nabla^m u)$ as $\dot W^2_m(\R^\dmn)$-functions for any $u\in\dot W^2_m(\R^\dmn_+)$ with $Lu=0$ in $\R^\dmn_+$; this then yields the higher-order Green's formula
\begin{equation}
\label{eqn:green}
\1_{\R^\dmn_+} \nabla^m u=-\nabla^m \D^{\mat A}(\Tr_{m-1}^+  u) + \nabla^m \s^{\mat A}(\M_{\mat A}^+  u)
.\end{equation}
We will also find formulas \eqref{eqn:D:fundamental} and~\eqref{eqn:S:fundamental} for the double and single layer potentials in terms of the fundamental solution; these formulas will also parallel formulas \eqref{eqn:introduction:D} and~\eqref{eqn:introduction:S}.

We now establish our claim that, if  $\arr f=\Tr_{m-1}^+ F$ for some $F\in W^2_m(\R^\dmn_+)$, then $\D^{\mat A}\arr f$ is a well-defined element of $\dot W^2_m(\R^\dmn_+)$ and $\dot W^2_m(\R^\dmn_-)$.
It suffices to show that, if $\Tr_{m-1}^+ F=0$, then the right-hand side of formula~\eqref{dfn:D:newton} is zero.

Suppose that $\Tr_{m-1}^+ F=\Trace\nabla^{m-1} F=0$. It is well known (see, for example, the proof of \cite[Theorem~5.5.2]{Eva98}) that in this case, $\nabla^{m-1}F$ lies in the completion in $\dot W^2_1(\R^\dmn_+)$ of the set of smooth functions compactly supported in $\R^\dmn_+$. Now, if $\varphi$ is compactly supported in $\R^\dmn_+$, then we may extend $\varphi$ to a function in all of $\R^\dmn$ by letting $\varphi=0$ in $\R^\dmn_-$. By density, we have that $\nabla^{m-1} F$ also extends by zero to a function in $\dot W^2_1(\R^\dmn)$; thus, we may extend $F$ to a polynomial of degree $m-1$ in $\R^\dmn_-$. Without loss of generality we may take this to be the zero polynomial. Thus, if $\Tr_{m-1}^+ F=0$, then $\1_+F\in \dot W^2_m(\R^\dmn)$.
We may apply the identity~\eqref{eqn:newton:identity} to $\1_+ F$, and so the right-hand side of formula~\eqref{dfn:D:newton} is zero in $\dot W^2_m(\R^\dmn)$, that is, up to adding polynomials of degree $m-1$.

We will need two alternative formulations of~$\D^{\mat A}\arr f$.
Notice that we may extend $ F$ to a $\dot W^2_m(\R^\dmn)$ function even if $\Tr_{m-1}^+ F=\arr f\neq 0$; then $\Tr_{m-1}^- F=\arr f$ as well. Then by formula~\eqref{eqn:newton:identity},
\begin{equation}
\label{dfn:D:newton:exterior}
\D^{\mat A}\arr f =
\1_- F
-\Pi^L(\1_- \mat A\nabla^m F)
\qquad\text{if } \arr f=\Tr_{m-1}^- F
.\end{equation}
By formula~\eqref{eqn:fundamental:2}, if $\abs{\alpha}=m$, then for almost every $x\in\R^n$ and $t>0$, we have that
\begin{align}
\label{eqn:D:fundamental}
\partial^\alpha \D^{\mat A} \arr f(x,t)
&=
	- \!\!\sum_{\substack{\scriptstyle\abs{\beta}=m\\\abs{\xi}=m}} \int_{\R^\dmn_-} \partial_{x,t}^\alpha \partial_{y,s}^\beta E^L(x,t,y,s)\,A_{\beta\xi}(y,s) \, \partial^\xi F(y,s)\,ds\,dy
.\end{align}
A corresponding formula, involving an integral over $\R^\dmn_+$, is valid if $t<0$.

We will establish a bound on $\D^{\mat A}\arr f$ in terms of the $L^2$ norm of the tangential derivative $\nabla_\pureH \arr f$ of~$\arr f$. In order to use existing theorems concerning $L^2$ boundedness, we will want to slightly modify the definition of the double layer potential, by defining
\begin{align}
\label{dfn:D:tilde}
\widetilde\D^{\mat A} (\Tr_{m,\semiH} F)(x,t)
&=
	\D^{\mat A} (\Tr_{m-1} F)(x,t)
\end{align}
for all sufficiently well-behaved functions~$F$.

We now must define the single layer potential.
Let $\arr g$ be a bounded linear operator on the space
\begin{equation*}\dot W\!A^2_{m-1/2}(\R^n)=\{\Tr_{m-1}^+  F: F\in \dot W^2_m({\R^\dmn_+})\}.\end{equation*}
The operator $T_{\arr g}  F=\langle \arr g, \Tr_{m-1}^+ F\rangle_{\R^n}$ is a bounded linear operator on $\dot W^2_m({\R^\dmn_+})$. We may identify $\dot W^2_m({\R^\dmn_+})$ with a subspace of $(L^2({\R^\dmn_+}))^q$, where $q$ is the number of multiindices $\alpha$ of length~$m$, via the map $ F\mapsto\nabla^m F$. We may then extend $T_{\arr g}$ to an operator on $(L^2({\R^\dmn_+}))^q$. Let $\arr G\in (L^2({\R^\dmn_+}))^q$ be the kernel of $T_{\arr g}$, so $T_{\arr g}(\arr H)=\langle \arr G,\arr H\rangle_{{\R^\dmn_+}}$ for all $\arr H\in (L^2({\R^\dmn_+}))^q$. In particular, $\langle \arr G, \nabla^m\varphi\rangle_{\R^\dmn_+} = \langle \arr g, \Tr_{m-1}^+\varphi\rangle_{\partial{\R^\dmn_+}}$ for all $\varphi\in \dot W^2_m({\R^\dmn_+})$. Let
\begin{equation}
\label{dfn:S:newton}
\s^{\mat A}\arr g=\Pi^L(\1_+\arr G)\quad\text{if } \langle \arr G, \nabla^m\varphi\rangle_{\R^\dmn_+} = \langle \arr g, \Tr_{m-1}^+\varphi\rangle_{\partial{\R^\dmn_+}} \text{ for all }\varphi\in \dot W^2_m.
\end{equation}
As in the case of the double layer potential, it is straightforward to establish that $\s^{\mat A}\arr g\in \dot W^2_m(\R^\dmn)$ and that the value of $\s^{\mat A}\arr g$ is independent of the choice of $\arr G$, that is, of the choice of extension of $T_{\arr g}$ from $\dot W^2_m({\R^\dmn_+})$ to $(L^2({\R^\dmn_+}))^q$. The formula $\s^{\mat A}(\M_{\mat A}^+ u)=\Pi^L(\1_+ \mat A\nabla^m u)$ follows immediately from the definitions \eqref{eqn:Neumann} and~\eqref{dfn:S:newton} of Neumann boundary data and the single layer potential.

Again we will need alternative formulations of layer potentials. First, observe that
\begin{equation}
\label{dfn:S:newton:exterior}
\s^{\mat A}\arr g=\Pi^L(\1_-\arr G)\quad\text{if } \langle \arr G, \nabla^m\varphi\rangle_{\R^\dmn_-} = \langle \arr g, \Tr_{m-1}^-\varphi\rangle_{\partial{\R^\dmn_+}} \text{ for all }\varphi\in \dot W^2_m.
\end{equation}
By formula~\eqref{eqn:fundamental:2}, if $\abs{\alpha}=m$ then
\begin{equation*}\partial^\alpha\s^{\mat A}\arr g(x,t)
= \partial^\alpha\Pi^L(\1_- \arr G)(x,t)
= \sum_{\abs{\beta}=m}
\int_{\R^\dmn_-} \partial^\alpha_{x,t} \partial_{y,s}^\beta E^L(x,t,y,s)\,G_{\beta}(y,s) \,ds\,dy.
\end{equation*}
But by the bound~\eqref{eqn:fundamental:far}, $\varphi(y,s)=\partial_{x,t}^\alpha E^L(x,t,y,s)$ is a $\dot W^2_m(\R^\dmn_-)$-function for almost every $x\in\R^n$ and $t>0$; thus we may write
\begin{equation}
\label{eqn:S:fundamental}
\partial^\alpha\s^{\mat A}\arr g(x,t)
=
\sum_{\abs{\gamma}=m-1}
\int_{\R^n} \partial_{x,t}^\alpha \partial_{y,s}^\gamma E^L(x,t,y,0)\,g_{\gamma}(y) \,dy
.\end{equation}

\subsection{Function spaces on the boundary}
\label{sec:whitney}

We have now defined $\D^{\mat A}$ and $ \s^{\mat A}$ as operators on $\dot W\!A^2_{m-1/2}(\R^n)$ and its dual space, respectively.
We wish to extend $\D^{\mat A}$ and $\s^{\mat A}$ to bounded operators on the space $L^2(\R^n)$.
However, notice that $\D^{\mat A}$ acts naturally only on traces of gradients; that is, density arguments will only allow us to extend $\D^{\mat A}$ to a subspace of $L^2(\R^n)$. We will define this subspace as follows.

\begin{dfn} We let $\dot W\!A^2_{m-1}(\R^n)$ be the completion of the set
\begin{equation*}\{\Tr_{m-1}\varphi:\varphi\text{ smooth and compactly supported}\}\end{equation*}
under the $L^2$ norm. % $\doublebar{\arr\varphi}_{\dot W\!A^2_{m-1}(\partial\Omega)} = \doublebar{\arr\varphi}_{L^2(\partial\Omega)}$.

We let $\dot W\!A^2_{m,\semiH}(\R^n)$ be the completion of the set
\begin{equation*}\mathfrak{D}=\{\Tr_{m,\semiH}\varphi:\varphi\text{ smooth and compactly supported}\}\end{equation*}
under the $L^2$ norm. % $\doublebar{\arr\varphi}_{\dot W\!A^2_{m-1}(\partial\Omega)} = \doublebar{\arr\varphi}_{L^2(\partial\Omega)}$.
\end{dfn}

It is well known that that the space $\dot W\!A^2_{m-1/2}(\R^n)$ used above is the completion of $\mathfrak{D}$ under the norm in the Besov space ${\dot B^{2,2}_{1/2}(\R^n)}$. This space is often called a Whitney-Besov space and has been used in the theory of higher-order boundary-value problems; see, for example, \cite{AdoP98,Agr07,MazMS10,MitMW11,BreM13,MitM13A,Bar15p}. The spaces $\dot W\!A^2_{m-1}(\R^n)$ or $\dot W\!A^2_{m,\semiH}(\R^n)$  are called Whitney-Sobolev spaces; they have also been used extensively in the theory, for example, in \cite{Ver90,PipV95A,PipV95B,Ver96,She06A,She06B,KilS11b}. The goal of this paper is to extend the double and single layer potentials to bounded operators on Whitney-Sobolev spaces by establishing boundedness results.

%As noted in Remark~\ref{rmk:potentials:defined}, $\widetilde\D^{\mat A}$ is defined for all $\arr\varphi\in \mathfrak{D}$; our goal in this paper is to show that $\widetilde\D^{\mat A}$ extends to an operator defined on all of $\dot W\!A^2_{m,\semiH}(\R^n)$. This space naturally occurs in the theory of boundary-value problems with $L^2$ boundary data; see \cite{Ver90,PipV95B,Ver96}. Thus, we hope that our results may be useful in analyzing such boundary-value problems.

\begin{rmk}\label{rmk:potentials:defined}
We remark that $\widetilde \D^{\mat A}$ is a well-defined operator on the space
\begin{equation*}\mathfrak{D}=\{\Tr_{m,\semiH}\varphi: \varphi\text{ is smooth and compactly supported}\}.\end{equation*}
we will extend $\widetilde\D^{\mat A}$ to $\dot W\!A^2_{m,\semiH}(\R^n)$ by density.

We will also extend the single layer potential; in this case we wish to extend $\s^{\mat A}$ to all arrays of functions $\arr g\in L^2(\R^n)$. It will be convenient to have a dense subspace~$\mathfrak{N}$ at our disposal on which $\s^{\mat A}$ is known to be well-defined. We claim that $\s^{\mat A}\arr g$ is well-defined for any $\arr g\in L^2(\R^n)$ that is compactly supported and integrates to zero.

The reasoning is as follows. Recall that $\s^{\mat A}\arr g$ is well-defined whenever $\arr g$ is a bounded linear operator on
\begin{equation*}\dot W\!A^2_{m-1/2}(\R^n)=\{\Tr_{m-1} \Phi:\nabla^m\Phi\in L^2(\R^\dmn_+)\}.\end{equation*}
Now, suppose that $\int g_\gamma=0$. Choose some function $\Phi\in \dot W^2_m(\R^\dmn_+)$ that is smooth up to the boundary.
Then
\begin{equation*}\int_{\R^n} g_\gamma\, \partial^\gamma \Phi = \int_{\R^n} g_\gamma\, (\partial^\gamma \Phi-c_\Phi)\end{equation*}
for any constant $c_\Phi$. Suppose that $g_\gamma$ is
supported in $\R^n\cap B((x_0,0),R)$ for some $x_0\in\R^n$ and some $R>0$. Let $\Omega=\R^\dmn_+\cap B((x_0,0),R)$. It is well known (see, for example, \cite{Eva98}) the trace map is bounded from $L^2(\Omega)\cap \dot W^2_1(\Omega)$ to $L^2(\partial\Omega)$. Thus,
\begin{align*}\abs[bigg]{\int_{\R^n} g_\gamma\, \partial^\gamma \Phi}
&\leq
	\doublebar{g_\gamma}_{L^2(\R^n)} \doublebar{\partial^\gamma\Phi-c_\Phi}_{L^2(\partial\Omega)}
\\&\leq
	CR^{1/2}\doublebar{g_\gamma}_{L^2(\R^n)} \doublebar{\partial^\gamma\Phi-c_\Phi}_{L^2(\Omega)}
	\\&\qquad+
	CR^{3/2}\doublebar{g_\gamma}_{L^2(\R^n)} \doublebar{\nabla\partial^\gamma\Phi}_{L^2(\Omega)}
.\end{align*}
By the Poincar\'e inequality, if we choose $c_\Phi$ correctly then we may control the quantity $R^{-1}\doublebar{\partial^\gamma\Phi-c_\Phi}_{L^2(\Omega)}$ by $\doublebar{\nabla\partial^\gamma\Phi}_{L^2(\Omega)}\leq \doublebar{\nabla^m\Phi}_{L^2(\R^\dmn_+)}$, and so we see that $\arr g$ gives rise to a bounded operator on $\dot W\!A^2_{m-1/2}(\R^n)$. Thus, for such~$\arr g$, $\s^{\mat A}\arr g$ is a well-defined element of $\dot W^2_m(\R^\dmn)$.
\end{rmk}

\section{Preliminary arguments}
\label{sec:preliminary}

In this section we will establish some basic results that will be useful throughout the paper.

We begin with some bounds on solutions to elliptic equations. Specifically, we begin with the following higher-order generalization of the Caccioppoli inequality; in its full generality it was proven in \cite{Bar14p}, but the $j=m$ case was proven in \cite{Cam80} and an intriguing version appears in \cite{AusQ00}.

\begin{lem}[The Caccioppoli inequality]\label{lem:Caccioppoli}
Suppose that $L$ is a divergence-form elliptic operator associated to coefficients $\mat A$ satisfying the ellipticity conditions \eqref{eqn:elliptic} and~\eqref{eqn:elliptic:bounded}. Let $ u\in \dot W^2_m(B(X_0,2r))$ with $L u=0$ in $B(X_0,2r)$.

Then we have the bound
\begin{equation*}
\fint_{B(X,r)} \abs{\nabla^j  u(x,s)}^2\,dx\,ds
\leq \frac{C}{r^2}\fint_{B(X,2r)} \abs{\nabla^{j-1}  u(x,s)}^2\,dx\,ds
\end{equation*}
for any $j$ with $1\leq j\leq m$.
\end{lem}

We may use the following lemma to bound $u$ not in balls of dimension~$\dmn$, but on horizontal slices of dimension~$\dmnMinusOne$; the second-order case of this lemma is known and is Proposition~2.1 in \cite{AlfAAHK11}.

\begin{lem}\label{lem:slices}
Let $t$ be a constant, and let $Q\subset\R^n$ be a cube.

Suppose that $\partial_s \arr u(x,s)$ satisfies the Caccioppoli-like inequality
\begin{equation*}
\fint_{B(X,r)} \abs{\partial_s \arr u(x,s)}^2\,dx\,ds
\leq \frac{c_0}{r^2}\fint_{B(X,2r)} \abs{\arr u(x,s)}^2\,dx\,ds
\end{equation*}
whenever $B(X,2r)\subset\{(x,s):x\in 2Q, t-\ell(Q)<s<t+\ell(Q)\}$. %, for some $1\leq p\leq\infty$.

Then
\begin{equation*}\int_Q \abs{\arr u(x,t)}^2\,dx \leq \frac{C(c_0)}{\ell(Q)}
\int_{2Q}\int_{t-\ell(Q)}^{t+\ell(Q)} \abs{\arr u(x,s)}^2\,ds\,dx.\end{equation*}

In particular, if $Lu=0$ in $2Q\times(t-\ell(Q),t+\ell(Q))$, and $L$ is an operator of order~$2m$ associated to $t$-independent coefficients~$A$, then
\begin{equation*}\int_Q \abs{\nabla^j \partial_t^k u(x,t)}^2\,dx \leq \frac{C}{\ell(Q)}
\int_{2Q}\int_{t-\ell(Q)}^{t+\ell(Q)} \abs{\nabla^j \partial_s^k u(x,s)}^2\,ds\,dx\end{equation*}
for any $0\leq j\leq m$ and any integer $k\geq 0$.
\end{lem}

\begin{proof}
Begin by observing that
\begin{align*}
\biggl(\int_Q \abs{\arr u(x,t)}^2\,dx \biggr)^{1/2}
&\leq
	\biggl(\int_Q \abs[bigg]{\arr u(x,t)-\fint_{t}^{t+\ell(Q)/2} \arr u(x,s)\,ds}^2\,dx \biggr)^{1/2}
	\\&\qquad+
	\biggl(\int_Q \fint_{t}^{t+\ell(Q)/2} \abs{\arr u(x,s)}^2\,ds\,dx \biggr)^{1/2}
.\end{align*}
But
\begin{align*}
\int_Q \abs[bigg]{\arr u(x,t)-\fint_{t}^{t+\ell(Q)/2} \arr u(x,s)\,ds}^2\,dx
&\leq
	\int_Q \abs[bigg]{\fint_{0}^{\ell(Q)/2}
	\int_0^s \partial_r \arr u(x,t+r)\,dr\,ds}^2\,dx
\\&\leq
	\int_Q \int_{0}^{\ell(Q)/2}
	\abs{\partial_r \arr u(x,t+r)}^2\,dr\,dx
.\end{align*}
Applying the Caccioppoli inequality completes the proof.
\end{proof}

Throughout this paper we will frequently need to bound the fundamental solution of Theorem~\ref{thm:fundamental} on horizontal slices.
The following estimate follows from Lemma~\ref{lem:slices}, the Caccioppoli inequality and the bound~\eqref{eqn:fundamental:far}; we will use it several times. Suppose that $Q$ is a cube and that either $j\geq 1$ or $j\geq 0$ and $\ell(Q)\leq \abs{s-t}$.
Suppose further that $q$, $s$, $i$ and $k$ are nonnegative integers with $q\leq m$, $s\leq m$ and $q-k<\pdmn/2$, $s-i\leq \pdmn/2$.
Then for some $\varepsilon>0$,
\begin{equation}
\label{eqn:fundamental:slices}
\int_Q \int_{A_j(Q)} \abs{\nabla_{x,t}^{m-q} \nabla^{m-s}_{y,s} \partial_t^k \partial_s^i E^L(x,t,y,s)}^2\,dy\,dx \leq \frac{C}{\ell(Q)^{2r+2} } 2^{-j(2(i-s)+1+\varepsilon)}
\end{equation}
where $r = k+i-q-s= (m-q)+(m-s)+k+i-2m$. In applying this formula it is always useful to remember formula~\eqref{eqn:fundamental:vertical}, that is, that we may take vertical derivatives in either the $s$ variable or the $t$ variable.

Now, recall that we seek to bound the double layer potential~$\D^{\mat A}$. Furthermore, recall that if $\arr f=\Tr_{m-1} F$, then we may write $\D^{\mat A}\arr f$ in terms of~$F$; see formulas \eqref{dfn:D:newton} or~\eqref{eqn:D:fundamental}. Thus, we will be concerned with extensions of elements of $\dot W\!A^2_{m,\semiH}(\R^n)$, or of elements of the dense subspace $\mathfrak{D}$ of Section~\ref{sec:whitney}. The following lemma provides some basic extensions and some estimates upon such functions.

\begin{lem}\label{lem:extension}
Let $\arr f=\Tr_{m-1} F$ for some smooth, compactly supported function~$F$.

Then there is some function~$H$ defined in~$\R^\dmn_+$ such that $\Tr_{m-1} H = \arr f$ and such that
\begin{align}
\label{eqn:extension:L2:space}
\doublebar{\nabla^m H}_{L^2(\R^\dmn)}^2
&\leq C \int_{\R^n} \abs{\xi}\,\abs{\widehat{\arr f}(\xi)}^2\,d\xi
,\\
\label{eqn:extension:L2:slices}
\sup_{t\neq 0} \doublebar{\nabla^{m-1}H(\,\cdot\,,t)}_{L^2(\R^n)}^2
&\leq \int_{\R^n} \abs{\arr f(x)}^2\,dx
,\\
\label{eqn:extension:L2:slices:2}
\sup_{t\neq 0} \doublebar{\nabla^{m}H(\,\cdot\,,t)}_{L^2(\R^n)}^2
&\leq \int_{\R^n} \abs{\nabla_\pureH\arr f(x)}^2\,dx
,\\
\label{eqn:extension:square}
\int_{\R^n}\int_0^\infty\abs{\nabla^m H(x,t)}^2\,{t}\,dt\,dx
&\leq \int_{\R^n} \abs{\arr f(x)}^2\,dx
.\end{align}
Furthermore, if $\arr f=0$ in some cube $Q$, then $\nabla^{m-1} H=0$ in $\{(x,t):\dist(x,\R^n\setminus Q)>t\}$, and in particular in~$(1/2)Q\times (0,\ell(Q)/4)$.
\end{lem}

\begin{proof}
For each $0\leq j\leq m-1$, let $f_j(x)=\partial_\dmn^j F(x,0)$; observe that up to adding polynomials of appropriate degree, $f_j$ is determined entirely by $\arr f = \nabla^{m-1} F(x,0)$.

Let $\eta:\R^n\mapsto\R$ be smooth, nonnegative, supported in $B(0,1)$, and satisfy $\int_{\R^n} \eta=1$ and $\int_{\R^n} x^\zeta \,\eta(x)\,dx=0$ for all multiindices $\zeta\in \N^n$ with $1\leq \abs{\zeta}\leq m-1$. Let $\eta_t(x)=t^{-n}\eta(x/t)$.
Let
\begin{equation*}H_j(x,t) = \frac{1}{j!} t^j\, f_j*\eta_t(x)=\frac{1}{j!} t^j \int_{\R^n} f_j(x-ty)\,\eta(y)\,dy,
\quad H(x,t)=\sum_{j=0}^{m-1} H_j(x,t).\end{equation*}
By inspection, $\lim_{t\to 0}\partial_t^j H_j(x,t) =f_j(x)$, and if $0\leq k\leq m-1$ with $j\neq k$, then
$\lim_{t\to 0}\partial_t^k H_j(x,t) =0$ if $j\neq k$. Thus $\Tr_{m-1} H= \arr f$, as desired.

We may bound $H$ in terms the functions $f_j$ using the Fourier transform in the $x$-variable and Plancherel's theorem, and we may bound appropriate derivatives of~$f_j$ using the array~$\arr f$. We omit the routine details.

If $\arr f=0$ in~$Q$, then $\nabla^{m-1-j} f_j=0$ in $Q$, and so $f_j$ is a polynomial in~$Q$. Thus, we may write $f_j(x-ty)=\sum_{\abs{\gamma}<m-1-j} t^{\abs{\gamma}} y^\gamma P_\gamma(x)$ for some polynomials $P_\gamma(x)$. Notice $P_0(x)=f_j(x)$. By our moment condition on~$\eta$, if $\dist(x,\R^n\setminus Q)>t$, then \begin{equation*}H_j(x,t)=\frac{1}{j!} t^j \int_{\R^n} f_j(x-ty)\,\eta(y)\,dy=\frac{1}{j!} t^j \int_{\R^n} P_0(x)\,\eta(y)\,dy=\frac{1}{j!} t^j f_j(x).\end{equation*}
Thus, $H(x,t)$ is equal to a polynomial of degree at most $m-2$ in this region, as desired.
\end{proof}

\section{Operators to be bounded}
\label{sec:Theta}

Recall that Theorem~\ref{thm:square} involves bounding the quantities $\nabla^m\partial_\dmn \s^{\mat A}\arr g$ and $\nabla^m\partial_\dmn \D^{\mat A}\arr f$. In this section we will reduce to the case of the purely vertical derivatives; that is, we will show that bounding $\partial_\dmn^{m+k} \s^{\mat A}\arr g$ and $\partial_\dmn^{m+k} \D^{\mat A}\arr f$, for any $k\geq 1$, suffices to bound $\nabla^m\partial_\dmn \s^{\mat A}\arr g$ and $\nabla^m\partial_\dmn \D^{\mat A}\arr f$. We will also establish some notation for these operators.

Let $k\geq 1$ be an integer, to be chosen later. Let
\begin{align}
\label{dfn:Theta:S}
\Theta_t^S  \arr g(x) &= t^k \partial_t^{m+k}\s^{\mat A} \arr g(x,t)
.\end{align}
Observe that by formula~\eqref{eqn:S:fundamental},
\begin{equation}
\label{dfn:Theta:S:integral}
\Theta_t^S \arr g(x) = \sum_{\abs{\gamma}=m-1} t^k \int_{\R^n} \partial_t^{m+k} \partial_{y,s}^\gamma E^L(x,t,y,0)\,g_{\gamma}(y)\,dy
.\end{equation}
Notice that by the bound~\eqref{eqn:fundamental:slices}, if $k$ is large enough and if $\arr g\in L^2(\R^\dmn)$, then the integral converges for almost every $(x,t)\in\R^\dmn_+$. We will elaborate on this point in Section~\ref{sec:decay}.

If $\arr f$ lies in the space $\mathfrak{D}$ of Remark~\ref{rmk:potentials:defined},
we let
\begin{align}
\label{dfn:Theta:D}
\Theta_t^D \arr f(x)
&= t^k \partial_t^{m+k}\widetilde \D^{\mat A} \arr f(x,t)
.\end{align}
Establishing a bound on $\Theta_t^D$ of some sort in terms of the $L^2$ norm of $\arr f$ will allow us to extend $\Theta_t^D$ to all of~${\dot W\!A^2_{m,\semiH}(\R^n)}$.
%If $\arr h=\Tr_{m-1} H$ for some~$H$, we let
%\begin{align}
%\label{dfn:Theta:D}
% \Theta_t^D \arr h(x)
%&= t^k \partial_t^{m+k-1} \D^{\mat A} \arr h(x,t)
%.\end{align}
%We will use various different projection operators to extend $\Theta_t^D$ to all of $L^2(\R^n\mapsto\C^p)$.
%We will extend $\Theta_t^D$ to an operator defined on all of $L^2(\R^n)$ in Section~\ref{sec:projection}.

Note that $\Theta_t^S$, $\Theta_t^D$ implicitly depend on $k\geq 1$.

We begin by reducing the proof of Theorem~\ref{thm:square} to establishing bounds on $\Theta_t^S$ and $\Theta_t^D$; the remainder of this paper will be devoted to establishing these bounds.

\begin{rmk} The conclusion of Theorem~\ref{thm:square} is a bound in the whole space $\R^\dmn$; for notational convenience, we will establish a bound only in the upper half-space $\R^\dmn_+$ and note that the corresponding bound in $\R^\dmn_-$ follows by careful argument involving the change of variables $(x,t)\mapsto (x,-t)$.
\end{rmk}

\begin{lem}\label{lem:square:Theta}
Let $\arr f\in\mathfrak{D}$ and $\arr g\in\mathfrak{N}$, where $\mathfrak{D}$ and $\mathfrak{N}$ are as in Remark~\ref{rmk:potentials:defined}.
If $k\geq 1$, then we have the bounds
\begin{align*}
\int_{\R^\dmn_+} \abs{\nabla^m \partial_t\s^{\mat A} \arr g(x,t)}^2\,t\,dx\,dt
&\leq
C \int_{\R^{n+1}_+}
	\abs{\Theta_t^S \arr g(x)}^2\,\frac{1}{t}\,dx\,dt
,\\
\int_{\R^\dmn_+} \abs{\nabla^m \partial_t\widetilde \D^{\mat A} \arr f(X)}^2\,t\,dx\,dt
&\leq
C \int_{\R^{n+1}_+}
	\abs{\Theta_t^D \arr f(x)}^2\,\frac{1}{t}\,dx\,dt
%,\\
%\int_{\Omega} \abs{\nabla^m \D^{\mat A} \arr h(X)}^2\,\dist(X,\partial\Omega)\,dX
%&\leq
%C \int_{\R^{n+1}_+}
%	\abs{ \Theta_t^D \arr h(x)}^2\,\frac{1}{t}\,dx\,dt
.\end{align*}
\end{lem}

\begin{proof}%[Proof of Lemma~\ref{lem:square:Theta}]
We follow the proof of the similar formula~(5.5) in~\cite{AlfAAHK11}.
Let $u=\s^{\mat A}\arr g$ or $u=\D^{\mat A}\arr f$, and define
\begin{equation*}U_j(t)=\int_{\R^n} \abs{\nabla^m \partial_t^j u(x,t)}^2 \,dx,\quad
V_j(t)=\int_{\R^n} \abs{\partial_t^{m+j} u(x,t)}^2 \,dx.\end{equation*}
To prove the lemma we need only establish the bound
\begin{equation*}\int_0^\infty t\,U_1(t)\,dt\leq C\int_0^\infty t^{2k-1} V_k(t)\,dt.\end{equation*}
By Lemma~\ref{lem:slices} and the Caccioppoli inequality, if $j\geq -m$ then
\begin{equation*}U_{m+j}(t)\leq \frac{C}{t^{2m}}\fint_{t/2}^{2t} V_{j}(s)\,ds\end{equation*}
and thus we may easily show that
\begin{equation*}\int_0^\infty t^{2m+2k-1} \,U_{m+k}(t)\,dt \leq \int_0^\infty t^{2k-1} V_k(t)\,dt.\end{equation*}
Thus, we need only show that
\begin{equation*}\int_0^\infty t\,U_1(t)\,dt\leq C\int_0^\infty t^{2m+2k-1} \,U_{m+k}(t)\,dt .\end{equation*}

Observe that $\nabla^m u\in L^2(\R^\dmn_+)$. By Lemma~\ref{lem:slices} and the Caccioppoli inequality, we have that if $j\geq 0$, then
\begin{equation*}U_j(t) \leq \frac{C}{t^{1+2j}} \doublebar{\nabla^m u}_{L^2(\R^\dmn_+)}^2.\end{equation*}
Suppose $j>0$. Then if $0<\varepsilon<S<\infty$, we have that
\begin{align*}
\int_{\varepsilon}^S t^{2j-1} \,U_j(t)\,dt
&=
\int_{\varepsilon}^S t^{2j-1}\, U_j(S)\,dt
-\int_{\varepsilon}^S t^{2j-1}\int_t^S U_j'(s)\,ds\,dt
\\&\leq
\frac{C}{S}\doublebar{\nabla^m u}_{L^2(\Omega)}^2
+\frac{1}{2j}\int_{\varepsilon}^Ss^{2j}\, \abs{U_j'(s)}\,ds.
\end{align*}
Observe that
$\abs{U_j'(s)}
\leq 2 \sqrt{U_j(s)\,U_{j+1}(s)}\leq \frac{1}{s}U_j(s)+s\,U_{j+1}(s)
$.
Thus,
\begin{align*}
\int_{\varepsilon}^S t^{2j-1} \,U_j(t)\,dt
&\leq
	\frac{C}{S}\doublebar{\nabla^m u}_{L^2(\Omega)}^2
	+\frac{1}{2j}\int_{\varepsilon}^Ss^{2j-1}\, U_j(s)\,ds
	\\&\qquad
	+\frac{1}{2j}\int_{\varepsilon}^Ss^{2j+1}\,U_{j+1}(s)\,ds
.\end{align*}
Rearranging terms, we have that if $j\geq 1$ then
\begin{align*}
\int_{\varepsilon}^S t^{2j-1} \,U_j(t)\,dt
&\leq
	\frac{C}{S}\doublebar{\nabla^m u}_{L^2(\Omega)}^2
	+C\int_{\varepsilon}^Ss^{2j+1}\,U_{j+1}(s)\,ds
.\end{align*}
Taking the limit as $\varepsilon\to 0^+$ and $S\to \infty$, we have that if $j>0$ then
\begin{align*}
\int_0^\infty t^{2j-1} \,U_j(t)\,dt
&\leq
	C\int_0^\infty s^{2j+1}\,U_{j+1}(s)\,ds
.\end{align*}
Iterating, we see that
\begin{equation*}\int_0^\infty t \,U_1(t)\,dt
\leq C(k) \int_0^\infty t^{2m+2k-1} \,U_{m+k}(t)\,dt\end{equation*}
as desired.\end{proof}

%In the remainder of Section~\ref{sec:square} we will bound the operators $\Theta_t^D$ and~$\Theta_t^S$. We will frequently use the assumption that $k\geq 1$ is large, as the quantity $\partial_t^{m+k}\partial_{y,s}^\gamma E^L(x,t,y,0)$ decays rapidly as $y\to \infty$ if $k$ is large.
\begin{rmk} In Sections~\ref{sec:Tb}--\ref{sec:high} we will bound the operators $\Theta_t^D$ and $\Theta_t^S$ for some values of~$k$. In particular, we will make many arguments that are only valid for $k$ large enough; we will not make any arguments that are only valid for $k$ small enough, and thus there will be some $k$ large enough that all our arguments are valid.
\end{rmk}

\section{A vector-valued \texorpdfstring{$T(b)$}{Tb} theorem}
\label{sec:Tb}

Our goal now is to produce square-function estimates for the operators $\Theta_t^S$ and $\Theta_t^D$. In this section, we will review some known theorems that may be used to establish square-function estimates on singular integral operators.

We begin with one of the first such results, the Christ-Journ\'e $T1$ theorem from Section~2 of \cite{ChrJ87}, which is a square function analogue of the well-known result of David and Journ\'e
\cite{DavJ84}.

\begin{thm}\label{thm:CJ}
Suppose that the family of linear operators $\{\Theta_t\}_{t>0}$ are given by
\begin{equation*}\Theta_t f(x) = \int_{\R^\dmn}\psi_t(x,y)\,f(y)\,dy\end{equation*}
for some kernels $\psi_t$ that satisfy
\begin{align*}
\abs{\psi_t(x,y)}&\leq C_0
\frac{t^\varepsilon}{(t+\abs{x-y})^{n+\varepsilon}}
,\\
\abs{\psi_t(x,y)-\psi_t(x,z)}&\leq C_1 \frac{t^\varepsilon\abs{y-z}^\varepsilon}{(t+\abs{x-y})^{n+2\varepsilon}}
 &&\text{for all } \abs{y-z} \leq \frac{1}{2}(t+\abs{x-y})
\end{align*}
for some constants $C_0$, $C_1$ and some $\varepsilon>0$.
If
\begin{equation*}
\sup_{Q\subset\R^n} \frac{1}{\abs{Q}}\int_0^{\ell(Q)} \int_Q \abs{\Theta_t 1(x)}^2\frac{dx\,dt}{t}\leq C_2\end{equation*}
then we have the bound
\begin{equation*}\int_0^\infty \int_{\R^n} \abs{\Theta_t f(x)}^2\,\frac{dx\,dt}{t} \leq C\doublebar{f}_{L^2(\R^n)}^2\end{equation*}
where $C$ depends only on the constants $\varepsilon$, $C_0$, $C_1$, $C_2$ and the dimension~$\dmn$.
\end{thm}
We will use this theorem directly in Section~\ref{sec:b:S} below. However, this theorem is too restrictive to apply to the operators $\Theta_t^D$ and~$\Theta_t^S$. (In particular, $\Theta_t^D$ and $\Theta_t^S$ lack smooth kernels.) There are many generalizations of this theorem; we will need the following $T1$ and $Tb$ theorems from \cite{GraH14p}. (We will define a CLP family in Section~\ref{sec:Q}.)

\begin{thm}[{\cite[Theorem~4.5]{GraH14p}}]\label{thm:grau:hofmann:1}
Consider a family of operators $\{\Theta_t\}_{t>0}$ taking values in $\C^{p+1}$, $p\geq 0$, so that $\Theta_t =(\Theta_t^1,\Theta_t^2,\dots,\Theta_t^{p+1})$, where each $\Theta_t^j$ acts on scalar-valued $L^2(\R^n)$, and where for $\vec g=(g_1,g_2,\dots,g_{p+1})\in L^2(\R^n\mapsto\C^{p+1})$, we set
\begin{equation*}%\label{eqn:Theta}
\Theta_t\vec g = \sum_{j=1}^{p+1} \Theta_t^j g_j
.\end{equation*}

Suppose that there is some $\theta>0$ and some $C>0$ such that, for all dyadic cubes~$Q$, all integers $j\geq 0$, and all functions $\vec g_j\in L^2(A_j(Q))$, where $A_j(Q)$ is as in formula~\eqref{eqn:annuli}, we have the estimate
\begin{align}
\label{eqn:Theta:decay}
\doublebar{\Theta_t (\1_{A_j(Q)}\vec g^j)}_{L^2(Q)}
&\leq C 2^{-j(n+2+\theta)/2} \doublebar{\vec g^j}_{L^2(A_j(Q))}
\quad\text{if }\ell(Q)\leq t\leq 2\ell(Q).
\end{align}
Suppose further that for some $\theta>0$, some CLP family of operators $Q_s$, and some subspace $H$ of $L^2(\R^n)$, we have that
\begin{align}
\label{eqn:Theta:Q}
\doublebar{\Theta_t Q_s \vec h}_{L^2(\R^n)}
&\leq C\biggl(\frac{s}{t}\biggr)^\theta \doublebar{\vec h}_{L^2(\R^n)}
\quad \text{for all }\vec h\in H\text{ and all }s\leq t.
\end{align}
Finally, suppose that
\begin{align}
\label{eqn:Theta:1:carleson}
\int_0^{\ell(Q)} \int_Q \abs{\Theta_t 1(x)}^2\frac{dx\,dt}{t}
&\leq C_0\abs{Q}
\end{align}
where $1$ denotes the $(p+1)\times(p+1)$ identity matrix. Equivalently, we may require that
\begin{align*}
%\label{eqn:Theta:1:carleson}
\int_0^{\ell(Q)} \int_Q \abs{\Theta_t^j 1(x)}^2\frac{dx\,dt}{t}
&\leq C_0\abs{Q}
\end{align*}
for each $1\leq j\leq p+1$, where $1$ denotes the function that is one everywhere.

Then for all $\vec f\in H$, we have that
\begin{equation}\label{eqn:Theta:1:square}\int_0^\infty\int_{\R^n} \abs{\Theta_t \vec f(x)}^2\frac{dx\,dt}{t}
\leq C\doublebar{\vec f}_{L^2(\R^n)}^2.\end{equation}
\end{thm}
%Here $\vec 1$ is the vector-valued function $\vec 1(x)=(1,1,\dots,1)$.
%We will provide the definition of a CLP family in Section~\ref{sec:Q}, when we prove that our operators $\Theta_t^S$ and $\Theta_t^D$ satisfy the estimate~\eqref{eqn:Theta:Q}.

\begin{rmk}\label{rmk:Theta:L2}
The uniform $L^2$ bound
\begin{align}
\label{eqn:Theta:L2}
\sup_{t>0}\doublebar{\Theta_t \vec g}_{L^2(\R^n)}
&\leq C \doublebar{\vec g}_{L^2(\R^n)}
\end{align}
follows from the bound~\eqref{eqn:Theta:decay} by summing over dyadic cubes $Q$ of side-length~$2^j$, $2^j\leq t<2^{j+1}$. In particular, establishing the bound~\eqref{eqn:Theta:decay} suffices to show that $\Theta_t$ is defined on $L^2(\R^n)$.
\end{rmk}

The major advantage of Theorem~\ref{thm:grau:hofmann:1} over Theorem~\ref{thm:CJ}, from our perspective, is that we need not have pointwise estimates on the kernels of our operators~$\Theta_t$. Rough kernels appear in the theory of second-order equations (see \cite[Section~3]{GraH14p}) and are an essential part of our treatment of layer potentials for higher-order equations.  On the other hand, we note that the proof of Theorem~\ref{thm:grau:hofmann:1} is an easy modification of that of Theorem~\ref{thm:CJ}; to the best of our knowledge, Grau de la Herran and Hofmann are the first authors to treat square-function estimates via $Tb$ theorems in this generality.

We now outline the bounds that we will prove using Theorem~\ref{thm:grau:hofmann:1}.

In Section~\ref{sec:decay}, we will show that $\Theta_t^D$ and $\Theta_t^S$ satisfy the estimate~\eqref{eqn:Theta:decay}. Notice that $\Theta_t^D$ is required to satisfy this estimate for all $\arr g\in L^2(\R^n)$, not only all $\arr g\in \dot W\!A^2_{m,\semiH}(\R^n)\subset L^2(\R^n\mapsto \C^{q})$; to deal with this technical requirement, in Section~\ref{sec:extension} we will extend $\Theta_t^D$ to an operator defined on all of~$L^2(\R^\dmnMinusOne)$. This extension is fairly artificial and is used only for this technical requirement; a similar extension will be used for another purpose in Section~\ref{sec:b:S}.

In Section~\ref{sec:Q} we will show that $\Theta_t^S$ satisfies the estimate~\eqref{eqn:Theta:Q} for all $\arr h\in L^2(\R^n)$, and that $\Theta_t^D$ satisfies the estimate~\eqref{eqn:Theta:Q} for all $\arr h\in \dot W\!A^2_{m,\semiH}(\R^n)$. (We do not need to extend this estimate to all $\arr h\in L^2(\R^n)$.)

Finally, in Section~\ref{sec:S:not-normal}, we will show that if
\begin{align}
\label{eqn:Theta:perp}
\Theta_t^\perp f(x) &= \Theta_t^S(f\arr e_\perp)(x),
\\
\label{eqn:Theta:prime}
\Theta_t^{S'} \arr f(x) &= \sum_{\gamma_\dmn<\abs{\gamma}=m-1} \Theta_t^S (f_\gamma \arr e_\gamma)(x),
\end{align}
then $\Theta_t^{S'}\arr 1=0$ for almost every $x$ and~$t$. Thus, the estimate~\eqref{eqn:Theta:1:square} is valid for $\Theta_t=\Theta_t^{S'}$.  Indeed, one can see that the splitting \mbox{\eqref{eqn:Theta:perp}--\eqref{eqn:Theta:prime}} corresponds to the case when $\gamma=(0,\dots,m-1)$  in formula~\eqref{dfn:Theta:S:integral} (that is, all derivatives under the integral are in $t,s$) and the case when each term of the integrand has at least one $y$-derivative, respectively.

We will not be able to show directly that $\Theta_t^\perp 1$ or $\Theta_t^D \arr 1$ satisfy the bound~\eqref{eqn:Theta:1:carleson}. In Section~\ref{sec:D:carleson}, we will show that if $2m>n$, then
\begin{equation*}\Theta_t^D \arr 1(x) = \Upsilon_t(x)+\Theta_t^S \arr a(x)\end{equation*}
where $\Upsilon_t$ is a Carleson measure (that is, satisfies the estimate~\eqref{eqn:Theta:1:carleson}) and where $\arr a(x)$ is a uniformly bounded function. Standard techniques will allow us to control $\Theta_t^S \arr a$ by $\Theta_t^S \arr 1$, using only the fact that $\arr a$ is bounded (that is, without using any special cancellation properties); see Lemma~\ref{lem:square:carleson} below. Thus, a bound on $\Theta_t^\perp 1$ together with the equation $\Theta_t^{S'}\arr 1=0$ will give us a bound on $\Theta_t^D\arr 1$ and thus allow us to use Theorem~\ref{thm:grau:hofmann:1}.

However, this argument does require control on $\Theta_t^\perp$, and Theorem~\ref{thm:grau:hofmann:1} will not suffice to bound~$\Theta_t^\perp$.

We will bound $\Theta_t^\perp$ (giving us a bound on $\Theta_t^D$) using the following theorem, with $\Theta_t^{p+1}=\Theta_t^\perp$ and $\Theta_t'=(\Theta_t^D,\Theta_t^{S'})$.

\begin{thm}[{\cite[Theorem~2.13]{GraH14p}}]\label{thm:grau:hofmann}
Consider a family $\{\Theta_t\}_{t>0}$ of operators taking values in $\C^{p+1}$, $p\geq 0$, so that $\Theta_t = (\Theta_t',\Theta_t^{p+1})=(\Theta_t^1,\Theta_t^2,\dots,\Theta_t^{p+1})$, where each $\Theta_t^j$ acts on scalar-valued $L^2(\R^n)$, and where for $\vec g=( g',g_{p+1})=(g_1,g_2,\dots,g_{p+1})\in L^2(\R^n\mapsto\C^{p+1})$, we set
\begin{equation*}
\Theta_t\vec g = \sum_{j=1}^{p+1} \Theta_t^j g_j,\qquad \Theta_t' g' = \sum_{j=1}^p \Theta_t^j g_j
.\end{equation*}

Suppose that $\Theta_t$ satisfies the bound~\eqref{eqn:Theta:decay}, that $\Theta_t^{p+1}$ satisfies the bound~\eqref{eqn:Theta:Q} for all $h\in L^2(\R^n)$, and that there is some subspace $H'\subset L^2(\R^n\mapsto \C^p)$ such that $\Theta_t'$ satisfies the bound~\eqref{eqn:Theta:Q} for all $\arr h\in H'$.

We define the $\mathcal{C},\delta$-norm as
\begin{equation*}\doublebar{\Upsilon_t}_{\mathcal{C},\delta}^2 = \sup_{\ell(Q)>\delta} \frac{1}{\abs{Q}} \int_\delta^{\min(\ell(Q),1/\delta)} \int_Q \abs{\Upsilon_t(x)}^2\frac{dx\,dt}{t}.\end{equation*}
Suppose that
\begin{align}
\label{eqn:Theta:T1}
\doublebar{\Theta'_t 1}_{\mathcal{C},\delta}
&\leq C_1 + C_1\doublebar{\Theta^{p+1}_t 1}_{\mathcal{C},\delta}
\qquad\text{for all $\delta>0$ small enough}
\end{align}
where $1$ denotes either the $p\times p$ identity matrix or the number one.

Suppose that for each dyadic cube $Q\subset\R^n$, we have a measure $\mu_Q$ such that
\begin{equation}
\label{eqn:mu}
d\mu_Q=\phi_Q\,dx,\qquad \doublebar{\nabla\phi_Q}_{L^\infty(\R^n)}\leq C_0 \ell(Q)^{-1}, \qquad \frac{1}{C_0}\leq \phi_Q \text{ on }Q.\end{equation}

Suppose further that for each such $Q$ there exists a vector-valued function $\arr b_Q=(\arr b_Q',b_Q^{p+1})\in L^2(\R^n)\times H'$ such that
\begin{align}
\label{eqn:b:carleson}
\int_0^{\ell(Q)} \int_Q \abs{\Theta_t \arr b_Q(x)}^2\frac{dx\,dt}{t}
&\leq C_0\abs{Q}
,\\
\label{eqn:b:L2}
\int_{\R^n} \abs{\arr b_Q(x)}^2\,dx
&\leq C_0\abs{Q}
,\\
\label{eqn:b:below}
\re \fint_Q b_Q^{p+1} \,d\mu_Q &\geq \sigma
,\\
\label{eqn:b:above}
\abs[bigg]{\fint_Q \arr b_Q' \,d\mu_Q} &\leq \eta \sigma, && \eta\leq 1/(2C_1+4)
.\end{align}

Then for all $ f\in H'\times L^2(\R^n)$,
\begin{equation}\label{eqn:Theta:square}
\int_0^\infty\int_{\R^n} \abs{\Theta_t  f(x)}^2\frac{dx\,dt}{t}
\leq C\doublebar{f}_{L^2(\R^n)}^2
.\end{equation}
\end{thm}

This theorem is a local $Tb$ theorem; that is, we may test $\Theta_t \arr b_Q$ near~$Q$, for some $\arr b_Q$ adapted to our particular cube~$Q$, rather than testing $\Theta_t \arr 1$ in an arbitrary cube. There is an extensive body of work devoted to generalizing $T1$ theorems to $Tb$ theorems and local $Tb$ theorems; see, for example, the survey paper \cite{Hof10}, and in particular \cite{McIM85,DavJS85,Sem90,Chr90} for a few of the important milestones of the theory.

As mentioned above, we will establish the bound~\eqref{eqn:Theta:T1} in Section~\ref{sec:D:carleson}, for $\Theta_t^{p+1}=\Theta_t^\perp$ and $\Theta_t'=(\Theta_t^D,\Theta_t^{S'})$, provided $2m>n$. We will construct the measure $\mu_Q$ and test functions $\arr b_Q=(\arr b_Q^S,\arr b_Q^D)$ at the beginning of Section~\ref{sec:b}, and therein will establish the estimates~\eqref{eqn:b:carleson}; we will establish the bounds~\eqref{eqn:b:L2}, \eqref{eqn:b:below} and~\eqref{eqn:b:above} in Sections~\ref{sec:b:D} and~\ref{sec:b:S}. The assumption $2m>n$ will be useful in Section~\ref{sec:b} as well as Section~\ref{sec:D:carleson}. This will allow us to bound $\Theta_t^\perp$, and so together with Lemma~\ref{lem:square:Theta} will complete the proof of Theorem~\ref{thm:square} in the case $2m>n$. We will extend to the case $2m\leq n$ in Section~\ref{sec:high}.

\section{The decay estimate \myeqref{eqn:Theta:decay}}
\label{sec:decay}

In this section, we will show that the operators $\Theta_t^S$ and $\Theta_t^D$ satisfy the bound~\eqref{eqn:Theta:decay} for all $\arr g^j$ in $L^2(A_j(Q))$.

By formula~\eqref{dfn:Theta:S:integral} for $\Theta_t^S$,
\begin{align*}
\int_Q \abs{ \Theta_t^S\arr g^j}^2
&=
	\int_Q t^{2k} \abs[bigg]{\sum_{\abs\gamma=m-1}\int_{A_{j}(Q)} \partial_t^{m+k} \partial^\gamma_{y,s} E^L(x,t,y,0)\,g_\gamma(y)\,dy}^2\,dx
.\end{align*}
By H\"older's inequality
\begin{align*}
\int_Q \abs{ \Theta_t^S\arr g^j}^2
&\leq
	C\doublebar{\arr g}_{L^2(A_j(Q))}\int_Q t^{2k} \int_{A_{j}(Q)} \abs{\partial_t^{m+k} \nabla^{m-1}_{y,s} E^L(x,t,y,0)}^2\,dy\,dx
.\end{align*}
Finally, by the bound~\eqref{eqn:fundamental:slices} on the fundamental solution,
\begin{align*}
\int_Q \abs{ \Theta_t^S\arr g^j}^2
&\leq
	C 2^{-j(2k-1+\varepsilon)}
	\doublebar{\arr g}_{L^2(A_j(Q))}^2
.\end{align*}
Thus, if $k$ is large enough then the operator $\Theta_t^S$ satisfies bound~\eqref{eqn:Theta:decay}.

\begin{rmk}
Suppose that $\arr g^j(z)=\partial_{z_i} \arr h^j(z)$ for some $1\leq i\leq n$ and some $\arr h^j$ supported in $A_{j,1}(Q)$. Then
\begin{equation*}
\int_Q \abs{ \Theta_t^S(\partial_i \arr h^j)}^2
=
	\int_Q t^{2k} \abs[bigg]{\sum_{\abs\gamma=m-1}\int_{A_{j,1}(Q)} \partial_t^{m+k} \partial^\gamma_{y,s} E^L(x,t,y,0)\,\partial_{y_i} h_\gamma^j(y)\,dy}^2\,dx
.\end{equation*}
Integrating by parts in~$y_i$, and applying the bound~\eqref{eqn:fundamental:slices}, we see that
\begin{equation*}
\int_Q \abs{ \Theta_t^S(\partial_i \arr h^j)}^2
\leq
	\frac{C}{t^2}2^{-j(2k+1+\varepsilon)}\doublebar{\arr h^j}_{L^2(A_{j,1}(Q))}^2
.\end{equation*}
In particular, for any $\arr h\in L^2(\R^n)$, we have the uniform $L^2$ estimate
\begin{equation}\label{eqn:S:div}
\doublebar{\Theta_t^S (\partial_{i} \arr h)}_{L^2(\R^n)} \leq \frac{C}{t} \doublebar{\arr h}_{L^2(\R^n)},\qquad 1\leq i\leq n.\end{equation}
This formula will be useful in Section~\ref{sec:Q}.
\end{rmk}

We now wish to show that $\Theta_t^D$ satisfies the decay estimate~\eqref{eqn:Theta:decay}, that is, that
\begin{align*}
\doublebar{\Theta_t^D \arr f^j}_{L^2(Q)}
&\leq C 2^{-j(n+2+\theta)/2} \doublebar{\arr f^j}_{L^2(A_j(Q))}
\qquad\text{for all }\ell(Q)\leq t\leq 2\ell(Q)
\end{align*}
for {all} $\arr f^j\in L^2(\R^n)$ and supported in $A_j(Q)$.

Choose some dyadic cube $Q$ and some $t$ with $\ell(Q)\leq t<2\ell(Q)$.
Suppose first that $\arr f^j\in \mathfrak{D}$, where $\mathfrak{D}$ is as in Remark~\ref{rmk:potentials:defined}, and is supported in $A_{j,1}(Q)$. Recall that $\widetilde\D^{\mat A}\arr f^j$ (and thus $\Theta_t^D\arr f^j$) is defined in terms of extensions $F^j$ of~$\arr f^j$. Thus, we begin by choosing an appropriate extension. Let $H^j$ be the function given by Lemma~\ref{lem:extension}; we then have that
\begin{equation*}\Tr_{m,\semiH}^- H^j=\arr f^j,
\qquad \sup_{t<0}\doublebar{\nabla^m H^j(\,\cdot\,,t)}_{L^2(\R^n)}
\leq C \doublebar{\arr f^j}_{L^2(\R^n)}.\end{equation*}
We may assume without loss of generality that $\Trace F^j\equiv 0$ outside of $2^{j+2}Q$.
Let $\eta_j(x,t)$ be smooth and satisfy the bound $\abs{\nabla^i\eta_j(x,t)}\leq C_i (2^{j}\ell(Q))^{-i}$, with
\begin{align*}
\eta_j(x,t)&= 1
	\text{ if }x\in 2^{j+2}Q\text{ and }-2^j\ell(Q)<t<2^j\ell(Q),
\\
\eta_j(x,t)&= 0 \text{ if }x\notin 2^{j+3}Q \text{ or } \abs{t}>2^{j+1}\ell(Q).
\end{align*}
Then $\Tr_{m,\semiH}^-(\eta_j H^j)= \Tr_{m,\semiH}^- H^j =\arr f^j$. We take $F^j=\eta_j\,H^j$. Observe that if $j\geq 2$, then $\nabla^m F^j(x,t) = \nabla^m H^j(x,t) =0$ if $\abs{t}<\dist(x,\R^n\setminus 2^{j-1}Q)$. Furthermore, we still have the bound
\begin{equation*}\sup_{t<0}\doublebar{\nabla^m F^j(\,\cdot\,,t)}_{L^2(\R^n)}
\leq C \doublebar{\arr f^j}_{L^2(\R^n)}.\end{equation*}

Now, by the definition~\eqref{dfn:Theta:D} of~$\Theta_t^D$ and by formulas~\eqref{dfn:D:tilde}, and~\eqref{eqn:D:fundamental}
\begin{equation*}\int_Q \abs{\Theta_t^D \arr f^j}^2
=
t^{2k}\int_Q \abs[bigg]{\sum_{\abs{\beta}=m} \int_{\R^\dmn_-} \partial_t^{m+k} \partial_{y,s}^\beta E^L(x,t,y,s)\,(\mat A\nabla^m F^j)_\beta (y,s)\,ds\,dy}^2\,dx
.\end{equation*}

Applying the bound~\eqref{eqn:fundamental:slices}, we see that
\begin{align*}
%\label{eqn:D:decay}
\int_Q \abs{\Theta_t^D \arr f^j}^2
&\leq
	C 2^{-j(2k+\varepsilon)}
	\doublebar{\arr f^j}_{L^2(\R^n)}^2
\end{align*}
for all $\arr f^j\in \mathfrak{D}$ supported in $A_{j,1}(Q)$; by density we may extend to all $\arr f \in \dot W\!A^2_{m,\semiH}(\R^n)$ supported in $A_{j,1}(Q)$.

\subsection{Extending \texorpdfstring{$\Theta_t^D$}{} to all of \texorpdfstring{$L^2(\R^n)$}{L2}}
\label{sec:extension}

We now must extend $\Theta_t^D$ to an operator defined on all of $L^2(\R^n)$ that still satisfies the estimate~\eqref{eqn:Theta:decay}. Essentially, this argument consists of defining a projection operator from $L^2(\R^n)$ to $\dot W\!A^2_{m,\semiH}(\R^n)$, the space on which $\Theta_t^D$ naturally acts.

Because $L^2(\R^n)$ is a Hilbert space, there is an orthogonal projection operator $O_W:L^2(\R^n) \mapsto \dot W\!A^2_{m,\semiH}(\R^n)$. For example, if $m=1$ then $O_W\vec f= \nabla_\pureH u$, where $\Delta_\pureH u= \nabla_\pureH \cdot \vec f$. This is the most natural mapping from $L^2(\R^n)$ to $\dot W\!A^2_{m,\semiH}(\R^n)$; however, this mapping does not satisfy adequate decay estimates. Thus, we must refine this mapping by applying cutoffs before and after projecting.

Let $W_j$ be the closure in $L^2(\R^n)$ of
\begin{equation*}\{\1_{2^{j+2}Q} \Tr_{m-1}\varphi + (1-\1_{2^{j+2}Q}) \arr f: \varphi \in C^\infty_0,\>\arr f\in L^2(\R^n)\}.\end{equation*}
Loosely, elements of $W_j$ are higher-order traces in the cube $2^{j+2}Q$ and are merely arbitrary $L^2$ arrays outside of that cube. Let $O_j$ denote orthogonal projection from $L^2(\R^n)$ onto the subspace $W_j$; observe that $O_j \arr f=\arr f$ outside of $2^{j+2}Q$. Furthermore, if $\varphi$ is a nice function then $O_j (\Tr_{m-1}\varphi)=\Tr_{m-1}\varphi$.

Let $\eta_j$ be a smooth partition of unity; that is, $\sum_j\eta_j(x,t)=1$ for $t$ near zero, with $\eta_j$ supported in $A_{j,1}(Q)\times (-2^j\ell(Q),2^j\ell(Q))$ and satisfying $\abs{\nabla^i\eta_j(x,t)}\leq C 2^{-ij}\ell(Q)^{-i}$ for all $(x,t)\in\R^\dmn$.

Define $\pi_j:W_j \mapsto \dot W\!A^2_{m,\semiH}(\R^n)$ as follows. Suppose that $\arr f=\Tr_{m,\semiH}\varphi$ in $2^{j+2}Q$ for some smooth function~$\varphi$.
We may renormalize $\varphi$ so that $\int_{Q} \Trace\partial^\zeta\varphi=0$ for all $\abs{\zeta}\leq m-1$. Let $\pi_j\arr f=\Tr_{m,\semiH}(\eta_j\varphi)$. We remark that $\pi_j\arr f$ is well-defined, that is, $\Tr_{m,\semiH}(\eta_j\varphi)$ depends only on $\Tr_{m,\semiH}\varphi$. Furthermore, observe that $\pi_j\arr f$ is supported in $A_{j,1}(Q)$. Finally, by the Poincar\'e inequality
\begin{equation*}\doublebar{\pi_j\arr f}_{L^2(A_{j,1}(Q))}
\leq C 2^{jn/2}\doublebar{\Tr_{m,\semiH}\varphi}_{L^2(2^{j+2}Q)}
= C 2^{jn/2}\doublebar{\arr f}_{L^2(2^{j+2}Q)}.\end{equation*}

We will extend to an operator on all of $L^2(\R^n)$ using the orthogonal projection operators $O_j$. Observe, first, that $\pi_j O_j\arr f=0$ outside of $A_{j,1}(Q)$, and second, that if $\arr f=0$ in $2^{j+2}Q$ then $O_j\arr f=\arr f$ and so $\pi_jO_j\arr f=0$.

If $\varphi$ is smooth and compactly supported, and renormalized as above, then
\begin{equation*}\Tr_{m,\semiH}\varphi
= \sum_{j=0}^\infty \Tr_{m,\semiH}(\eta_j\varphi)
= \sum_{j=0}^\infty \pi_j (\Tr_{m-1}\varphi)
= \sum_{j=0}^\infty \pi_j (O_j(\Tr_{m-1}\varphi))
.\end{equation*}

We define $\Theta_t^D \arr f = \sum_{j=0}^\infty\Theta_t^D (\pi_j O_j\arr f)$. Now, if $\arr f^j$ is supported in $A_j(Q)$, observe that $\pi_i O_i\arr f=0$ for all $i\leq j-2$, and furthermore that $\doublebar{\pi_i O_i\arr f^j}_{L^2(\R^n)}\leq C2^{in/2}\doublebar{\arr f^j}_{L^2(A_j(Q))}$ for all $i\geq j-1$. Then
\begin{align*}\doublebar{\Theta_t^D \arr f^j}_{L^2(Q)}
&\leq
	\sum_{i=j-1}^\infty \doublebar{\Theta_t^D (\pi_i O_i\arr f^j)}_{L^2(Q)}
\leq
	C \sum_{i=j-1}^\infty 2^{-i(k+\varepsilon/2)} \doublebar{\pi_i O_i\arr f^j}_{L^2(A_{i,1}(Q))}
\\&\leq
	C \sum_{i=j-1}^\infty 2^{-i(k+\varepsilon/2)} 2^{in/2}\doublebar{\arr f^j}_{L^2(A_j(Q))}\end{align*}
and so $\Theta_t^D$ satisfies the decay estimate~\eqref{eqn:Theta:decay} provided $k$ is large enough.

\section{The quasi-orthogonality estimate \myeqref{eqn:Theta:Q}}
\label{sec:Q}

A family of operators $\{Q_s\}_{s>0}$ is defined to be a Calder\'on-Littlewood-Paley family, or CLP family, if
\begin{equation*}Q_s f(x) = \int_{\R^n} s^{-n}\,\varphi(y/s) \,f(x-y)\,dy\end{equation*}
for some $\varphi\in L^1(\R^n)$ that satisfies the conditions
\begin{equation}\label{eqn:Q:kernel}
\abs{\widehat\varphi(\xi)}\leq C\min(\abs{\xi}^\sigma,\abs{\xi}^{-\sigma}),
\qquad \abs{\varphi(x)}\leq C(1+\abs{x})^{-n-\sigma}\end{equation}
for some $\sigma>0$, where $\widehat\varphi$ denotes the Fourier transform of~$\varphi$,
and such that $Q$ satisfies the conditions
\begin{align}
\label{eqn:Q:bounds}
\doublebar{Q_sf}_{L^2(\R^n)}+\doublebar{s\nabla Q_sf}_{L^2(\R^n)}&\leq C \doublebar{f}_{L^2(\R^n)}\quad\text{for all }s>0
,\\ \label{eqn:Q:square}
\int_{\R^n}\int_0^\infty \abs{Q_s f(x)}^2 \frac{ds\,dx}{s}&\leq C\doublebar{f}_{L^2(\R^n)}^2,
\\ \label{eqn:Q:identity}
\int_0^\infty Q_s^2\frac{ds}{s}=I
\end{align}
where convergence to the identity in the last formula is in the strong operator topology on $\mathcal{B}(L^2(\R^n))$.

We remark that, up to multiplying $\varphi$ by a constant, it suffices to require that $\varphi$ satisfy the bounds \ref{eqn:Q:kernel}, the bound $\abs{\widehat\varphi(\xi)}\leq C\abs{\xi}^{-1}$, be radial, and be real-valued and not identically zero.

Specifically, notice that
\begin{equation*}\widehat{Q_s f}(\xi) = \widehat\varphi(s\,\xi)\,\widehat f(\xi).\end{equation*}
Given this relation, the estimates \eqref{eqn:Q:bounds} and~\eqref{eqn:Q:square} follow from the estimate $\abs{\widehat\varphi(\xi)}\leq C\min(\abs{\xi}^\sigma,\abs{\xi}^{-1})$ by Plancherel's theorem. To establish the identity~\eqref{eqn:Q:identity}, we normalize $\varphi$ as follows. Because $\varphi$ is real-valued, so is~$\widehat\varphi$. Thus, the integral
\begin{equation*}\int_0^\infty \widehat \varphi(s\,\xi)^2\, \frac{1}{s}\,ds\end{equation*}
is independent of~$\xi$ provided $\xi\neq 0$. If $\varphi$ is both radial and real-valued, then $\widehat\varphi$ is real-valued, and so this integral is positive. We may normalize $\varphi$ so that this integral equals~$1$.
It is then straightforward to establish the condition~\eqref{eqn:Q:identity}.

In this section, let $\widehat\varphi$ be bounded, radial and supported in $B(0,2)\setminus B(0,1/2)$. (In Section~\ref{sec:kato} we will use a CLP family again; in that section it will be more convenient to take~$\varphi$, rather than~$\widehat\varphi$, compactly supported.)
We wish to establish the bound~\eqref{eqn:Theta:Q}, for the operators $\Theta_t=\Theta_t^S$ or $\Theta_t=\Theta_t^D$.
We proceed as in \cite{GraH14p}.

We begin with $\Theta_t^S$. Fix some $\arr h\in L^2(\R^n)$; we wish to bound $\Theta_t^S Q_s\arr h$.
For each $1\leq j\leq n$, let $f^j_\gamma$ satisfy
\begin{equation*}\widehat {f^j_\gamma}(\xi) = \frac{\xi_j}{2\pi i\abs{\xi}^2} \widehat{Q_s h_{\gamma}}(\xi) = \frac{\xi_j}{2\pi i\abs{\xi}^2} \, \widehat{\varphi}(s\,\xi)\,\widehat h_{\gamma}(\xi).\end{equation*}
Then $Q_s h_{\gamma} = \sum_{j=1}^n \partial_j f^j_\gamma$, and so
\begin{equation*}\Theta_t^S Q_s\arr h(x)=\sum_{j=1}^n \Theta_t^S (\partial_j \arr f^j)(x).\end{equation*}
By the bound~\eqref{eqn:S:div},
\begin{equation*}\doublebar{\Theta_t^S Q_s\arr h}_{L^2(\R^n)}
\leq \sum_{j=1}^n \doublebar{\Theta_t^S (\partial_j \arr f^j)}_{L^2(\R^n)}
\leq C\frac{1}{t}\sum_{j=1}^n \doublebar{\arr f^j}_{L^2(\R^n)}
.\end{equation*}
Notice that $\abs{\widehat\varphi(\xi)}\leq C\abs{\xi}$, and so $\doublebar{\arr f}_{L^2(\R^n)}\leq C s \doublebar{\arr h}_{L^2(\R^n)}$. Thus,
\begin{equation*}\doublebar{\Theta_t^S Q_s\arr h}_{L^2(\R^n)}
\leq C\frac{s}{t}\doublebar{\arr h}_{L^2(\R^n)}.\end{equation*}
Therefore, $\Theta_t^S$ satisfies the bound~\eqref{eqn:Theta:Q} for $\theta=1$ (and thus for any $\theta\leq 1$).

We now consider $\Theta_t^D$.
Let the subspace $H$ be $\dot W\!A^2_{m,\semiH}(\R^n)$; recall that this is the natural space upon which $\widetilde \D^{\mat A}$ and $\Theta_t^D$ act. It suffices to show that
\begin{equation*}\doublebar{\Theta_t^D Q_s \Tr_{m,\semiH}\varphi}_{L^2(\R^n)}\leq C\biggl(\frac{s}{t}\biggr)^\theta \doublebar{\Tr_{m,\semiH}\varphi}_{L^2(\R^n)}\end{equation*}
for some $\theta>0$ and for all smooth, compactly supported functions~$\varphi$.
To establish this bound, we begin with the following lemma.
\begin{lem}\label{lem:Q:Whitney} If $\varphi$ is smooth and compactly supported, then
\begin{equation*}Q_s(\Tr_{m,\semiH}\varphi) = \Tr_{m,\semiH} \Phi_s\end{equation*}
for some function $\Phi_s$ that satisfies
\begin{equation*}\doublebar{\nabla^m\Phi_s}_{L^2(\R^\dmn_-)}\leq C \sqrt{s} \doublebar{\Tr_{m,\semiH}\varphi}_{L^2(\R^n)}.\end{equation*}
\end{lem}

\begin{proof}
For each $j$ with $0\leq j\leq m-1$, let $f^j(x)=\partial_\dmn^j \varphi(x,0)$. Observe that if $\beta$ is a multiindex and $\beta_\dmn<\abs{\beta}=m$, then
\begin{equation*}(\Tr_{m,\semiH}\varphi)_\beta
= \partial_x^{\beta_\pureH}f^{\beta_\dmn}(x,0).
\end{equation*}
$Q_s$ is a convolution operator and so commutes with horizontal derivatives, and thus
\begin{equation*}(Q_s\Tr_{m,\semiH}\varphi)_\beta
= \partial_x^{\beta_\pureH} Q_s f^{\beta_\dmn}(x,0).\end{equation*}
Let $g^j_s = Q_s f^j$. Then $\widehat{g}^j_s(\xi)=\widehat\varphi(s\,\xi)\,\widehat{f}^j(\xi)$ and so $\abs{\xi}^{-1}\,\abs{\widehat{g}^j_s(\xi)}^2\leq Cs\abs{\widehat{f}^j(\xi)}^2$. Thus,
\begin{equation*}\int_{\R^n} \abs{\widehat{g}^j_s(\xi)}^2\abs{\xi}^{2m-2j-1}\,d\xi
\leq Cs\int_{\R^n} \abs{\widehat{f}^j(\xi)}^2\abs{\xi}^{2m-2j}\,d\xi
=Cs\int_{\R^n} \abs{\nabla_\pureH^{m-j}{f}^j(x)}^2\,dx
.\end{equation*}
Because $\abs{\nabla_\pureH^{m-j} f^j}\leq \abs{\Tr_{m,\semiH} \varphi}$, we have that
\begin{equation*}\int_{\R^n} \abs{\widehat{g}^j_s(\xi)}^2\abs{\xi}^{2m-2j-1}\,d\xi
\leq Cs\doublebar{\Tr_{m,\semiH} \varphi}_{L^2(\R^n)}^2
.\end{equation*}
Let $\arr g_s=Q_s (\Tr_{m-1}\varphi)$; then $\arr g_s$ satisfies $(\arr g_s)_\gamma = \partial^{\gamma_\pureH} g^{\gamma_\dmn}_s$ for each $\abs{\gamma}=m-1$. We have established that
\begin{equation*}\int_{\R^n} \abs{\widehat{\arr g}_s(\xi)}^2\abs{\xi}\,d\xi
\leq Cs\doublebar{\Tr_{m,\semiH} \varphi}_{L^2(\R^n)}^2
.\end{equation*}
Extending $\arr g_s$ using Lemma~\ref{lem:extension} completes the proof. Notice that more general extension theorems, also appropriate in this case, are well known; see, for example, \cite[Theorem~2.9.3]{Tri78} or \cite[Theorem~2.7.2]{Tri83}.
\end{proof}

The estimate~\eqref{eqn:Theta:Q} for $\Theta_t^D$ follows quickly from Lemma~\ref{lem:Q:Whitney}.
By the definition~\eqref{dfn:Theta:D},
\begin{equation*}\Theta_t^D(Q_s \Tr_{m,\semiH}\varphi)(x)
=t^k \partial_t^{m+k} \D^{\mat A} (\Tr_{m-1}\Phi_s)(x,t).
\end{equation*}
But by the definition~\eqref{dfn:D:newton:exterior},
\begin{equation*}\Theta_t^D(Q_s \Tr_{m,\semiH}\varphi)(x)
=-t^k \partial_t^{m+k} \Pi^L (\1_-A\nabla^m\Phi_s)(x,t).
\end{equation*}
But $u=\Pi^L (\1_-A\nabla^m\Phi_s)$ satisfies $Lu=0$ in $\R^\dmn_+$; furthermore, by the bound~\eqref{eqn:newton:bound}, $\doublebar{\nabla^m u}_{L^2(\R^\dmn_+)}\leq C\doublebar{\nabla^m\Phi_s}_{L^2(\R^\dmn_-)}\leq C \sqrt{s} \doublebar{\Tr_{m,\semiH}\varphi}_{L^2(\R^n)}$. Applying the Caccioppoli inequality and Lemma~\ref{lem:slices} in small cubes of sidelength $t/C$ suffices to establish that $\Theta_t^D$ satisfies the bound~\eqref{eqn:Theta:Q} for $\theta=1/2$.

\section{The semi-horizontal single layer potential}
\label{sec:S:not-normal}

In this section we will prove the following theorem.

\begin{thm}\label{thm:S:horizontal}
Suppose that
$\gamma_{n+1}<\abs{\gamma}=m-1$.
Then we have the square-function estimate
\begin{align*}
\int_{\R^{n+1}_+}
	\abs{\Theta_t^{S} (g\arr e_{\gamma})(x)}^2\,\frac{1}{t}\,dx\,dt
&\leq C\doublebar{g}_{L^2(\R^n)}^2.
\end{align*}
\end{thm}

\begin{proof} Let
$\Theta_t g=\Theta_t^S(g\arr e_{\gamma})$. We want to apply Theorem~\ref{thm:grau:hofmann:1}. As shown in Sections~\ref{sec:decay} and~\ref{sec:Q}, the bounds \eqref{eqn:Theta:decay} and~\eqref{eqn:Theta:Q} are valid for this choice of~$\Theta_t$. We are left with the estimate~\eqref{eqn:Theta:1:carleson}.

Recall that by formula~\eqref{dfn:Theta:S:integral},
\begin{align*}\Theta_t g(x) = t^k  \int_{\R^n} \partial_t^{m+k} \partial_{y,s}^\gamma E^L(x,t,y,0)\,g(y)\,dy
.\end{align*}
In particular,
\begin{align*}\Theta_t 1(x) = t^k  \int_{\R^n} \partial_t^{m+k} \partial_{y,s}^\gamma E^L(x,t,y,0)\,dy
.\end{align*}
Let $j$ satisfy $1\leq j\leq n$ and $\gamma_j>0$; by assumption on $\gamma$ such a $j$ exists. Let $\zeta=\gamma-\vec e_j + \vec e_\dmn$. By formula~\eqref{eqn:fundamental:vertical}, we have that
\begin{align*}\Theta_t 1(x) = -t^k
\int_{\R^n} \partial_{y_j} \bigl(\partial_t^{m+k-1} \partial_{y,s}^\zeta E^L(x,t,y,0)\bigr)\,dy
.\end{align*}
By the bound~\eqref{eqn:fundamental:slices}, for almost every $(x,t)\in\R^\dmn$, if $k$ is large enough then $v(y)=\partial_t^{m+k-1} \partial_{y,s}^\zeta E^L(x,t,y,0)$ lies in both $L^1(\R^n)$ and in~$\dot W^1_1(\R^n)$. Thus,
\begin{equation*}\int_{\R^n}\partial_{y_j} \bigl(\partial_t^{m+k-1} \partial_{y,s}^\zeta E^L(x,t,y,0)\bigr)\,dy=0\end{equation*}
for almost every $(x,t)\in\R^\dmn$. Thus, $\Theta_t 1=0$, and in particular the bound~\eqref{eqn:Theta:1:carleson} is valid.
\end{proof}

We are now left with the double layer potential $\Theta_t^D$ and the vertical single layer potential $\Theta_t^\perp f=\Theta_t^S (f\arr e_{\perp})$. In the following sections we will use the full force of Theorem~\ref{thm:grau:hofmann} to bound $\Theta_t^\perp$; we will bound $\Theta_t^D$ as a component of the auxiliary operator $\Theta_t'$ in Theorem~\ref{thm:grau:hofmann}.

\section{The Carleson estimate \myeqref{eqn:Theta:T1}} \label{sec:D:carleson}

We will let $\Theta_t=(\Theta_t^D,\Theta_t^S)$, with $\Theta_t^\perp$ denoting the purely vertical component of $\Theta_t^S$ (that is, $\Theta_t^\perp f = \Theta_t^S(f\arr e_\perp)$). The bounds \eqref{eqn:Theta:decay} and~\eqref{eqn:Theta:Q} are valid.
We wish to show that the bound~\eqref{eqn:Theta:T1} is also valid. Recall that this bound is given by
\begin{align*}
\doublebar{\Theta'_t 1}_{\mathcal{C},\delta}
&\leq C_1 + C_1\doublebar{\Theta^\perp_t 1}_{\mathcal{C},\delta}
&&\text{if $\delta>0$ small enough}.
\end{align*}
In Section~\ref{sec:S:not-normal}, we showed that $\Theta_t^S(1\arr e_{\gamma})=0$ whenever $\gamma_\dmn<\abs{\gamma}=m-1$; the bound~\eqref{eqn:Theta:T1}, with $\Theta_t'=\Theta_t^{S'}$,
follows immediately. Thus, we need only bound $\Theta_t^D 1$ by $\Theta_t^\perp 1$ and a constant.

Recall that by formulas~\eqref{dfn:Theta:D}, \eqref{dfn:D:tilde} and~\eqref{eqn:D:fundamental}, if $\arr f=\Tr_{m,\semiH}F$, then
\begin{align*}
\Theta_t^D \arr f(x)
&=
	-\sum_{\abs{\alpha}=\abs{\beta}=m} t^k\int_{\R^\dmn_-} \partial_t^{m+k} \partial_{y,s}^\alpha E^L(x,t,y,s)\,A_{\alpha\beta}(y,s) \, \partial^\beta F(y,s)\,ds\,dy.
\end{align*}
Using the bound~\eqref{eqn:fundamental:slices}, we see that if $\nabla^m F$ is bounded then the integral converges absolutely for almost every $(x,t)\in\R^\dmn_+$.

Recall that $\Theta_t^D$ acts on arrays of functions of the form $\arr\varphi=\Tr_{m,\semiH}\varphi$; these arrays $\arr\varphi$ are indexed by multiindices $\beta$ with $\beta_\dmn<\abs{\beta}=m$. Fix some such~$\beta$.
By choosing
\begin{equation*}
 F(x,t) = \frac{1}{\beta!}(x,t)^\beta
\end{equation*}
we see that
\begin{align}\label{eqn:D1}
\Theta_t^D \arr e_{\beta}(x)
&=
	-\sum_{\abs{\alpha}=m} t^k\int_{\R^\dmn_-} \partial_t^{m+k} \partial_{y,s}^\alpha E^L(x,t,y,s)\,A_{\alpha\beta}(y) \,ds\,dy.
\end{align}
We use formula~\eqref{eqn:fundamental:vertical} to convert one of the derivatives in $t$ into a derivative in~$s$; we evaluate the integral $ds$ to see that
\begin{align*}
\Theta_t^D \arr e_{\beta}(x)
&=
	\sum_{\abs{\alpha}=m}  t^k\int_{\R^\dmnMinusOne} \partial_t^{m+k-1} \partial_{y,s}^\alpha E^L(x,t,y,0)\,A_{\alpha\beta}(y) \,dy.
\end{align*}
Observe that this is a sum of terms depending on $\alpha$ and~$\beta$. In Section~\ref{sec:kato:mixed:alpha} we will bound the terms for which $\alpha_\dmn>0$. In Sections~\ref{sec:kato}--\ref{sec:kato:3}, we will bound the terms for which $\alpha_\dmn=\beta_\dmn=0$; this case is the most involved, and will closely parallel the argument in \cite[Section~3.1]{GraH14p}. Finally, in Section~\ref{sec:kato:mixed} we will bound the remaining terms, that is, the terms for which $\alpha_\dmn=0$ and $\beta_\dmn>0$; this bound will rely on the bound in the case $\alpha_\dmn=\beta_\dmn=0$.

\subsection{Terms with \texorpdfstring{$\alpha_\dmn>0$}{partly vertical derivatives}}
\label{sec:kato:mixed:alpha}

Observe that if $\alpha_\dmn>0$, then $\alpha=\gamma+\vec e_\dmn$ for a unique $\gamma$ with $\abs{\gamma}=m-1$. For any such~$\gamma$, let $\widetilde\gamma=\gamma+\vec e_\dmn$. Then by formula~\eqref{eqn:fundamental:vertical},
\begin{align*}
\sum_{\substack{\abs{\alpha}=m\\\alpha_\dmn>0}} \partial^\alpha_{y,s} \partial_t^{m+k-1} E^L(x,t,y,0)\,A_{\alpha\beta}(y)
&=
	-\sum_{\abs{\gamma}=m-1} \partial^\gamma_{y,s} \partial_t^{m+k} E^L(x,t,y,0)\,A_{\widetilde\gamma\beta}(y)
.\end{align*}
Thus,
\begin{align*}
\Theta_t^D \arr e_{\beta}(x)
&=
	\sum_{{\alpha_\dmn}=0} t^k\int_{\R^\dmnMinusOne} \partial_t^{m+k-1} \partial_{y,s}^\alpha E^L(x,t,y,0)\,A_{\alpha\beta}(y) \,dy
	\\&\qquad
	-\sum_{\abs{\gamma}=m-1} \llap{$t^k$}\int_{\R^\dmnMinusOne} \partial_t^{m+k} \partial_{y,s}^\gamma E^L(x,t,y,0)\,A_{\widetilde\gamma\beta}(y) \,dy
	.
\end{align*}
In this section we will bound the second sum; we will consider the first sum in Sections~\ref{sec:kato}--\ref{sec:kato:mixed}.

By formula~\eqref{dfn:Theta:S:integral}, the second sum is equal to $\Theta_t^S \arr a_\beta$, where $(a_\beta)_{\gamma}=A_{\widetilde\gamma,\beta}$. Notice that $\arr a_\beta$ is bounded.
We may control $\Theta_t^S \arr a_\beta$ using a standard technique in the study of $T1$ theorems. Let
\begin{equation*}P_t f(x)=\int_{\R^n} \frac{1}{t^n} \psi\biggl(\frac{x-y}{t}\biggr)\,f(y)\,dy\end{equation*}
where $\psi$ is smooth, nonnegative and satisfies $\int_{\R^n}\psi=1$. We do not require that $\psi$ be compactly supported. Fix some $\gamma$ with $\abs{\gamma}=m-1$, and let
\begin{equation*}\Psi_t a(x)
=\Theta_t^S( a\,\arr e_{\gamma})(x)-P_t a(x)\,\Theta_t^S \arr e_{\gamma}(x).\end{equation*}
Then
\begin{equation*}\doublebar{P_t (a_\beta)_\gamma\,\Theta_t^S \arr e_{\gamma}}_{\mathcal{C},\delta}
\leq \doublebar{\arr a}_{L^\infty} \doublebar{\Theta_t^S \arr e_{\gamma}}_{\mathcal{C},\delta}
.\end{equation*}
Either $\gamma=\gamma_\perp$ and so $\Theta_t^S \arr e_\gamma=\Theta_t^\perp 1$, or $\gamma_\dmn<\abs{\gamma}$ and so $\Theta_t^S \arr e_{\gamma}=0$. (See the proof of Theorem~\ref{thm:S:horizontal}.) So to bound $\doublebar{\Theta_t^S\arr a_\beta}_{\mathcal{C},\delta}$, we need only control~$\Psi_t a(x)$ for arbitrary bounded functions~$a$.

We wish to show that $\doublebar{\Psi_t a}_{\mathcal{C}}=\doublebar{\Psi_t a}_{\mathcal{C},0}\leq C\doublebar{a}_{L^\infty}$. (We will not need the $\doublebar{\Theta_t^\perp 1}_{\mathcal{C},\delta}$ term on the right-hand side; thus, we will simply bound the left-hand side for $\delta=0$ and observe that this gives our desired bound for any positive~$\delta$.)
We will do this by applying Theorem~\ref{thm:grau:hofmann:1} to~$\Psi_t$.
Observe that $\Psi_t 1(x)=0$, and so the bound~\eqref{eqn:Theta:1:carleson} is valid. We need only verify the bounds \eqref{eqn:Theta:decay} and~\eqref{eqn:Theta:Q} for the operator~$\Psi_t$; because these bounds have been verified for~$\Theta_t^S$, we need only consider $\Upsilon_t a(x)=P_t a(x)\,\Theta_t^S \arr e_{\gamma}(x)$.

We begin with the bound~\eqref{eqn:Theta:Q}. Observe that
\begin{equation*}\widehat{P_t Q_s h}(\xi) =\widehat\psi(t\,\xi)\,\widehat\varphi(s\,\xi) \,\widehat h(\xi)\end{equation*}
where $Q_s$ is the operator defined in Section~\ref{sec:Q}.
Recall that $\widehat\varphi$ is supported in $B(0,2)\setminus B(0,1/2)$.
If we require that $\widehat\psi$ be smooth and supported in $B(0,1/2)$, then $P_t Q_s h=0$ whenever $s\leq t$; thus $\Upsilon_t$ satisfies the bound~\eqref{eqn:Theta:Q}.

We now establish the bound~\eqref{eqn:Theta:decay}.
By Section~\ref{sec:decay}, we know that the operator $\Theta_t^S$ satisfies the decay estimate~\eqref{eqn:Theta:decay}. From this we may verify that, if $Q\subset\R^n$ is a cube with $\ell(Q)\leq t<2\ell(Q)$, then
\begin{equation*}\doublebar{\Theta_t^S\arr e_{\gamma}}_{L^2(Q)}\leq C \abs{Q}^{1/2}.\end{equation*}
If $\widehat\psi$ is smooth as well as being supported in $B(0,1/2)$, then $\psi$ is a Schwartz function and satisfies the estimate $\abs{\psi(y)}\leq C_N (1+\abs{y})^{-N}$ for any $N>0$. If $g_j$ is supported in $A_j(Q)$, and $\ell(Q)\leq t<2\ell(Q)$, then
\begin{equation*}\sup_{x\in Q}\abs{P_t g_j(x)}
\leq C_N t^{-n}2^{-jN}\doublebar{g_j}_{L^1(A_j(Q))}
\leq C_N t^{-n/2}2^{-j(N-n/2)}\doublebar{g_j}_{L^2(A_j(Q))}
\end{equation*}
and so if we choose $N$ large enough, then the operator $\Upsilon_t$ satisfies the bound~\eqref{eqn:Theta:decay}.

Thus, by Theorem~\ref{thm:grau:hofmann:1}, we have that the operators $\Psi_t$ satisfy the square-function estimate~\eqref{eqn:Theta:1:square}. To show that $\doublebar{\Psi_t a}_{\mathcal{C}}\leq C\doublebar{a}_{L^\infty(\R^n)}$, we need only show that the estimate~\eqref{eqn:Theta:1:square} for $L^2$ test functions implies an estimate for $L^\infty$ test functions.

\begin{lem}\label{lem:square:carleson} Suppose that the operators $\Psi_t$ satisfy the square-function estimate
\begin{equation*}\int_0^\infty\int_{\R^n} \abs{\Psi_t f(x)}^2\frac{dx\,dt}{t}
\leq C_0\doublebar{f}_{L^2(\R^n)}^2\end{equation*}
for all $f\in L^2(\R^n\mapsto\C)$, and that for some $\theta>-2$, $\Psi_t$ satisfies the off-diagonal decay estimate
\begin{align*}
\doublebar{\Psi_t  g_j}_{L^2(Q)}
&\leq C_1 2^{-j(n+2+\theta)/2} \doublebar{ g_j}_{L^2}
&&\text{if }\ell(Q)\leq t\leq 2\ell(Q)
\end{align*}
for all $j\geq 1$ and all $g_j$ supported in $A_j(Q)=2^{j+1}Q\setminus 2^jQ$.

Then there is some $C$ depending only on $C_0$, $C_1$ and~$\theta$ such that $\Psi_t$ satisfies the Carleson condition
\begin{equation*}\sup_Q \frac{1}{\abs{Q}}\int_Q \int_0^{\ell(Q)} \abs{\Psi_t b(x)}^2 \frac{dt\,dx}{t} \leq C \doublebar{b}_{L^\infty(\R^n)}^2\end{equation*}
for all bounded functions~$b$.
\end{lem}

\begin{proof} Choose some cube~$Q$ and some bounded function~$b$.
Let $b_j=b\1_{A_j(Q)}$; recall that $b_0=b\1_{A_0(Q)}=b\1_{2Q}$.
Then by the square-function estimate,
\begin{equation*}\biggl(\int_Q \int_0^{\ell(Q)} \abs{\Psi_t b_0(x)}^2 \frac{dt\,dx}{t}\biggr)^{1/2} \leq C \doublebar{b_0}_{L^2(2Q)}\leq C \doublebar{b}_{L^\infty(\R^n)}\abs{Q}^{1/2}.\end{equation*}

Furthermore, if $j\geq 1$, then by the decay estimate applied in cubes $R\subset Q$ with side-length $t\leq\ell(R)<2t$,
\begin{equation*}\biggl(\int_0^{\ell(Q)} \int_Q \abs{\Psi_t b_j(x)}^2\,dx\, \frac{dt}{t}\biggr)^{1/2}
\leq C \doublebar{b}_{L^\infty(\R^n)}\abs{Q}^{1/2} 2^{-j(2+\theta)/2}.\end{equation*}

Summing in $j$, we see that
\begin{equation*}\biggl(\int_Q \int_0^{\ell(Q)} \abs{\Psi_t b(x)}^2 \frac{dt\,dx}{t}\biggr)^{1/2} \leq C \doublebar{b}_{L^\infty(\R^n)}\abs{Q}^{1/2}\end{equation*}
as desired.
\end{proof}

\subsection{Terms with \texorpdfstring{$\alpha_\dmn=\beta_\dmn=0$}{purely horizontal derivatives}: preliminaries}
\label{sec:kato}

We have now established that
\begin{align*}
\Theta_t^D \arr e_{\beta}(x)
&=
	\sum_{{\alpha_\dmn}=0} t^k\int_{\R^\dmnMinusOne} \partial_t^{m+k-1} \partial_{y,s}^\alpha E^L(x,t,y,0)\,A_{\alpha\beta}(y) \,dy
	-\Theta_t^S \arr a_\beta(x)
\end{align*}
and that $\Theta_t^S\arr a_\beta$ satisfies the desired estimate.
In order to show that the remaining sum satisfies the estimate~\eqref{eqn:Theta:T1}, we need only show that the operator $\Theta_t^\beta $, defined as
\begin{equation} \label{eqn:Theta:Carleson}
\Theta_t^\beta f(x) =
	\sum_{\alpha_{n+1}=0}  t^k\int_{\R^\dmnMinusOne}
	\partial_t^{m+k-1} \partial_{y}^\alpha E^L(x,t,y,0)\,A_{\alpha\beta}(y) \,f(y)\,dy,
\end{equation}
satisfies the Carleson-measure estimate
\begin{equation}
\label{eqn:kato}
\doublebar{\Theta_t^\beta 1}_{\mathcal{C},\delta} \leq C_1 + C_1 \doublebar{\Theta_t^\perp 1}_{\mathcal{C},\delta}
\end{equation}
for any multiindex $\beta$ with $\beta_\dmn<\abs{\beta}=m$.

We will make use of the horizontal operator $L_\pureH$, defined as follows.
Recall that $L$ is an operator acting on $\dot W^2_{m,loc}(\R^\dmn)$-functions. We may (formally) define the operator $L_\pureH$, acting on $\dot W^2_{m,loc}(\R^n)$-functions, as
\begin{equation}\label{eqn:L:parallel}
L_\pureH f = (-1)^m \sum_{\substack{\abs{\alpha}=\abs{\beta}=m\\\alpha_\dmn=\beta_\dmn=0}} \partial^\alpha (A_{\alpha\beta} \partial^\beta \! f)
\end{equation}
where $\partial^\alpha$, $\partial^\beta$ are understood to be derivatives in the $n$ horizontal directions. (The operator $L_\pureH$ has a weak formulation, as in formula~\eqref{eqn:L:integral}.)

To establish the bound~\eqref{eqn:kato} for $\beta_\dmn=0$, we will follow the argument of \cite[Section~3.1]{GraH14p}. We remark that the argument we will make in this section is valid only in the case where the order $2m$ of $L$ (and thus $L_\pureH$) satisfies the inequality $2m>n$. Thus, the argument of Sections~\ref{sec:D:carleson} and~\ref{sec:b} will only establish boundedness of $\Theta_t^D$ and $\Theta_t^\perp$ in the case of operators of very high order. In Section~\ref{sec:high} we will show that bounds on $\Theta_t^D$ and $\Theta_t^S$, for operators of high order, imply the corresponding bounds for operators of lower order, completing the proof of Theorem~\ref{thm:square}.

We will use some tools from the proof of the Kato conjecture, in particular from the paper \cite{AusHMT01}. The following lemma was established therein.

\begin{lem} Suppose that $2m\geq n$. There is some $W$ depending only on the standard constants such that, for each cube~$Q\subset\R^n$, there exist $W$ functions $f_{Q,w}$
that satisfy the estimates
\begin{align}
\label{eqn:kato:sobolev}
\int_R \abs{\nabla_\pureH^m f_{Q,w}}^2 &\leq C \abs{Q} \qquad \text{for any cube $R$ with $\ell(R)=\ell(Q)$},
\\
\label{eqn:kato:L}
\abs{L_\pureH f_{Q,w}(x)} &\leq \frac{C}{\ell(Q)^{m}}
,\end{align}
and such that, for any array~$\arr \gamma_t$,
\begin{multline}
\label{eqn:kato:test}
\sup_Q \frac{1}{\abs{Q}} \int_0^{\ell(Q)} \int_Q \abs{\arr \gamma_t(x)}^2 \frac{dx\,dt}{t}
\\\leq
	C \sum_{w=1}^W \sup_{Q} \frac{1}{\abs{Q}}
 	\int_0^{\ell(Q)} \int_Q \abs{\langle\arr \gamma_t(x), A_t^Q \nabla^m_\pureH f_{Q,w}(x)\rangle}^2 \frac{dx\,dt}{t}
\end{multline}
where
$A_t^Q f(x)=\fint_{Q'} f(y)\,dy$, for $Q'\subset Q$ the unique dyadic subcube that satisfies $x\in Q'$ and $t\leq \ell(Q')<2t$.
\end{lem}

Specifically, the bound~\eqref{eqn:kato:L} is the bound (2.19) in \cite{AusHMT01}. The bound~\eqref{eqn:kato:sobolev} follows from the bound (2.18) in~\cite{AusHMT01} (if $R=Q$) and the observation that, by Lemma~3.1 in \cite{AusHMT01} and the definition of $f_{Q,w}$ therein, $\nabla_\pureH^m f_{Q,w}=\nabla_\pureH^m f_{R,w}$ whenever $\ell(Q)=\ell(R)$. Finally, the bound~\eqref{eqn:kato:test} is simply Lemma~2.2 of \cite{AusHMT01}. The requirement that $2m\geq n$ is a sufficient condition (see \cite[Propositon~2.5]{AusHMT01} or \cite{Dav95,AusT98}) for $L_\parallel$ to satisfy a pointwise upper bound; this condition is assumed in the proofs of the above results.

Let $(\gamma_t)_\beta=\1_{\delta<t<1/\delta}\Theta_t^\beta 1$; notice that $\arr \gamma_t$ includes $\Theta_t^\beta 1$ provided $\beta_\dmn=0$.
Thus, the estimates~\eqref{eqn:kato} and thus~\eqref{eqn:Theta:T1}, for $\beta_\dmn=0$, follow from bounds on the quantity
\begin{equation*}\frac{1}{\abs{Q}}\int_\delta^{\min(1/\delta,\ell(Q))} \int_Q \abs[Big]{\sum_{\beta_\dmn=0} \Theta_t^\beta 1(x)\, A_t^Q \partial^\beta_\pureH \! f_{Q,w}(x)}^2 \frac{dx\,dt}{t}
\end{equation*}
for $\delta<\ell(Q)$ and $\delta<1$.

We will divide this quantity into a sum of controllable terms as follows.
Let $P_t f(x)=f*\psi_t(x)$, where $\psi_t(x)=t^{-n}\psi(x/t)$. We require that $\psi$ be smooth and nonnegative, that $\int\psi=1$, and that $\psi(x)=0$ whenever $\abs{x}>1/2$. (We will later impose some additional constraints on~$\psi$. Notice that it is convenient to use a different approximate identity $P_t$ in this section from that used in Section~\ref{sec:kato:mixed:alpha}.) Let
\begin{align*}
R_t^{1,\beta} F(x) &= \Theta_t^\beta 1(x)\, (A_t^Q F(x)-P_t F(x))
,\\
R_t^{2,\beta} F(x) &= \Theta_t^\beta 1(x)\, P_t F(x) - \Theta_t^\beta (P_t F)(x)
,\\
R_t^{3,\beta} F(x) &= \Theta_t^\beta (P_t F-F)(x)
,\\
R_t^{4,\beta} F(x) &= \Theta_t^\beta F(x)
\end{align*}
so that we seek to bound
\begin{equation*}{\sum_{\beta_\dmn=0} \Theta_t^\beta 1(x)\, A_t^Q \partial^\beta \! f_{Q,w}(x)}
	= \sum_{j=1}^4 \sum_{\beta_\dmn=0}
	R_t^{j,\beta} \partial^\beta\!f_{Q,w}(x).
\end{equation*}

We begin with $R_t^{4,\beta}$. Observe that by the definition~\eqref{eqn:Theta:Carleson} of $\Theta_t^\beta$,
\begin{equation*}\sum_{\beta_\dmn=0}
	R_t^{4,\beta} \partial^\beta\!f_{Q,w}(x)
=
	\sum_{\substack{\alpha_\dmn=0\\ \beta_\dmn=0}} t^k\int_{\R^\dmnMinusOne} \partial_y^\alpha \partial_t^{m+k-1} E^L(x,t,y,0)\,A_{\alpha\beta}(y) \, \partial^\beta\! f_{Q,w}(y) \,dy.
\end{equation*}
Using formula~\eqref{eqn:fundamental:vertical} and then integrating by parts in~$y$, we see that this quantity is equal to
\begin{equation*}\sum_{\beta_\dmn=0}
	R_t^{4,\beta} \partial^\beta\!f_{Q,w}(x)
=	(-1)^{k-1}t^k\int_{\R^\dmnMinusOne} \partial_t^{m}\partial_s^{k-1} E^L(x,t,y,0)\,
L_\pureH f_{Q,w}(y) \,dy
\end{equation*}
and by the bounds \eqref{eqn:kato:L} and~\eqref{eqn:fundamental:slices}, we see that if $k$ is large enough then
\begin{equation*}\frac{1}{\abs{Q}}\int_0^{\ell(Q)} \int_Q \abs[Big]{\sum_{\beta_\dmn=0}
	R_t^{4,\beta} \partial^\beta\!f_{Q,w}(x)}^2\frac{dx\,dt}{t}\leq C.\end{equation*}

\subsection{\texorpdfstring{The term $R_t^{1,\beta} \partial^\beta\!f_{Q,w}$}{The first term}}

Next, we bound $R_t^{1,\beta} \partial^\beta \! f_{Q,w}$. We begin by bounding $\Theta_t^\beta 1(x)$.
Recall the following special case of Morrey's inequality (see, for example, \cite[Section~5.6.3]{Eva98}): if $x\in Q\subset\R^n$, then
\begin{equation*}\abs{v(x)}\leq \sum_{i=0}^m C\ell(Q)^{i}\biggl(\fint_Q \abs{\nabla_\pureH^i v}^2\biggr)^{1/2}\quad\text{provided }2m>n.\end{equation*}
We apply this bound to the function $v(x)=\partial_y^\alpha \partial_t^{m+k-1} E^L(x,t,y,0)$, a locally Sobolev function for almost all $y$ and~$t$. Then by the bound~\eqref{eqn:fundamental:slices}, we have that if $\abs{\alpha}=m$ and either $\abs{t}=\ell(R)$ or $\abs{t}<\ell(R)$ and $j\geq 1$, then
\begin{equation}\label{eqn:fundamental:slices:higher}
\int_{A_j(R)} \abs{\partial_{y,s}^\alpha \partial_t^{m+k-1} E^L(x,t,y,0)}^2 \,dy
\leq C\ell(R)^{-n-2k} 2^{-j(n+2k)}
.\end{equation}
Observe that this bound is valid for any $k>1/2-n/2$; in this section we will need this bound only for $k$ large, but in Section~\ref{sec:b} we will need this bound for $k=0$ and $k=1$ as well. Also, by formula~\eqref{eqn:fundamental:symmetric} it is valid with the roles of $y$ and $x$ reversed.

Using H\"older's inequality and summing over~$j$, we see that if $k>0$ then
\begin{equation*}\abs{\Theta_t^\beta 1(x)} = \abs[bigg]{\sum_{\alpha_\dmn=0} t^k \int_{\R^\dmnMinusOne} \partial_y^\alpha \partial_t^{m+k-1} E^L(x,t,y,0)\,A_{\alpha\beta}(y) \,dy}\leq C\end{equation*}
and so
\begin{equation*}\abs{R_t^{1,\beta}\partial^\beta\!f_{Q,w}(x)} \leq C \abs{A_t^Q \partial^\beta\! f_{Q,w}(x) - P_t \partial^\beta \! f_{Q,w}(x)}.\end{equation*}
Thus, we need only bound $A_t^Q \partial^\beta\! f_{Q,w}(x) - P_t \partial^\beta \! f_{Q,w}(x)$. We will do this using a standard orthogonality argument.

Let $\widetilde R_t^{1}=A_t^Q-P_t$. Recall that the kernel $\psi$ of $P_t$ is supported in $B(0,1/2)$; thus, if $x\in Q$ and $t<\ell(Q)$, then $\widetilde R_t^{1} F(x)=\widetilde R_t^{1}(\1_{2Q} F)(x)$, and so
\begin{align*}
\int_0^{\ell(Q)} \! \int_Q \abs{R_t^{1,\beta} \partial^\beta\!f_{Q,w}(x)}^2 \frac{dx\,dt}{t}
&\leq
	C \int_0^\infty  \int_{\R^n}\abs{\widetilde R_t^{1} (\1_{2Q}\partial^\beta\!f_{Q,w})(x)}^2 \frac{dx\,dt}{t}
.\end{align*}
Let $\{Q_s\}$ be a CLP family, as in Section~\ref{sec:Q}, but with the kernel $\varphi$ (and not its Fourier transform~$\widehat\varphi$) supported in $B(0,1/2)$. By the identity~\eqref{eqn:Q:identity},
\begin{align*}
\int_0^{\ell(Q)} \! \int_Q \abs{R_t^{1,\beta} \partial^\beta\!f_{Q,w}(x)}^2 \frac{dx\,dt}{t}
&\leq
	C \int_0^\infty  \int_{\R^n}\abs[bigg]{\int_0^\infty\widetilde R_t^{1} Q_s^2 (\1_{2Q}\partial^\beta\!f_{Q,w})(x)\,\frac{ds}{s}}^2 \frac{dx\,dt}{t}
.\end{align*}
By H\"older's inequality, for any number $\varepsilon>0$,
\begin{multline*}
\int_0^{\ell(Q)} \int_Q \abs{R_t^{1,\beta} \partial^\beta\!f_{Q,w}(x)}^2 \frac{dx\,dt}{t}
\\\leq
	\frac{C}{\varepsilon} \int_0^\infty \int_0^\infty\int_{\R^n}\abs{\widetilde R_t^{1} Q_s^2 (\1_{2Q}\partial^\beta\!f_{Q,w})(x)}^2\, dx\, \max\biggl(\frac{s}{t},\frac{t}{s}\biggr)^\varepsilon \frac{ds}{s} \frac{dt}{t}
.\end{multline*}
We claim that
\begin{equation}\label{eqn:kato:orthogonal}
\doublebar{\widetilde R_t^{1} Q_s g}_{L^2(\R^n)}\leq C \min\biggl(\frac{s}{t},\frac{t}{s}\biggr)^{1/6} \doublebar{g}_{L^2(\R^n)}.
\end{equation}
Choose $\varepsilon=1/6$. Assuming validity of the bound~\eqref{eqn:kato:orthogonal}, we have that
\begin{multline*}
\int_0^{\ell(Q)} \int_Q \abs{R_t^{1,\beta} \partial^\beta\!f_{Q,w}(x)}^2 \frac{dx\,dt}{t}
\\\leq
	C \int_0^\infty \int_0^\infty
	\int_{\R^n}\abs{ Q_s(\1_{2Q}\partial^\beta\!f_{Q,w})(x)}^2 \,dx\, \min\biggl(\frac{s}{t},\frac{t}{s}\biggr)^{1/6} \frac{ds}{s} \frac{dt}{t}
.\end{multline*}
Interchanging the order of integration and evaluating the integral in~$t$, we see that
\begin{equation*}
\int_0^{\ell(Q)} \int_Q \abs{R_t^{1,\beta} \partial^\beta\!f_{Q,w}(x)}^2 \frac{dx\,dt}{t}
\leq
	C \int_0^\infty
	\int_{\R^n}\abs{ Q_s(\1_{2Q}\partial^\beta\!f_{Q,w})(x)}^2 \,dx\,\frac{ds}{s}
.\end{equation*}
By the bounds \eqref{eqn:Q:square} and~\eqref{eqn:kato:sobolev}, we have that
\begin{equation*}
\int_0^{\ell(Q)} \int_Q \abs{R_t^{1,\beta} \partial^\beta\!f_{Q,w}(x)}^2 \frac{dx\,dt}{t}
\leq
	C \doublebar{\partial^\beta\!f_{Q,w}}_{L^2(2Q)}^2\leq C\abs{Q}
\end{equation*}
as desired. Thus, to complete our bound on $R_t^{1,\beta} \partial^\beta\!f_{Q,w}$, we need only establish the estimate~\eqref{eqn:kato:orthogonal}.

Suppose first that $t\leq s$ and so $\min(s/t,t/s)=t/s$.
By definition of $\widetilde R_t^Q$ and~$Q_s$,
\begin{equation*}\widetilde R_t^1 Q_s g(x) =
\int_{\R^n} \int_{\R^n} \biggl(\frac{1}{\abs{Q'}}\1_{Q'}(y)-\psi_t(x-y)\biggr) \varphi_s(y-z) \, g(z)\,dy\,dz.\end{equation*}
Notice that $\int_{\R^n} \frac{1}{\abs{Q'}}\1_{Q'}(y)-\psi_t(x-y)\,dy=0$, and that the integrand is zero unless $\abs{x-y}<Ct$.
Thus
\begin{equation*}\widetilde R_t^1 Q_s g(x) \leq
\int_{\R^n} \int_{B(x,Ct)} \biggl(\frac{1}{\abs{Q'}}\1_{Q'}(y)-\psi_t(x-y)\biggr) \bigl(\varphi_s(y-z)-\varphi_s(x-z)\bigr) \,dy\,g(z)\,dz.\end{equation*}
Suppose that $y\in B(x,Ct)$. Because $\varphi$ is supported in $B(0,1/2)$ and $s\geq t$, if $\abs{x-z}>2Cs$, then $\varphi_s(y-z)-\varphi_s(x-z)=0$. Otherwise, \[\abs{\varphi_s(y-z)-\varphi_s(x-z)}\leq C s^{-n-1}\abs{y-x}.\]
Thus,
\begin{align*}\widetilde R_t^1 Q_s g(x)
&=
	C\int_{B(x,2Cs)} \frac{t}{s^{n+1}}
	\int_{B(x,Ct)} \abs[bigg]{\frac{1}{\abs{Q'}}\1_{Q'}(y)-\psi_t(x-y)}  \,dy\,\abs{g(z)}\,dz
\\&\leq
	C\frac{t}{s}\fint_{B(x,Cs)}
	\abs{g(z)}\,dz
\leq C\frac{t}{s} Mg(x)
\end{align*}
where $Mg$ denotes the Hardy-Littlewood maximal function of~$g$. It is well known that $M$ is bounded $L^p(\R^n)\mapsto L^p(\R^n)$ for any $1<p\leq\infty$, and so the estimate~\eqref{eqn:kato:orthogonal} is valid whenever $t<s$.

Recall that the kernel $\varphi_s$ of $Q_s$ also integrates to zero and that the kernel $\psi_t$ of $P_t$ is also smooth. Thus, by a similar argument, if $s\leq t$ then $\doublebar{P_t Q_s g}_{L^2(\R^n)} \leq C (s/t) \doublebar{g}_{L^2(\R^n)}$. Bounding $A_t^Q Q_s g$ is somewhat more involved, because the kernel $\frac{1}{\abs{Q'}}\1_{Q'}$ of $A_t^Q$ is not smooth.

Suppose $s\leq t$. Let $\eta=\eta_{t,s,x}$ be a smooth cutoff function that is identically $1$ in $Q'$ and is supported in $(1+\sqrt{s/t})Q'$. Then $\abs{\nabla\eta}\leq C /\sqrt{st}$. Let
\begin{equation*}B_{t,s}^Q G(x)=\frac{1}{\abs{Q'}}\int \eta(y)\,G(y)\,dy.\end{equation*}
By the same argument as above, we may show that if $s\leq t$ then $\abs{B_{t,s}^Q Q_s g(x)}\leq C(s/t)^{1/2}Mg(x)$. To conclude the argument, notice that
\begin{align*}\abs{A_t^Q Q_s g(x)-B_{t,s}^Q Q_s g(x)}
&\leq \frac{C}{t^n} \int_{\supp\eta\setminus Q'} \abs{Q_s g}
.\end{align*}
Notice that $\abs{\supp \eta\setminus Q'}\leq C t^n \sqrt{s/t}$. We apply H\"older's inequality to see that
\begin{align*}
\abs{A_t^Q Q_s g(x)-B_{t,s}^Q Q_s g(x)}
&\leq
	C\biggl(\frac{1}{t^n}\int_{\supp \eta}\abs{Q_s g}^{3/2}\biggr)^{2/3} \biggl(\frac{s}{t}\biggr)^{1/6}
\\&\leq
	C \biggl(\frac{s}{t}\biggr)^{1/6} M(\abs{Q_s g}^{3/2})(x)^{2/3}
.\end{align*}
Because $3/2<2$, the estimate~\eqref{eqn:kato:orthogonal} is valid for $s\leq t$ as well as $t\leq s$. This establishes our desired bound on $R_t^{1,\beta} \partial^\beta\!f_{Q,w}$.

\subsection{\texorpdfstring{The term $R_t^{2,\beta} \partial^\beta\!f_{Q,w}$}{The second term}}
\label{sec:kato:2}

Next, we bound $R_t^{2,\beta} \partial^\beta\!f_{Q,w}$. We will use the following lemma from \cite{AlfAAHK11}; this is a square-function $T1$ theorem that is somewhat simpler than Theorem~\ref{thm:grau:hofmann:1} but has more stringent requirements.
\begin{lem}[Lemna~3.5(ii) in \cite{AlfAAHK11}] \label{lem:AAAHKT1}
Suppose that $\{R_t\}_{t>0}$ is a family of operators satisfying
\begin{equation}\label{eqn:R1}
\doublebar{R_t(F\1_{A_j(Q)})}_{L^2(Q)}^2 \leq C 2^{-nj} \biggl(\frac{t}{2^j\ell(Q)}\biggr)^{4}\doublebar{F}_{L^2(A_j(Q))}^2
\end{equation}
for all $0<t<\ell(Q)$ and all $j\geq1$. %Suppose that $\doublebar{R_t F}_{L^2(\R^n)}\leq C \doublebar{F}_{L^2(\R^n)}$ for all $F\in L^2(\R^n)$, for some constant $C$ independent of~$t$.
Suppose further that for all $t>0$, all $F\in L^2(\R^n)$, and all well-behaved vector-valued functions~$\vec F$, we have the bounds
\begin{equation}\label{eqn:R1:div}
\doublebar{R_t F}_{L^2(\R^n)}\leq C \doublebar{F}_{L^2(\R^n)}
,\qquad
\doublebar{R_t\Div_\pureH \vec F}_{L^2(\R^n)}\leq \frac{C}{t}\doublebar{\vec F}_{L^2(\R^n)}.\end{equation}
Finally, suppose that $R_t 1=0$ for all $t>0$.

Then
\begin{equation*}\int_{\R^n}\int_0^\infty \abs{R_t F(x)}^2\frac{dt\,dx}{t}\leq C\doublebar{F}_{L^2(\R^n)}^2\end{equation*}
for all $F\in L^2(\R^n)$.
\end{lem}

By the definitions of $\Theta_t^\beta$ and $R_t^{2,\beta}$, we have that
\begin{align*}
R_t^{2,\beta} F(x)
&= \Theta_t^\beta 1(x)\, P_t F(x) - \Theta_t^\beta (P_t F)(x)
\\&=\sum_{\alpha_{n+1}=0}  t^k\int_{\R^\dmnMinusOne}
	\partial_t^{m+k-1} \partial_{y}^\alpha E^L(x,t,y,0)\,A_{\alpha\beta}(y) \bigl(P_t F(x)-P_t F(y)\bigr)\,dy.
\end{align*}
Let
\begin{equation}
\label{eqn:Rt2}
\widetilde R_t^{2,\alpha,\beta} F(x)=t^k\int_{\R^\dmnMinusOne} \partial_y^\alpha \partial_t^{m+k-1} E^L(x,t,y,0)\,A_{\alpha\beta}(y) \, (P_t F(x) - P_t F(y)) \,dy.
\end{equation}
In this section we need only bound $\widetilde R_t^{2,\alpha,\beta}$ for $\alpha_\dmn=0$; in Section~\ref{sec:kato:3} we will need an estimate on $\widetilde R_t^{2,\alpha,\beta}$ in the case where $\alpha_\dmn>0$.

Observe that $P_t 1(x)=P_t1(y)=1$ and so $\widetilde R_t^{2,\alpha,\beta}1 = 0$.

Now, recall that $P_t F(x) = \int t^{-n}\psi((x-y)/t)\,F(y)\,dy$ for some smooth, compactly supported function~$\psi$; then
\begin{equation*}\doublebar{P_t F}_{L^2(\R^n)} \leq C \doublebar{F}_{L^2(\R^n)} \quad \text{and} \quad \doublebar{P_t (\Div_\pureH\vec F)}_{L^2(\R^n)} \leq \frac{C}{t} \doublebar{\vec F}_{L^2(\R^n)}.\end{equation*}
%The second bound lets us formally extend $P_t \Div_\pureH$ to an operator defined on all of $L^2(\R^n\mapsto \C^n)$.
We must use this fact to establish the bounds \eqref{eqn:R1} and~\eqref{eqn:R1:div}.

Let $Q$ be a cube and let $0<t<\ell(Q)$.
If $j\geq 1$ and $F$ is supported in $A_{j}(Q)$, observe that $P_t F(x)$ is supported in $A_{j,1}(Q)$. Thus, by the bound~\eqref{eqn:fundamental:slices},
\begin{equation*}
\doublebar{\widetilde R_t^{2,\alpha,\beta} (F\1_{A_j(Q)})}_{L^2(Q)} \leq
C \frac{t^{k}}{\ell(Q)^{k}} 2^{-j(k-1/2+\varepsilon/2)}
\doublebar{P_t (F\1_{A_j(Q)})}_{L^2(\R^n)}
.\end{equation*}
(In the case $j=1$ some extra care must be taken to establish this estimate; however, it may be done by considering the cases $t>\ell(Q)/2$ and $t\leq \ell(Q)/2$ separately.)

This implies the bound~\eqref{eqn:R1}. We are left with the uniform $L^2$ bounds~\eqref{eqn:R1:div}.

Suppose that $F$ is supported in $8Q$ and that $\ell(Q)/2<t\leq\ell(Q)$. Then $P_t F(y)=0$ for all $y\notin 16Q$ and so
\begin{align*}
\abs{\widetilde R_t^{2,\alpha,\beta} F(x)}
&\leq
	Ct^k \int_{\R^\dmnMinusOne}
	\abs{\partial_y^\alpha \partial_t^{m+k-1} E^L(x,t,y,0)}\abs{P_t F(x)-P_t F(y)} \,dy
\\&\leq
	Ct^k\int_{16Q}
	\abs{\partial_y^\alpha \partial_t^{m+k-1} E^L(x,t,y,0)}\abs{P_t F(x)-P_t F(y)} \,dy
	\\&\qquad
	+ Ct^k \sum_{j=1}^\infty \abs{P_t F(x)}\int_{A_j(Q)}\abs{\partial_y^\alpha \partial_t^{m+k-1} E^L(x,t,y,0)}\,dy
.\end{align*}
Applying the bound~\eqref{eqn:fundamental:slices:higher}, we see that
\begin{align*}
\abs{\widetilde R_t^{2,\alpha,\beta} F(x)}
&\leq
	C \abs{P_t F(x)} + C t^{-n/2} \doublebar{P_t F}_{L^2(16Q)}
.\end{align*}
Thus,
\begin{equation*}\doublebar{R_t(F\1_{8Q})}_{L^2(Q)} \leq C \doublebar{P_t (F\1_{8Q})}_{L^2(\R^n)}.\end{equation*}
We sum over cubes of side-length~$t$; this yields the bound
\begin{equation*}\doublebar{R_tF}_{L^2(Q)} \leq C \doublebar{P_t F}_{L^2(\R^n)}\end{equation*}
and, combined with the existing bounds on $P_t F$ and $P_t\Div_\pureH F$, yields the desired estimates~\eqref{eqn:R1:div}.

Thus, Lemma~\ref{lem:AAAHKT1} applies and we may bound the operator~$\widetilde R_t^{2,\alpha,\beta}$.
In particular, using the bound~\eqref{eqn:kato:sobolev} and arguing as in the proof of Lemma~\ref{lem:square:carleson}, we have the desired Carleson bound on~$ R_t^{2,\beta} \partial^\beta\!f_{Q,w}$.

\subsection{\texorpdfstring{The term $R_t^{3,\beta} \partial^\beta\!f_{Q,w}$}{The third term}}
\label{sec:kato:3}

Finally, we consider the term $R_t^{3,\beta} \partial^\beta\!f_{Q,w}$.
As in the case of $R_t^{4,\beta}$, but unlike $R_t^{1,\beta}$ and $R_t^{2,\beta}$, we will not be able to bound the individual terms $R_t^{3,\beta} \partial^\beta\!f_{Q,w}$; we will only be able to bound
\begin{align*}\widetilde R_t^3 f_{Q,w} (x)
&=
	\sum_{\beta_\dmn=0} R_t^{3,\beta} \partial^\beta\!f_{Q,w} (x)
\\&=
	\sum_{\substack{\alpha_\dmn=0\\ \beta_\dmn=0}} \! t^k\int_{\R^\dmnMinusOne} \partial_y^\alpha \partial_t^{m+k-1} E^L(x,t,y,0)\,A_{\alpha\beta}(y) \, \partial^\beta (P_t f_{Q,w}-f_{Q,w})(y) \,dy
.\end{align*}
Another complication of this section is that we will only be able to establish the bound
\begin{equation*}\doublebar{\widetilde R_t^3 f_{Q,w}}_{\mathcal{C},\delta} \leq C_1+C_1\doublebar{\Theta_t^\perp 1}_{\mathcal{C},\delta}\end{equation*}
and not simply a bound of the form $\doublebar{\widetilde R_t^3 f_{Q,w}}_{\mathcal{C}}\leq C_1$.

Let $f=(f_{Q,w}-p_{Q,w})\eta_Q$, where $p_{Q,w}$ is an appropriate polynomial of degree $m-1$ and where $\eta_Q$ is a smooth cutoff function that is identically 1 in $2Q$ and is supported in $4Q$. By the bound~\eqref{eqn:kato:sobolev},
\begin{equation*}\doublebar{\nabla^m_\pureH f}_{L^2(\R^n)}\leq C \abs{Q}^{1/2}.\end{equation*}
Furthermore, $\partial^\beta f=\partial^\beta f_{Q,w}$ in $2Q$ whenever $\abs{\beta}=m$.
Using the bound~\eqref{eqn:fundamental:slices:higher} on $E^L$ and the bound~\eqref{eqn:kato:sobolev} on $\nabla^m_\pureH f_{Q,w}$, we may show that
\begin{equation*}\frac{1}{\abs{Q}}\int_Q \int_0^{\ell(Q)} \abs{\widetilde R_t^3 f_{Q,w} (x)-\widetilde R_t^3 f (x)}^2\frac{dt\,dx}{t}\leq C\end{equation*}
and so we need only bound~$\widetilde R_t^3 f(x)$.

By the definition~\eqref{eqn:L:parallel} of $L_\parallel$, and by formula~\eqref{eqn:fundamental:symmetric},
\begin{equation*}
\widetilde R_t^3 f(x) =
	t^k\int_{\R^\dmnMinusOne}
	\partial_t^{m+k-1} \overline{L_\pureH^* E^{L^*}(y,0,x,t) }
	\, (P_t f-f)(y) \,dy
\end{equation*}
where $L_\pureH^*$ is taken in the $y$ variable.
Recalling that $L^*_{y,s}(E^{L^*}(y,s,x,t))=0$ away from $(x,t)$, we see that
\begin{equation*}
\overline{L_\pureH^* E^{L^*}(y,0,x,t)}
= (-1)^{m+1} \sum_{\xi_\dmn+\zeta_\dmn\geq 1}
\partial^\zeta_{y,s}( {A_{\xi\zeta}(y)}\, \partial^\xi_{y,s} {E^{L}(x,t,y,0)} )
.\end{equation*}
Thus, we need only bound the quantities
\begin{equation*}
\widetilde R_t^{3,\xi,\zeta} f(x)
= (-1)^m t^k\int_{\R^\dmnMinusOne}
	\partial^\zeta_{y,s}( {A_{\xi\zeta}(y)} \, \partial^\xi_{y,s} \partial_t^{m+k-1} {E^{L}(x,t,y,0)} )
	\, (P_t f-f)(y) \,dy
\end{equation*}
where at least one of $\zeta_\dmn$ and $\xi_\dmn$ is positive.

For each multiindex $\zeta$, we write $\zeta=\zeta_\pureH+\zeta_\perp\vec e_\perp$, where $\zeta_\perp=\zeta_\dmn$ and where $\zeta_\pureH$ is a multiindex with $(\zeta_\pureH)_\dmn=0$. Integrating by parts and applying formula~\eqref{eqn:fundamental:vertical}, we see that
\begin{equation*}
\widetilde R_t^{3,\xi,\zeta} f (x)
=
t^k\int_{\R^\dmnMinusOne}
	{A_{\xi\zeta}(y)} \,\partial^\xi_{y,s} \partial_t^{m+k-1+\zeta_\perp} {E^{L}(x,t,y,0)}
	\, \partial^{\zeta_\pureH}(P_t f-f)(y) \,dy
.\end{equation*}
We wish to bound $\int_Q \abs{\widetilde R_t^{3,\xi,\zeta} f}^2$ for $0<t<\ell(Q)$. Let $S\subset Q$ be a dyadic subcube with $t/2<\ell(S)\leq t$.
Let $\sum_{j=0}^\infty \eta_j$ be a smooth partition of unity with $\eta_j$ supported in $A_{j,1}(S)$ and with the usual bounds on the derivatives of~$\eta_j$. Let $f_j=\eta_j\, f$.
Then
\begin{equation*}
\widetilde R_t^{3,\xi,\zeta} f (x)
=
\sum_{j=0}^\infty t^k\int_{\R^n}
	{A_{\xi\zeta}(y)} \,\partial^\xi_{y,s} \partial_t^{m+k-1+\zeta_\perp} {E^{L}(x,t,y,0)}
	\, \partial^{\zeta_\pureH}(P_t f_j-f_j)(y) \,dy
.\end{equation*}
By the bound~\eqref{eqn:fundamental:slices:higher}, if $x\in S$ then
\begin{multline*}\abs[bigg]{t^k\int_{\R^n}
	{A_{\xi\zeta}(y)} \,\partial^\xi_{y,s} \partial_t^{m+k-1+\zeta_\perp} {E^{L}(x,t,y,0)}
	\, \partial^{\zeta_\pureH}(P_t f_j-f_j)(y) \,dy}^2
\\\leq
	Ct^{-n-2\zeta_\perp}2^{-j(n+2k+2\zeta_\perp)}
	\int_{A_{j,2}(S)}
	\abs{ \partial^{\zeta_\pureH}(P_t f_j-f_j)(y)}^2 \,dy
.\end{multline*}
Thus,
\begin{equation*}\int_Q \abs{\widetilde R_t^{3,\xi,\zeta} f(x)}^2\,dx
\leq
	\sum_{S\subset Q} \sum_{j=0}^\infty
	Ct^{-2\zeta_\perp}2^{-j(n+2k+2\zeta_\perp)}
	\int_{A_{j,2}(S)}
	\abs{ \partial^{\zeta_\pureH}(P_t f_j-f_j)(y)}^2 \,dy.
	\end{equation*}
Summing carefully, we see that
\begin{equation*}\int_Q \abs{\widetilde R_t^{3,\xi,\zeta} f(x)}^2\,dx
\leq Ct^{-2\zeta_\perp}
	\int_{\R^n}
	\abs{ \partial^{\zeta_\pureH}(P_t f-f)(y)}^2 \,dy.\end{equation*}
Now, by Plancherel's theorem,
\begin{multline*}\int_0^\infty t^{-2\zeta_\perp}\int_{\R^n}
	\abs{\partial^{\zeta_\pureH}(P_t f-f)(y)}^2\,dy \,\frac{dt}{t}
\\\leq
	\int_{\R^n} \abs{\omega}^{2\abs{\zeta_\pureH}}\,
	\abs{\widehat{f}(\omega)}^2\,
	\int_0^\infty t^{-2\zeta_\perp}(1-\widehat\psi(t\omega))^2\,\frac{dt}{t}
	\,d\omega
\end{multline*}
where $\psi_t(x)=t^{-n}\psi(x/t)$ is the convolution kernel of~$P_t$. We require that $\psi$ be radial and make the change of variables $s=t\abs{\omega}$. Then
\begin{multline*}\int_0^\infty t^{-2\zeta_\perp} \int_{\R^n}
	\abs{\partial^{\zeta_\pureH}(P_t f-f)(y)}^2\,dy \,\frac{dt}{t}
\\\leq
	\int_{\R^n} \abs{\omega}^{2\abs{\zeta}} \,
	\abs{\widehat{f}(\omega)}^2\,
	\int_0^\infty s^{-2\zeta_\perp} (1-\widehat\psi(s))^2\,\frac{ds}{s}
	\,d\omega
.\end{multline*}
We require that $\int \psi=1$, and that the higher moments are zero, that is, that $\int x^\theta \,\psi(x)\,dx=0$ for all $\abs{\theta}$ small enough. This implies that $1-\widehat\psi(s)$ is small for $s$ small and so $s^{-2\zeta_\perp-1}(1-\widehat\psi(s))^2$ is integrable near zero. Because $\psi$ is smooth and compactly supported, we have that $\widehat\psi(s)$ is bounded. If $\zeta_\perp>0$ then the integral in $s$ converges. (We have that $\widehat\psi(s)\to 0$ as $s\to \infty$, and so the integral must diverge if $\zeta_\perp=0$.) Because $\doublebar{\nabla_\pureH^m f}_{L^2(\R^n)}\leq C\sqrt{\abs{Q}}$, we have that
\begin{equation*}\int_0^{\ell(Q)} \int_Q\abs{\widetilde R_t^{3,\xi,\zeta} f(x) }^2\,\frac{dx\,dt}{t}
\leq C \abs{Q}\end{equation*}
whenever $\zeta_\perp>0$.

We are left with the terms $\widetilde R_t^{3,\xi,\zeta} f$ for $\zeta_\perp=0$; recall that we need only consider $\zeta_\perp+\xi_\perp\geq 1$ and so we may assume $\xi_\perp\geq 1$. Because $\zeta_\perp=0$, we have that
\begin{equation*}
\widetilde R_t^{3,\xi,\zeta} f (x)
=
t^k\int_{\R^\dmnMinusOne}
	{A_{\xi\zeta}(y)} \,\partial^\xi_{y,s} \partial_t^{m+k-1} {E^{L}(x,t,y,0)}
	\, (\partial^{\zeta}P_t f(y)-\partial^\zeta f(y)) \,dy
.\end{equation*}
Recall from Section~\ref{sec:kato:2} that the operator $\widetilde R_t^{2,\xi,\zeta}$, given by formula~\eqref{eqn:Rt2}, satisfies square-function estimates. Thus, we need only bound
\begin{multline*}\widetilde R_t^{3,\xi,\zeta} f (x) + \widetilde R_t^{2,\xi,\zeta} f (x)
\\=
t^k\int_{\R^\dmnMinusOne}
	{A_{\xi\zeta}(y)} \,\partial^\xi_{y,s} \partial_t^{m+k-1} {E^{L}(x,t,y,0)}
	\, (\partial^{\zeta}P_t f(x)-\partial^\zeta f(y)) \,dy
.\end{multline*}
Let $\gamma=\xi-\vec e_\perp$. We use formula~\eqref{eqn:fundamental:vertical}; we then see that
\begin{multline*}\widetilde R_t^{3,\xi,\zeta} f (x) + \widetilde R_t^{2,\xi,\zeta} f (x)
\\=
-t^k\int_{\R^\dmnMinusOne}
	{A_{\xi\zeta}(y)} \,\partial^\gamma_{y,s} \partial_t^{m+k} {E^{L}(x,t,y,0)}
	\, (\partial^{\zeta}P_t f(x)-\partial^\zeta f(y)) \,dy
.\end{multline*}
We recognize the integrand as being much like the kernel of the single layer potential. By formula~\eqref{dfn:Theta:S:integral} for $\Theta_t^S$, we have that
\begin{equation*}\widetilde R_t^{3,\xi,\zeta} f (x) + \widetilde R_t^{2,\xi,\zeta} f (x)
=
-P_t (\partial^{\zeta}f)(x)\,\Theta_t^S (A_{\xi\zeta}\arr e_\gamma)(x)
	+ \Theta_t^S (A_{\xi\zeta} \,\partial^\zeta \!f\,\arr e_\gamma )(x)
.\end{equation*}

If $\gamma\neq \gamma_\perp$, then the operator $\Theta_t$ given by $\Theta_t f= \Theta_t^S(f\arr e_\gamma)$ satisfies the bound~\eqref{eqn:Theta:1:square} (see Section~\ref{sec:S:not-normal}).
If $\gamma=\gamma_\perp$, then
\begin{equation*}\Theta_t^S (A_{\xi\zeta} \,\partial^\zeta \!f\,\arr e_\gamma ) = \Theta_t^\perp (A_{\xi\zeta} \,\partial^\zeta \!f ).\end{equation*} The operator $\Theta_t^\delta = 1_{\delta<t<1/\delta} \Theta_t^\perp$ satisfies the conditions of Theorem~\ref{thm:grau:hofmann:1}, albeit with constants depending on $\doublebar{\Theta_t^\perp 1}_{\mathcal{C},\delta}$, and so also satisfies the bound~\eqref{eqn:Theta:1:square}. Thus, in either case, we have the bound
\begin{equation*}\int_{\R^n} \int_\delta^{1/\delta} \abs{\Theta_t^S (A_{\xi\zeta} \,\partial^\zeta \!f\,\arr e_\gamma )(x)}^2\frac{dx\,dt}{t}
\leq (C+C\doublebar{\Theta_t^\perp}_{\mathcal{C},\delta}) \abs{Q}.
\end{equation*}

To bound $P_t (\partial^{\zeta}f)(x)\,\Theta_t^S (A_{\xi\zeta}\arr e_\gamma)(x)$, recall Carleson's lemma (see, for example, \cite[Chapter II, Section 2.2]{Ste93}).
\begin{lem}
Let $F(x,t)$ be a function and $d\mu$ be a measure defined on $\R^\dmn_+$. Then
\begin{equation*}\abs[bigg]{\int_{\R^\dmn_+} F(x,t)\,d\mu(x,t)} \leq C \biggl(\sup_{R\subset\R^n} \frac{1}{\abs{R}}\int_R \int_0^{\ell(R)} \abs{d\mu}\biggr) \biggl( \int_{\R^n} \sup_{\abs{x-y}<t} \abs{F(y,t)} \,dx\biggr)\end{equation*}
provided the right-hand side is finite, where the supremum is taken over cubes $R\subset\R^n$.
\end{lem}
We wish to bound
\begin{equation*}\frac{1}{\abs{Q}} \int_Q \int_\delta^{\min(1/\delta,\ell(Q))} \abs{P_t (\partial^{\zeta}f)(x)}^2 \, \abs{\Theta_t^S (A_{\xi\zeta}\arr e_\gamma)(x)}^2\,\frac{1}{t}\,dt\,dx\end{equation*}
for $\delta$ small enough.
Let $F(x,t)= \abs{P_t (\partial^{\zeta}f)(x)}^2$; because $P_t$ is a smooth identity with a convolution kernel it is elementary to show that $\sup_{\abs{x-y}<t} \abs{F(y,t)} \leq C M(\partial^\zeta f)(x)$. Let
\begin{equation*}d\mu(x,t) = \1_{\delta<t<1/\delta} \abs{\Theta_t^S (A_{\xi\zeta}\arr e_\gamma)(x)}^2\,\frac{1}{t}\,dt\,dx.\end{equation*}
By Lemma~\ref{lem:square:carleson} and the preceding remarks, we have that
\begin{equation*}\sup_R \frac{1}{\abs{R}} \int_R \int_0^{\ell(R)} \abs{d\mu} \leq C + C \doublebar{\Theta_t^\perp 1}_{\mathcal{C},\delta}^2.\end{equation*}
This establishes the desired bound on $\widetilde R_t^{3,\xi,\zeta} f$.

\subsection{Terms with \texorpdfstring{$\alpha_\dmn=0$ and $\beta_\dmn>0$}{partly vertical derivatives}}
\label{sec:kato:mixed}

We conclude this section by bounding $\Theta_t^\beta 1$ for multiindices $\beta$ with $\beta_\dmn>0$. Recall that
\begin{equation*}\Theta_t^\beta f(x)=\sum_{\alpha_{n+1}=0}  t^k\int_{\R^\dmnMinusOne}
\partial_t^{m+k-1} \partial_{y}^\alpha E^L(x,t,y,0)\,A_{\alpha\beta}(y)\,f(y) \,dy.\end{equation*}

By a well known argument of Fefferman and Stein \cite{FefS72}, using decay of the kernel
of $\Theta^\beta_t $ (that is, the bound~\eqref{eqn:fundamental:slices}), we find that if $k$ is large enough then
\begin{equation*}\doublebar{\Theta^\beta_t 1}_{\mathcal{C}} \leq C + \sup_Q\frac{C}{|Q|}\int_0^{\ell(Q)}\!\!\int_Q
|\Theta^\beta_t (\1_{4Q})(x)|^2\frac{dx\,dt}{t}
.\end{equation*}

Let $( F_\beta)_\alpha:= A_{\alpha\beta} \, \1_{4Q}$; then $\arr F_\beta$ is an $L^2$ array-valued function.
More precisely, let $q$ be the number of multiindices $\zeta\in\N^n$ of length~$m$; alternatively, $q$ is the number of multiindices $\zeta\in\N^\dmn$ of length $m$ with $\zeta_\dmn=0$. We will think of $\C^q$ as the vector space of arrays of numbers indexed by such multiindices. Then for each~$\beta$, $\arr F_\beta $ is a function in $L^2(\R^n\mapsto\C^q)$.

Now, observe that
\begin{equation}
\label{eqn:Hodge}
\mat A_\pureH \nabla_\pureH^m L_\pureH^{-1}\Div_{m,\pureH} : L^2(\R^n\mapsto\C^q)\mapsto L^2(\R^n\mapsto\C^q)\end{equation}
is a bounded operator, where $\nabla_\pureH^m$ and $\Div_{m,\pureH}$ are as defined in Section~\ref{sec:dfn},
and so we have a Hodge decomposition of $L^2(\R^n\mapsto\C^q)$. Specifically, if $\arr F\in L^2(\R^n\mapsto\C^q)$, then
\begin{equation*}%\label{eqn:Hodge}
\arr F = \arr H + \mat A_\pureH\nabla^m\Phi
\end{equation*}
for some $\arr H\in L^2(\R^n\mapsto\C^q)$ and some $\Phi\in \dot W^2_m(\R^n)$, with $\Div_{m,\pureH}\arr H=0$ and with
\begin{equation*} %\label{eqn:Hodge:bound}
\|\arr H\|_{L^2(\R^n\mapsto\C^q)} \, +\, \|\Phi\|_{ \dot{W}^{m,2}(\R^n\mapsto\C)}\,\leq C\, \|\arr F\|_{L^2(\R^n\mapsto\C^q)} \,.
\end{equation*}

Applying the Hodge decomposition to $\arr F_\beta$, we see that
\begin{equation*}
\Theta_t^\beta \1_{4Q}(x)=\sum_{\alpha_{n+1}=0}  t^k\int_{\R^\dmnMinusOne}
\partial_t^{m+k-1} \partial_{y}^\alpha E^L(x,t,y,0)\,\Big(\arr H_\beta+\mat A_\pureH\nabla^m_\pureH\Phi_\beta\Big)_\alpha \,dy.
\end{equation*}
But because $\Div_{m,\pureH}\arr H_\beta=0$, we have that
\begin{align*}
\Theta_t^\beta \1_{4Q}(x)
&=
	\sum_{\alpha_{n+1}=0}  t^k\int_{\R^\dmnMinusOne}
	\partial_t^{m+k-1} \partial_{y}^\alpha E^L(x,t,y,0)\,\Big(\mat A_\pureH \nabla_\pureH^m\Phi_\beta\Big)_\alpha \,dy.
%\\&=
%	\sum_{\alpha_{n+1}=0}\sum_{\zeta_{n+1}=0}  t^k\int_{\R^\dmnMinusOne}
%	\partial_t^{m+k-1} \partial_{y}^\alpha E^L(x,t,y,0)\,A_{\alpha\zeta}(y)\,\partial^\zeta\Phi_\beta(y) \,dy.
\end{align*}
We may extend $\Phi_\beta$ to a function defined on $\R^\dmn$ by letting $\Phi_\beta(y,s)=\Phi_\beta(y)$. Observe that $\partial^\zeta\Phi_\beta=0$ unless $\zeta_\dmn=0$. Also, if $\alpha_\dmn=0$, then $(\mat A_\pureH\nabla_\pureH^m \Phi_\beta)_\alpha=(\mat A\nabla^m \Phi_\beta)_\alpha$. Thus,
\begin{align*}
\Theta_t^\beta \1_{4Q}(x)
&=
	\sum_{\alpha_{n+1}=0}\sum_{\abs{\zeta}=m}  t^k\int_{\R^\dmnMinusOne}
	\partial_t^{m+k-1} \partial_{y}^\alpha E^L(x,t,y,0)\,A_{\alpha\zeta}(y)\,\partial^\zeta\Phi_\beta(y) \,dy
.\end{align*}
Recall from formulas \eqref{dfn:Theta:D} and~\eqref{eqn:D:fundamental} that
\begin{align*}
\Theta_t^D (\Tr_{m,\semiH}\Phi_\beta)(x)
&=
	-\!\!\sum_{\abs{\alpha}=\abs{\zeta}=m} \!\! t^k\int_{\R^\dmn_-}
	\partial_t^{m+k} \partial_{y,s}^\alpha E^L(x,t,y,s)\,A_{\alpha\zeta}(y)\,\partial^\zeta\Phi_\beta(y) \,dy\,ds
.\end{align*}
Using the identity~\eqref{eqn:fundamental:vertical} and integrating in~$s$, we see that
\begin{align*}
\Theta_t^D (\Tr_{m,\semiH}\Phi_\beta)(x)
&=
	\sum_{\abs{\alpha}=\abs{\zeta}=m}  t^k\int_{\R^n}
	\partial_t^{m+k-1} \partial_{y,s}^\alpha E^L(x,t,y,0)\,A_{\alpha\zeta}(y)\,\partial^\zeta\Phi_\beta(y) \,dy
.\end{align*}
Thus,
\begin{align*}
\Theta_t^\beta \1_{4Q}(x)
&=
	-\sum_{\alpha_{n+1}>0}\sum_{\abs{\zeta}=m}  t^k\int_{\R^\dmnMinusOne}
	\partial_t^{m+k-1} \partial_{y,s}^\alpha E^L(x,t,y,0)\,A_{\alpha\zeta}(y)\,\partial^\zeta\Phi_\beta(y) \,dy
	\\&\qquad+\Theta_t^D(\Tr_{m,\semiH}\Phi_\beta)(x)
.\end{align*}
If $\alpha_\dmn>0$, then $\alpha=\gamma+\vec e_\dmn$ for some multiindex $\gamma$ with $\abs{\gamma}=m-1$. Conversely, if $\abs{\gamma}=m-1$, let $\widetilde \gamma=\gamma+\vec e_\dmn$. We may write
\begin{align*}
\Theta_t^\beta \1_{4Q}(x)
&=
	\sum_{\abs{\gamma}=m-1}\sum_{\abs{\zeta}=m}  t^k\int_{\R^\dmnMinusOne}
	\partial_t^{m+k} \partial_{y,s}^\gamma E^L(x,t,y,0)\,A_{\widetilde\gamma\zeta}(y)\,\partial^\zeta\Phi_\beta(y) \,dy
	\\&\qquad
	+\Theta_t^D(\Tr_{m,\semiH}\Phi_\beta)(x)
.\end{align*}
By formula~\eqref{dfn:Theta:S:integral} for $\Theta_t^S$, we see that
\begin{align*}
\Theta_t^\beta \1_{4Q}(x)
&=
	\Theta_t^S \arr G_\beta (x)
	+\Theta_t^D(\Tr_{m,\semiH}\Phi_\beta)(x)
\end{align*}
where $(G_\beta)_\gamma = \sum_\zeta A_{\widetilde\gamma\zeta} \partial^\zeta \Phi_\beta$.

Observe that $\arr G_\beta\in L^2(\R^n)$. We may bound the term $\Theta_t^S\arr G_\beta (x)$ as usual.

Recall that $\Theta_t^D$ acts on the space $\dot W\!A^2_{m,\semiH}(\R^n)$, the completion of $\{\Tr_{m,\semiH}\varphi:\varphi\in C^\infty_0(\R^\dmn)\}$ under the $L^2$ norm. Consider the subspace $W$, the completion of
\begin{equation*}\{\Tr_{m,\semiH}\varphi:\varphi\in C^\infty_0(\R^\dmn),\>\partial_\dmn^j\varphi(x,0)=0\text{ for all }x\in\R^n\text{ and all }j\geq 1\}\end{equation*}
under the $L^2$ norm. We may let $\Theta_t^{D\pureH}$ denote the restriction of $\Theta_t^D$ to the space~$W$. Notice that $\Tr_{m,\semiH}\Phi_\beta=\nabla_\pureH^m \Phi_\beta\big\vert_{\R^n}$, and so $\Theta_t^D(\Tr_{m,\semiH}\Phi_\beta)(x) =\Theta_t^{D,\pureH}(\Tr_{m,\semiH} \Phi_\beta)(x)$.
As we established in Sections~\ref{sec:kato}--\ref{sec:kato:3},
\begin{equation*}\doublebar{\Theta_t^{D\pureH} 1}_{\mathcal{C},\delta} \leq C + C\doublebar{\Theta_t^\perp 1}_{\mathcal{C},\delta}\end{equation*}
and so we may control $\Theta_t^{D,\pureH}(\Tr_{m,\semiH} \Phi_\beta)(x)$. This completes the argument that $\Theta_t^D 1$ satisfies a Carleson estimate.

\section{Test functions \texorpdfstring{$\arr b_Q$}{b}}
\label{sec:b}

In this section we will choose test functions $\arr b_Q$ such that we may apply Theorem~\ref{thm:grau:hofmann} to bound $\Theta_t^\perp$ and~$\Theta_t^D$. (The remaining components of $\Theta_t^S$ were bounded in Section~\ref{sec:S:not-normal}.) We will follow the example of \cite{GraH14p}, which considers the case $m=1$.

As in Section~\ref{sec:D:carleson}, we will make the assumption $2m>n$. Again, by Morrey's inequality, this implies that functions locally in~$\dot W^2_m(\R^n)$ are locally H\"older continuous. By Lemma~\ref{lem:slices}, if $2m>n$ then solutions to elliptic equations are locally in $L^2(\R^n\times\{t\})$ for constants~$t$, and thus are also locally H\"older continuous. (See also \cite[Appendix~B]{AlfAAHK11}, in which a similar argument is made.)

Fix some dyadic cube~$Q$. Let $y_Q$ be its midpoint.
Let
\begin{align}\label{eqn:F}
F_s(x,t) &=
	\partial_s^{m-1} E^L(x,t,y_Q,s)
\end{align}
and let $F_\pm = F_{\pm \kappa\ell(Q)}$
for some small positive number $\kappa$ to be chosen later.
By the bound~\eqref{eqn:fundamental:slices:higher} and the symmetry relation~\eqref{eqn:fundamental:symmetric}, we may see that if $\dmn\geq 3$ then $F_\pm(x,t)\in \dot W^2_m(\R^\dmn_\mp)$; furthermore, by Theorem~\ref{thm:fundamental} we see that $LF_\pm=0$ in $\R^\dmn_\mp$.
Thus, by the higher-order Green's formula~\eqref{eqn:green}, if $t>0$ then
\begin{align*}
\partial_t^m F_-(x,t)
&=
	-\partial_t^m \D^{\mat A} (\Tr_{m-1}^+F_-)(x,t)
	+\partial_t^m \s^{\mat A} (\M_{\mat A}^+ F_-)(x,t)
\end{align*}
and by the corresponding formula in $\R^\dmn_-$, if $t>0$ then
\begin{align*}
0
&=
	\partial_t^m \D^{\mat A} (\Tr_{m-1}^-F_+)(x,t)
	+\partial_t^m \s^{\mat A} (\M_{\mat A}^- F_+)(x,t)
.\end{align*}
Adding and applying the definition~\eqref{dfn:D:tilde} of~$\widetilde \D$, we see that
\begin{align*}
\partial_t^m F_-(x,t)
&=
	\partial_t^m \widetilde\D^{\mat A} (\Tr_{m,\semiH}F_+ -\Tr_{m,\semiH}F_-)(x,t)
	\\&\qquad
	+\partial_t^m \s^{\mat A} (\M_{\mat A}^+ F_- + \M_{\mat A}^- F_+)(x,t)
.\end{align*}
Thus, by the definitions \eqref{dfn:Theta:S} and~\eqref{dfn:Theta:D} of $\Theta_t^S$ and~$\Theta_t^D$,
\begin{multline}\label{eqn:F:green}
t^k\partial_t^{m+k} F_-(x,t)
\\=
	\Theta_t^D (\Tr_{m,\semiH}F_+ -\Tr_{m,\semiH}F_-)(x)
	%\\&\qquad\nonumber
	+\Theta_t^S (\M_{\mat A}^+ F_- + \M_{\mat A}^- F_+)(x)
.\end{multline}
Let
\begin{equation}\label{eqn:b:D}
\arr b_Q^D = \abs{Q}(\Tr_{m,\semiH}F_+ -\Tr_{m,\semiH}F_-)
.\end{equation}
%By the bound~\eqref{eqn:fundamental:slices:higher} and the symmetry relation~\eqref{eqn:fundamental:symmetric}, we have that $\arr b_Q^D\in L^2(\R^n)$ with $\doublebar{\arr b_Q^D}\leq C \kappa^{-n/2}\sqrt{\abs{Q}}$.

Recall that $\M_{\mat A}^+ u$ is only defined as a linear functional on~$\dot W\!A^2_{m-1/2}(\R^n)$, that is, as an operator acting on $m-1$th-order traces of $\dot W^2_m$-functions. Let $\arr b_Q^S$ be a representative of the operator $\abs{Q}(\M_{\mat A}^+ F_- + \M_{\mat A}^- F_+)$; that is, $\arr b_Q^S$ is an array of functions that satisfies
\begin{align}\label{eqn:b:S}
\langle \Tr_{m-1}\varphi, \arr b_Q^S\rangle_{\R^n}
&=
	\abs{Q}
	\langle \Tr_{m-1}\varphi, \M_{\mat A}^+ F_- + \M_{\mat A}^- F_+\rangle_{\R^n}
\\\nonumber
&=\abs{Q} \langle \nabla^m\varphi, A\nabla^m F_-\rangle_{\R^{n+1}_+}
	+ \abs{Q} \langle \nabla^m\varphi, A\nabla^m F_+\rangle_{\R^{n+1}_-}
\end{align}
for all smooth, compactly supported functions~$\varphi$. In Section~\ref{sec:b:S} we will show that there is some such array of functions that in addition lies in $L^2(\R^n)$.

Now, by formula~\eqref{eqn:fundamental:vertical} and by definition of $\arr b_Q^D$, $\arr b_Q^S$ and $F_-$,
\begin{align*}\Theta_t(\arr b_Q^D,\arr b_Q^S)
&=\Theta_t^D\arr b_Q^D(x)+\Theta_t^S\arr b_Q^S(x)
=\abs{Q}t^k\partial_t^{m+k} F_-(x,t)
\\&=\abs{Q}t^k\partial_t^{m} \partial_s^{m+k-1} E^L(x,t,y_Q,-\kappa\ell(Q))
.\end{align*}
An application of the bound~\eqref{eqn:fundamental:slices:higher}, with the roles of $x$ and $y$ reversed, reveals that the bound ~\eqref{eqn:b:carleson} is valid for this choice of $\arr b_Q=(\arr b_Q^S,\arr b_Q^D)$, albeit with constant $C_0$ that depends on our choice of~$\kappa$. 

We thus need only show that this choice of $\arr b_Q$ satisfies the bounds~\eqref{eqn:b:L2}, \eqref{eqn:b:below} and~\eqref{eqn:b:above}, with the distinguished component $b_Q^{p+1}=b_Q^\perp$ in~\eqref{eqn:b:below} the $\arr e_\perp$-component of $\arr b_Q^S$.

\begin{rmk}
Although we will not make use of this fact, we observe that by the definition
\eqref{eqn:F} of $F_\pm$, the symmetry property~\eqref{eqn:fundamental:symmetric}, and formula~\eqref{eqn:D:fundamental} for the double layer potential, we have that
\begin{align*}
\langle  \arr b_Q^S, \Tr_{m-1}\varphi\rangle_{\R^n}
&=
	\abs{Q}\partial_\dmn^{m-1} \D^{A^*}(\Tr_{m-1}\varphi)(y_Q,-\kappa\ell(Q))
	\\&\qquad
	- \abs{Q}	\partial_\dmn^{m-1} \D^{A^*}(\Tr_{m-1}\varphi)(y_Q,\kappa\ell(Q))
.\end{align*}
Thus, $\arr b_Q^S$ may be viewed as the kernel of the double layer potential. If $m=1$, then the classic jump relation $\Trace^-\D^{\mat A^*} f-\Trace^+\D^{\mat A^*} f=f$ is well known. In the higher-order case, the analogous jump relation (see \cite{Bar15p}) is
\begin{equation*}\Tr_{m-1}^- \D^{\mat A^*}\arr f-\Tr_{m-1}^+ \D^{\mat A^*}\arr f=\arr f.\end{equation*}
Thus, if $\kappa$ is small enough and $\varphi$ is smooth, then $\langle  \arr b_Q^S, \Tr_{m-1}\varphi\rangle_{\R^n}$ is approximately $\abs{Q}\partial_\dmn^{m-1} \varphi(y_Q,0)$. Thus, we expect $\arr b_Q^S$ to be approximately equal to $\arr e_\perp$ near~$Q$, and so it is reasonable to expect the bounds \eqref{eqn:b:below} and~\eqref{eqn:b:above} to be valid for~$\arr b_Q^S$.
\end{rmk}

\begin{rmk} Recall from formula~\eqref{eqn:fundamental:unique} that if $2m\geq\dmn$, precisely the case considered in this section, then the fundamental solution $E^L$ in the definition~\eqref{eqn:F} of $F_s$ is only determined up to adding polynomials. However, note the presence of the vertical derivative $\partial_s^{m-1}$ in the definition of~$F_s$; this vertical derivative suffices to remove the terms of the form $f_\zeta(x,t) \, (y,s)^\zeta$ in formula~\eqref{eqn:fundamental:unique}, leaving $F_s(x,t)$ well-defined up to adding polynomials in $x$ and~$t$. The function $F_s$ is a tool used to define $\arr b_Q^S$ and $\arr b_Q^D$; notice from formulas \eqref{eqn:b:D} and~\eqref{eqn:b:S} that these quantities depend only on the higher-order derivatives of $F_s$, and so the lower-order terms in formula~\eqref{eqn:fundamental:unique} do not affect our results.
\end{rmk}

\begin{rmk} \label{rmk:2d} The conclusions of this section are also valid if $\dmn=2$; the analysis is somewhat more complicated because $F_{\pm}$ is no longer in $\dot W^2_m(\R^\dmn_\pm)$.

By Morrey's inequality, Lemma~\ref{lem:slices}, and the bound~\eqref{eqn:fundamental:far:lowest:2}, we have that if $\dmn=2$ and $R$ is a cube of side-length $\abs{s-t}$ then
\begin{equation}\label{eqn:fundamental:slices:higher:2}
\int_{A_j(R)} \abs{\nabla_{x,t}^m \partial_s^{m-1} E^L(x,t,y,s)}^2 \,dy
\leq \frac{C(\delta)}{\abs{s-t}} 2^{-j(1-\delta)}
\end{equation}
for any $\delta>0$. In particular,
\begin{equation*}\int_{\R^1} \abs{\nabla^{m} F_s(x,t)}^2\,dx \leq \frac{C}{\abs{t-s}}
.\end{equation*}
This is similar to the proof of the bounds \eqref{eqn:fundamental:slices} and~\eqref{eqn:fundamental:slices:higher}, but we must use the bound~\eqref{eqn:fundamental:far:lowest:2} instead of the bound~\eqref{eqn:fundamental:far} in order to take $m-1$ derivatives in the  variable $s$ rather than the variables~$(x,t)$.

We may use the bound~\eqref{eqn:fundamental:slices:higher} with the roles of $x$ and $y$ reversed to show that if $t<s<\sigma$ or $t>s>\sigma$, then
\begin{equation*}\int_{\R^1} \abs{\nabla^m F_s(x,t)-\nabla^m F_\sigma(x,t)}^2\,dx \leq C\frac{\abs{\sigma-s}^2}{\abs{t-s}^3}
.\end{equation*}
Thus, if $0<s<\sigma$, then $F_{\pm s}-F_{\pm \sigma} \in \dot W^2_m(\R^\dmn_\mp)$, and so we may apply the Green's formula~\eqref{eqn:green} and the equivalent in $\R^\dmn_-$ to see that
\begin{multline*}
t^k\partial_t^{m+k} F_{-s}(x,t) -t^k\partial_t^{m+k} F_{-\sigma}(x,t)
\\\begin{aligned}
&=
	\Theta_t^D (
	\Tr_{m,\semiH}F_s -\Tr_{m,\semiH}F_{-s}
	-\Tr_{m,\semiH}F_\sigma +\Tr_{m,\semiH}F_{-\sigma})(x,t)
	\\&\qquad
	+\Theta_t^S
	(
	\M_{\mat A}^+ F_{-s} + \M_{\mat A}^- F_s
	-\M_{\mat A}^+ F_{-\sigma} - \M_{\mat A}^- F_\sigma
	)(x,t)
.\end{aligned}\end{multline*}
Fix some $t>0$ and let $s=\kappa\,\ell(Q)$. Observe that if we take the limit as $\sigma\to\infty$, then the left-hand side approaches $t^k \partial_t^{m+k} F_{-s}(\,\cdot\,,t)$ in $L^2(\R^n)$. Furthermore, $\Tr_{m,\semiH}F_{\pm\sigma}\to 0$ in $L^2(\R^n)$. In Lemma~\ref{lem:b:L2} below, we will see that $\arr b_Q^S\to 0$ as the implied constant $\kappa\to\infty$; by definition of~$\arr b_Q^S$, this implies that $\M_{\mat A}^+ F_{- \sigma}+\M_{\mat A}^- F_{-\sigma} \to 0$ in $L^2(\R^n)$ as $\sigma\to \infty$. Thus, by the bound~\eqref{eqn:Theta:L2}, we have that
\begin{multline*}
t^k\partial_t^{m+k} F_{-s}(\,\cdot\,,t)
\\\begin{aligned}
&=
	\Theta_t^D (
	\Tr_{m,\semiH}F_s -\Tr_{m,\semiH}F_{-s}
	)(\,\cdot\,,t)
	%\\&\qquad
	+\Theta_t^S
	(
	\M_{\mat A}^+ F_{-s} + \M_{\mat A}^- F_s
	)(\,\cdot\,,t)
\end{aligned}\end{multline*}
as $L^2(\R^n)$-functions. Applying the Caccioppoli inequality, Lemma~\ref{lem:slices} and Morrey's inequality, we see that this equality must be true pointwise as well. Thus, formula~\eqref{eqn:F:green} is still valid if $\dmn=2$ and we may proceed as above.
\end{rmk}

\subsection{Bounds on \texorpdfstring{$\arr b_Q^D$}{bD}}
\label{sec:b:D}

By the bounds \eqref{eqn:fundamental:slices:higher} or~\eqref{eqn:fundamental:slices:higher:2}, we have that
\begin{equation*}\int_{\R^n} \abs{\arr b_Q^D}^2
\leq \frac{C}{\kappa^n}\abs{Q}.
\end{equation*}
Thus, $\arr b_Q^D$ satisfies the bound~\eqref{eqn:b:L2} with constant $C_0=C\kappa^{-n}$.

We now show that $\arr b_Q^D$ satisfies the bound~\eqref{eqn:b:above}.
Following \cite[Section~3]{GraH14p}, we fix a small positive constant $\omega$. Let $\phi_Q$ be supported on $(1+\omega)Q$ with $\phi_Q =1$ on $(1/2)Q$. We may choose $\phi_Q$ such that $\abs{\nabla\phi_Q(x)}<2/\ell(Q)$ for all~$x$, and such that $\phi_Q>\omega$ on~$Q$. We then set $d\mu_Q = \phi_Q\,dx$. Observe that the conditions~\eqref{eqn:mu} are valid for $C_0=\max(2,1/\omega)$.

Then by definition of $\arr b_Q^D$, if we let $F_Q=F_+-F_-$, then
\begin{align*}
\fint_Q \arr b_Q^D\,d\mu_Q
&=
\int_Q \Tr_{m,\semiH} F_Q\, \phi_Q
\\&=
	\int_{\R^n} \Tr_{m,\semiH} F_Q\, \phi_Q
	-\int_{\R^n\setminus Q} \Tr_{m,\semiH} F_Q\, \phi_Q
.\end{align*}
Recall that each component of $\Tr_{m,\semiH} F_Q$ may be written as $\partial^\beta F_Q(x,0)$ for some $\beta$ with $\beta_\dmn<\abs{\beta}=m$. In particular, $\beta=\vec e_j+\gamma$ for some $1\leq j\leq n$ and some multiindex~$\gamma$.
Integrating by parts, we see that
\begin{equation*}\abs[bigg]{\int_{\R^n} \Tr_{m,\semiH} F_Q\, \phi_Q}\leq \int_{\R^n} \abs{\Tr_{m-1} F_Q}\, \abs{\nabla\phi_Q}.\end{equation*}
Recalling the regions on which $\phi_Q$ and $\nabla\phi_Q$ are supported, we see that
\begin{align*}
\abs[bigg]{\fint_Q \arr b_Q^D\,d\mu_Q }
&\leq
	\frac{C}{\ell(Q)}\int_{(1+\omega)Q\setminus(1/2)Q}\abs{ \nabla^{m-1} F_Q}
	+
	\int_{(1+\omega)Q\setminus Q} \abs{\nabla^m F_Q}
.\end{align*}
Now, observe that $LF_Q=0$ away from $(y_Q,\pm \kappa\ell(Q))$. Thus, we may apply H\"older's inequality and Lemma~\ref{lem:slices} to see that
\begin{align*}
\abs[bigg]{\fint_Q \arr b_Q^D\,d\mu_Q }
&\leq
	\frac{C\abs{Q}^{1/2}}{\ell(Q)^{3/2}}
	\biggl(\int_{2Q\setminus(1/4)Q}\int_{-\ell(Q)}^{\ell(Q)}\abs{ \nabla^{m-1} F_Q}^2\biggr)^{1/2}
	\\&\qquad+
	C\frac{\abs{Q}^{1/2}}{\ell(Q)^{1/2}}\biggl(\int_{(3/2)Q\setminus(1/2)Q} \int_{-\ell(Q)/2}^{\ell(Q)/2} \abs{\nabla^m F_Q}^2\biggr)^{1/2}
.\end{align*}
We may use the Caccioppoli inequality (Lemma~\ref{lem:Caccioppoli}) to control the second term by the first term. For ease of notation let $S=(2Q\setminus(1/4)Q)\times(-\ell(Q),\ell(Q))$. Recalling the definition of~$F_Q$, we see that
\begin{align*}
\abs[bigg]{\fint_Q \arr b_Q^D\,d\mu_Q }^2
&\leq
	\frac{C\abs{Q}}{\ell(Q)^{3}}
	\int_S
	\abs[bigg]{
	\int_{-\kappa\ell(Q)}^{\kappa \ell(Q)} \nabla_{x,t}^{m-1} \partial_s^{m} E^L(x,t,y_Q,s)\,ds}^2
	\,dx\,dt
.\end{align*}
Applying H\"older's inequality again, we see that
\begin{align*}
\abs[bigg]{\fint_Q \arr b_Q^D\,d\mu_Q }^2
&\leq
	\frac{C\abs{Q}\kappa}{\ell(Q)^2}
	\int_S
	\int_{-\kappa\ell(Q)}^{\kappa \ell(Q)}
	\abs{  \nabla_{x,t}^{m-1} \partial_s^{m} E^L(x,t,y_Q,s)}^2 ds
	\,dx\,dt
\end{align*}
Using formula~\eqref{eqn:fundamental:vertical},~\eqref{eqn:fundamental:symmetric} and the bound~\eqref{eqn:fundamental:slices:higher}, we see that
\begin{align*}
\abs[bigg]{\fint_Q \arr b_Q^D\,d\mu_Q }^2
&\leq
	C\kappa^2
\end{align*}
Thus, if we choose $\kappa$ small enough, then $\arr b_Q^D$ satisfies the bound~\eqref{eqn:b:above}.
Notice that $\kappa$ may be chosen depending only on the constant $C_1$ in the bound~\eqref{eqn:Theta:T1}, that is, on the numbers determined in Section~\ref{sec:D:carleson}. It is acceptable for the numbers $C_0$ in the bounds~\eqref{eqn:mu}, \eqref{eqn:b:carleson} and~\eqref{eqn:b:L2} to grow as $\kappa\to 0$. In particular, recall that $\arr b_Q^D$ satisfies the bound~\eqref{eqn:b:L2} with a constant $C_0(\kappa)=C\kappa^{-n}$; this growth is acceptable.

Thus, $\arr b_Q^D$ satisfies all the conditions of Theorem~\ref{thm:grau:hofmann}.

\subsection{Bounds on \texorpdfstring{$\arr b_Q^S$}{bS}}
\label{sec:b:S}

To conclude the proof of Theorem~\ref{thm:square}, at least in the case $2m>n$, we need only show that $\arr b_Q^S$ satisfies the bounds~\eqref{eqn:b:L2}, \eqref{eqn:b:below} and~\eqref{eqn:b:above}.

The most involved argument of this section will be the proof of the following lemma.
\begin{lem}\label{lem:b:L2}
Suppose that $2m>n$.
If $\arr g\in \dot W\!A^2_{m-1}(\R^n)$, then
\begin{equation}
\label{eqn:b:L2:general}
\abs{\langle \arr g, \arr b_Q^S\rangle_{\R^n}}
	\leq \frac{C}{\kappa^{n/2}} {\sqrt{\abs{Q}}}
	\doublebar{\arr g}_{L^2(\R^n)}.\end{equation}
Furthermore, if $\arr g=0$ in $(1/4)Q$ and $\kappa<1/16$, then we have a better estimate:
\begin{equation}
\label{eqn:b:L2:close}
\abs{\langle \arr g, \arr b_Q^S\rangle_{\R^n}}
	\leq C\,{\kappa}{\sqrt{\abs{Q}}}
	\doublebar{\arr g}_{L^2(\R^n)}.\end{equation}
\end{lem}

The bound~\eqref{eqn:b:L2:general} is valid even for $\kappa$ large; recall that this bound was used in Remark~\ref{rmk:2d} to show that the bound~\eqref{eqn:b:carleson} is valid even in dimension $\dmn=2$.

Notice that this implies that $\arr b_Q^S$ is a bounded linear functional on $\dot W\!A^2_{m-1}(\R^n)$; if $m\geq 2$ then this is a proper subspace of $L^2(\R^n)$. Thus, $\arr b_Q^S$ lies in a quotient space of $L^2(\R^n)$. Once this lemma is proven we may extend $\arr b_Q^S$ to a bounded linear functional on $L^2(\R^n)$ (establishing the bound~\eqref{eqn:b:L2}), for example by orthogonal projection; we will need to select our projection carefully to ensure that $\arr b_Q^S$, after projection, satisfies the bound~\eqref{eqn:b:above}.

\begin{proof}[Proof of Lemma~\ref{lem:b:L2}]
It suffices to prove this lemma for all $\arr g$ such that $\arr g=\Tr_{m-1} \eta$ for some smooth, compactly supported function~$\eta$. Recall that
\begin{align*}\langle \arr g, \arr b_Q^S\rangle_{\R^n}
&=
	\abs{Q}\langle \arr g, \M_{\mat A}^+ F_-\rangle_{\R^n}
	+ \abs{Q}\langle \arr g, \M_{\mat A}^- F_+\rangle_{\R^n}
\\&=
	\abs{Q}\langle \nabla^m G, \mat A\nabla^m F_-\rangle_{\R^\dmn_+}
	+ \abs{Q}\langle \nabla^m G, \mat A\nabla^m F_+\rangle_{\R^\dmn_-}
\end{align*}
for any extension $G$ of $\arr g$. %In many of the following arguments we will discuss only the term $\abs{\langle \nabla^m G, A\nabla^m F_-\rangle_{\R^{n+1}_+}}$; the term $\abs{\langle \nabla^m G, A\nabla^m F_+\rangle_{\R^{n+1}_-}}$ is similar.

%First, observe that by Section~\ref{sec:b:0} we may assume that $g_{\gamma_\perp}=0$.

We will need to construct our extension $G$ of $\arr g$ carefully.
Let $H$ be the extension of~$\arr g$ given by Lemma~\ref{lem:extension}. Recall that $H$ satisfies the estimates \eqref{eqn:extension:L2:slices} and~\eqref{eqn:extension:square}, and that if $\arr g=0$ in $(1/4)Q$ then $\nabla^{m-1} H(x,t)=0$ whenever $\abs{t}<\dist(x,\R^n\setminus(1/4)Q)$.
Let $\varphi$ be smooth, supported in $B(0,1/2)$ and integrate to~$1$. Suppose further that the higher moments are zero, that is, $\int_{\R^n} x^\zeta\,\varphi(x)\,dx=0$ for all $1\leq \abs{\zeta}\leq m$.
Let
\begin{equation*}G(x,t)=\int_{\R^n} \frac{1}{t^n} \varphi\biggl(\frac{x-z}{t}\biggr)\,H(z,t)\,dz=\int_{\R^n} \varphi(z)\,H(x-zt,t)\,dz.\end{equation*}
To study the derivatives of~$G$, observe that for some constants $C_{\zeta,\xi}$,
\begin{equation*}\partial^\zeta G(x,t) =
\sum_{{\abs{\xi}=\abs{\zeta},\> \xi_\parallel \geq \zeta_\parallel}} C_{\zeta,\xi}
\int_{\R^n} z^{\xi_\parallel-\zeta_\parallel} \varphi(z)\,\partial^\xi H(x-zt,t)\,dz
\end{equation*}
where $\zeta_\pureH$ denotes the horizontal part of~$\zeta$, that is, $\zeta_\pureH = (\zeta_1,\dots,\zeta_n)$.
Let $J_{\zeta,\xi}(z)= C_{\zeta,\xi}z^{\xi_\parallel-\zeta_\parallel} \varphi(z)$, so that
\begin{equation}
\label{eqn:G:gradient}
\partial^\zeta G(x,t) = \sum_{\abs{\xi}=\abs{\zeta}} \int_{\R^n} \frac{1}{t^n} J_{\zeta,\xi}\biggl(\frac{x-z}{t}\biggr) \,\partial^\xi H(z,t)\,dz.
\end{equation}
By our moment condition on $\varphi$, we have that
\begin{equation*}\int_{\R^n} J_{\zeta,\xi}(z,t)\,dz=1\quad\text{if $\zeta=\xi$},
\qquad  \int_{\R^n} J_{\zeta,\xi}(z,t)\,dz=0\quad\text{otherwise}.\end{equation*}
Furthermore, $J_{\zeta,\xi}$ is a smooth cutoff function, and so  $\Tr_{m-1} G=\arr g$.
So we need only bound
\begin{equation*}
\abs{Q}\langle \nabla^m G, \mat A\nabla^m F_-\rangle_{\R^\dmn_+}
	+ \abs{Q}\langle \nabla^m G, \mat A\nabla^m F_+\rangle_{\R^\dmn_-}
.\end{equation*}

We will need some special arguments to establish the bound~\eqref{eqn:b:L2:close}. Arguing as in the proof of Lemma~\ref{lem:extension}, we see that if $\arr g=0$ in $(1/4)Q$ then $\nabla^{m-1} G(x,t)=0$ whenever $2\abs{t}<\dist(x,\R^n\setminus(1/4)Q)$. In particular, if $\kappa<1/8$ then $\nabla^{m-1} G=0$ near $(y_Q,\pm\kappa\ell(Q))$. (We require $\kappa<1/16$ so that $\nabla^{m-1} G=0$ everywhere within a fixed radius of $(y_Q,\pm\kappa\ell(Q))$.) Observe that $LF_+=0$ away from these points, and so
\begin{equation*}
0
=\langle \nabla^m G, \mat A\nabla^m F_+\rangle_{\R^\dmn}
=\langle \nabla^m G, \mat A\nabla^m F_+\rangle_{\R^\dmn_+}
	+\langle \nabla^m G, \mat A\nabla^m F_+\rangle_{\R^\dmn_-}.\end{equation*}
Thus, if $\arr g=0$ in $(1/4)Q$, we need only bound
\begin{equation*}
\abs{Q}\langle \nabla^m G, \mat A\nabla^m (F_- - F_+)\rangle_{\R^\dmn_+}
.\end{equation*}

We now introduce some notation. Let $\mathcal{G}=\{(x,t):t>(1/2)\dist(x,\R^n\setminus(1/4)Q)\}$, so if $\arr g=0$ in $(1/4)Q$ then $\supp\nabla^{m-1} G\cap\R^\dmn_+\subset\mathcal{G}$.
If $\abs{\alpha}=m$, let
\begin{align*}%\label{eqn:w}
w_\alpha^s(x,t)&= ({\mat A(x)}\nabla^m F_s(x,t))_\alpha = \sum_{\abs\beta=m} A_{\alpha\beta}(x)\,\partial^\beta F_s(x,t)
\end{align*}
with $w_\alpha^\pm=w_\alpha^{\pm \kappa\ell(Q)}$. Let $\widetilde w_\alpha=w_\alpha^+-w_\alpha^-$.
To establish the bound~\eqref{eqn:b:L2:close}, we need only bound
\begin{equation*}
\sum_{\abs{\alpha}=m} \int_{\mathcal{G}} \partial^\alpha G(x,t)\,\widetilde w_\alpha(x,t)\,dx\,dt
.\end{equation*}

To establish the bound~\eqref{eqn:b:L2:general}, we need only bound the quantity
\begin{equation*}\sum_{\abs{\alpha}=m} \int_{\R^\dmn_+} \partial^\alpha G(x,t)\,w_\alpha^-(x,t)\,dx\,dt+\int_{\R^\dmn_-} \partial^\alpha G(x,t)\,w_\alpha^+(x,t)\,dx\,dt.\end{equation*}
The second integral is similar to the first integral; thus, we will present the argument only for the first integral. In other words, we will work only in $\R^\dmn_+$, not $\R^\dmn_-$, whether our goal is to establish the bound \eqref{eqn:b:L2:general} or~\eqref{eqn:b:L2:close}.

We will need to bound $w_\alpha^s$; it will also help to bound vertical derivatives of $w_\alpha^s$. Let $j\geq 0$ be an integer.
Observe that by formula~\eqref{eqn:fundamental:vertical} and the definition of~$F_s$,
\begin{align*}
\int_{\R^n} \abs{\partial_t^j w_\alpha^s(x,t)}^2\,dx
&=
	\int_{\R^n} \abs{ \partial_s^{m+j-1} \nabla_{x,t}^m E^L(x,t,y_Q,s)}^2\,dx.
\end{align*}
By the bounds \eqref{eqn:fundamental:slices:higher} or~\eqref{eqn:fundamental:slices:higher:2}, if $2m>n$ and $j\geq 0$ then
\begin{equation*}\int_{\R^n} \abs{\partial_t^j w_\alpha^s(x,t)}^2\,dx \leq C \abs{t-s}^{-n-2j}.\end{equation*}
Thus, we have the bounds
\begin{align}
\label{eqn:w:uniform}
\sup_{t>0}\int_{\R^n} \abs{w_\alpha^-(x,t)}^2\,dx
&\leq \frac{C}{\kappa^n\abs{Q}}
,\\
\label{eqn:w:square}
\int_{\R^n}\int_0^\infty \abs{\partial_t w_\alpha^-(x,t)}^2\,t\,dt\,dx
&\leq \frac{C}{\kappa^n\abs{Q}}
,\\
\label{eqn:wt:decay}
\int_0^\infty \biggl(\int_{\R^n} \abs{\partial_t w_\alpha^-(x,t)}^2\,dx\biggr)^{1/2}\,dt
&\leq \frac{C}{\sqrt{\abs{Q}} \kappa^{n/2}}
.\end{align}

Now, by formula~\eqref{eqn:fundamental:vertical}, $\partial_t^j w_\alpha^s(x,t)=(-1)^j \partial_s^j w_\alpha^s(x,t)$. Furthermore, if $t>0$ $\abs{s}<\kappa\ell(Q)$, then by the bound~\eqref{eqn:fundamental:slices:higher}
\begin{equation*}
\int_{\R^n}\1_{\mathcal{G}}(x,t)\, \abs{\partial_s w_\alpha^s(x,t)}^2\,dx
\leq
	C(\ell(Q)+{t})^{-n-2}
.\end{equation*}
Thus, recalling that $\widetilde w_\alpha(x,t)=w_\alpha^+(x,t)-w_\alpha^-(x,t)$, we have that
\begin{align}
\label{eqn:w:uniform:close}
\sup_t\int_{\R^n} \1_{\mathcal{G}}(x,t)\, \abs{\widetilde w_\alpha(x,t)}^2\,dx
	&\leq \frac{C\kappa^2}{\abs{Q}}
,\\
\label{eqn:w:square:close}
\int_{\R^n}\int_0^\infty \1_{\mathcal{G}}(x,t)\,\abs{\partial_t \widetilde w_\alpha(x,t)}^2\,t\,dt\,dx
&\leq \frac{C\kappa^2}{\abs{Q}}
,\\
\label{eqn:wt:decay:close}
\int_0^\infty \biggl(\int_{\R^n} \1_{\mathcal{G}}(x,t)\,\abs{\partial_t \widetilde w_\alpha(x,t)}^2\,dx\biggr)^{1/2}\,dt
&\leq \frac{C\kappa}{\sqrt{\abs{Q}} }
.\end{align}

Recall that we wish to bound
\begin{equation*}\int_{\R^\dmn_+} \partial^\alpha G(x,t)\,\widetilde w_\alpha(x,t)\,dx\,dt
\qquad\text{or}\qquad
\int_{\R^\dmn_+} \partial^\alpha G(x,t)\,w_\alpha^-(x,t)\,dx\,dt.\end{equation*}

We will essentially proceed by integrating by parts to move one derivative from $G$ to $w_\alpha$; we will need separate arguments in the case where we integrate by parts in~$t$ (possible only if $\alpha_\dmn>0$) and in the case where we integrate by parts in a horizontal variable~$x_j$ (possible only if $\alpha_\dmn<m$).

First, if $\alpha_\dmn>0$, then $\alpha=\gamma+\vec e_\dmn$ for some multiindex $\gamma$ with $\abs{\gamma}=m-1$. So
\begin{equation*}
{\int_{\R^\dmn_+} \partial^\alpha G(x,t)\,w_\alpha^-(x,t)\,dx\,dt}
=
{\int_{\R^\dmn_+} \partial_t\partial^\gamma G(x,t)\,w_\alpha^-(x,t)\,dx\,dt}
.\end{equation*}
Integrating by parts in~$t$, we see that
\begin{align*}
{\int_{\R^\dmn_+} \partial^\alpha G(x,t)\,w_\alpha^-(x,t)\,dx\,dt}
&=
	-\lim_{t\to 0^+}{\int_{\R^n} \partial^\gamma G(x,t)\, w_\alpha^-(x,t)\,dx}
	\\&\qquad
	-{\int_{\R^\dmn_+} \partial^\gamma G(x,t)\,\partial_t w_\alpha^-(x,t)\,dx\,dt}
.\end{align*}
Recall that $H$ satisfies the uniform $L^2$ bound~\eqref{eqn:extension:L2:slices}; by formula~\eqref{eqn:G:gradient}, the same is true of~$G$.
We may control the first term using the estimate~\eqref{eqn:w:uniform}  and the second term using the estimate~\eqref{eqn:wt:decay}. This yields the bound
\begin{align*}
\abs[bigg]{\int_{\R^\dmn_+} \partial^\alpha G(x,t)\,w_\alpha^-(x,t)\,dx\,dt}
&\leq
	\frac{C}{\kappa^{n/2}\sqrt{\abs{Q}}}\doublebar{\arr g}_{L^2(\R^n)}
.\end{align*}

Similarly,
\begin{align*}
{\int_{\R^\dmn_+} \partial^\alpha G(x,t)\,\widetilde w_\alpha(x,t)\,dx\,dt}
&=
	-\lim_{t\to 0^+}{\int_{\R^n} \partial^\gamma G(x,t)\,\widetilde w_\alpha(x,t)\,dx}
	\\&\qquad
	-{\int_{\R^\dmn_+} \partial^\gamma G(x,t)\, \partial_t \widetilde w_\alpha(x,t)\,dx\,dt}
.\end{align*}
If $\arr g=0$ in $(1/4)Q$ then we may integrate over $\mathcal{G}$ rather than $\R^\dmn_+$. By the bounds \eqref{eqn:wt:decay:close} and~\eqref{eqn:w:uniform:close} on $\widetilde w_\alpha$ and the uniform $L^2$ estimate on $\nabla^{m-1} G$, we have that
\begin{align*}
{\int_{\R^\dmn_+} \partial^\alpha G(x,t)\,\widetilde w_\alpha(x,t)\,dx\,dt}
&\leq
	\frac{C\kappa}{\sqrt{\abs{Q}}}\doublebar{\arr g}_{L^2(\R^n)}.
\end{align*}

Now, we turn to the case where $\alpha_\dmn<\abs{\alpha}=m$. We still integrate by parts in~$t$. We see that, if $w_\alpha=\widetilde w_\alpha$ or $w_\alpha=w_\alpha^-$, then
\begin{multline*}
{\int_{\R^\dmn_+} \partial^\alpha G(x,t)\,w_\alpha(x,t)\,dx\,dt}
\\\begin{aligned}
&=
	-
	{\int_{\R^\dmn_+} t\,\partial_t\bigl(\partial^\alpha G(x,t)\,w_\alpha(x,t)\bigr)\,dx\,dt}
\\&=
	-
	{\int_{\R^\dmn_+} t\,\partial^\alpha \partial_t G(x,t)\,w_\alpha(x,t)\,dx\,dt}
	-
	{\int_{\R^\dmn_+} t\,\partial^\alpha G(x,t)\,\partial_t w_\alpha(x,t)\,dx\,dt}
.\end{aligned}\end{multline*}
Recall that $G$ as well as~$H$ satisfies the estimate~\eqref{eqn:extension:square}, and so by the bounds \eqref{eqn:w:square} and~\eqref{eqn:w:square:close},
\begin{align*}
\abs[bigg]{\int_{\R^\dmn_+} t\,\partial^\alpha G(x,t)\,\partial_t w_\alpha^-(x,t)\,dx\,dt}
&\leq
	\frac{C}{\kappa^{n/2}\sqrt{\abs{Q}}}
	\doublebar{\arr g}_{L^2(\R^n)}
,\\
\abs[bigg]{\int_{\mathcal{G}} t\,\partial^\alpha G(x,t)\,\partial_t \widetilde w_\alpha(x,t)\,dx\,dt}
&\leq
	\frac{C \kappa}{\sqrt{\abs{Q}}}
	\doublebar{\arr g}_{L^2(\R^n)}
.\end{align*}

We are left with the term
\begin{equation*}{\int_{\R^\dmn_+} t\,\partial^\alpha \partial_t G(x,t)\,w_\alpha(x,t)\,dx\,dt},\qquad \alpha_\dmn<\abs{\alpha}=m.\end{equation*}
We have square-function estimates on~$\partial_\dmn w_\alpha$ rather than $w_\alpha$; thus, we write
\begin{equation*}{\int_{\R^\dmn_+} t\,\partial^\alpha \partial_t G(x,t)\,w_\alpha(x,t)\,dx\,dt}
=
	{\int_{\R^\dmn_+} t\,\partial^\alpha \partial_t G(x,t)\,\int_t^\infty \partial_r w_\alpha(x,r)\,dr\,dx\,dt}
.\end{equation*}
Observe that if $r>t>0$ and $\arr g=0$ in $(1/4)Q$, then $\1_{\mathcal{G}}(x,r) \,\partial^\alpha \partial_t G(x,t)=\partial^\alpha \partial_tG(x,t)$; this is true because, if $\1_{\mathcal{G}}(x,r)\neq 1$, then $\nabla^m G(x,t)=0$.
Let $v_\alpha(x,r) = \partial_r w_\alpha^-(x,r)$ or $\1_{\mathcal{G}}(x,r)\,\partial_r \widetilde w_\alpha(x,r)$, depending on whether we seek to establish the bound \eqref{eqn:b:L2:general} or~\eqref{eqn:b:L2:close}.
Then
\begin{equation*}{\int_{\R^\dmn_+} t\,\partial^\alpha \partial_t G(x,t)\,w_\alpha(x,t)\,dx\,dt}
=
	{\int_{\R^\dmn_+} t\,\partial^\alpha \partial_t G(x,t)\,\int_t^\infty v_\alpha(x,r)\,dr\,dx\,dt}
.\end{equation*}

If $\alpha_\dmn<\abs{\alpha}=m$, let $j=j_\alpha$ be any integer such that $j\leq n$ and $\alpha_j>0$, and let $\zeta=\zeta_\alpha=\alpha- \vec e_j+\vec e_\dmn$. (For the sake of definiteness we may let $j_\alpha$ be the smallest such integer.) Then
\begin{equation*}{\int_{\R^\dmn_+} t\,\partial^\alpha \partial_t G(x,t)\,w_\alpha(x,t)\,dx\,dt}
=
	{\int_{\R^\dmn_+} t\,\partial_{j_\alpha} \partial^{\zeta_\alpha} G(x,t)\,\int_t^\infty v_\alpha(x,r)\,dr\,dx\,dt}
.\end{equation*}
For each pair of multiindices $\alpha$ and $\beta$ with $\abs{\alpha}=\abs{\beta}=m$, define the linear operator $T_{\alpha,\beta}$ by the relation
\begin{equation*}T_{\alpha,\beta} F(x,t)= \int_{\R^n}
	\frac{1}{t^n} J_{\alpha,\beta}\biggl(\frac{x-z}{t}\biggr) F(z,t) \,dz\end{equation*}
where $J_{\alpha,\beta}$ is as defined above in the discussion of $\nabla^m G$ and~$\nabla^m H$. Then
\begin{multline*}
{\int_{\R^\dmn_+} t\,\partial^\alpha \partial_t G(x,t)\,w_\alpha(x,t)\,dx\,dt}
\\=
	\sum_{\abs{\beta}=m}
	{\int_{\R^\dmn_+} t\,\partial_{j_\alpha} T_{\zeta_\alpha,\beta}\partial^\beta H(x,t)\,\int_t^\infty v_\alpha(x,r)\,dr\,dx\,dt}
.\end{multline*}
We may rearrange the terms of the integral to see that
\begin{multline*}
{\int_{\R^\dmn_+} t\,\partial_{j_\alpha} T_{\zeta_\alpha,\beta}\partial^\beta H(x,t)\,\int_t^\infty v_\alpha(x,r)\,dr\,dx\,dt}
\\\begin{aligned}
&=
	\int_0^\infty
	\int_{\R^n} \int_t^\infty \int_{\R^n} \frac{1}{t^n}(\partial_{j_\alpha}J_{\zeta_\alpha,\beta})\biggl(\frac{x-z}{t}\biggr) v_\alpha(x,r)\,dx\,dr \,\partial^\beta H(z,t)
	\,dz\,dt
.\end{aligned}\end{multline*}
We will use the Christ-Journ\'e $T1$ theorem (Theorem~\ref{thm:CJ} above) to bound
\begin{equation*}W_{\alpha,r}(z,t)=\int_{\R^n} \frac{1}{t^n}(\partial_{j_\alpha}J_{\zeta_\alpha,\beta})\biggl(\frac{x-z}{t}\biggr) v_\alpha(x,r)\,dx.\end{equation*}
Let $\psi_t(z,x)=\frac{1}{t^n}(\partial_{j_\alpha}J_{\zeta_\alpha,\beta}) \left(\frac{x-z}{t}\right) $. We then have that
\begin{equation*}\abs{\psi_t(z,x)}\leq \frac{C}{t^n},\qquad \abs{\nabla_x \psi_t(z,x)} \leq \frac{C}{t^\dmn}\end{equation*}
and both terms are zero if $\abs{x-z}>t/2$.
Finally, observe that $\int \psi_t(z,x)\,dx=0$ and so $\Theta_t1(x)=0$; thus, we have the estimate
\begin{equation*}\int_0^\infty \int_{\R^n} \abs{W_{\alpha,r}(z,t)}^2\,\frac{dz\,dt}{t} \leq C\doublebar{v_\alpha(\,\cdot\,,r)}_{L^2(\R^n)}^2.\end{equation*}
Thus,
\begin{multline*}
\abs[bigg]{\int_{\R^\dmn_+} t\,\partial_{j_\alpha} T_{\zeta_\alpha,\beta}\partial^\beta H(x,t)\,\int_t^\infty v_\alpha(x,r)\,dr\,dx\,dt}
\\\begin{aligned}
&=
	\abs[bigg]{\int_0^\infty \int_0^s
	\int_{\R^n} W_{\alpha,r}(z,t) \,\partial^\beta H(z,t)
	\,dz\,dt\,dr}
.\end{aligned}\end{multline*}
By the bound~\eqref{eqn:extension:square} and the above bound on $W_{\alpha,r}$,
\begin{equation*}
\abs[bigg]{\int_0^\infty \int_0^s
	\int_{\R^n} W_{\alpha,r}(z,t) \,\partial^\beta H(z,t)
	\,dz\,dt\,dr}
\leq
	\int_0^\infty
	\doublebar{v_\alpha(\,\cdot\,,r)}_{L^2(\R^n)}
	\doublebar{\arr g}_{L^2(\R^n)}
	\,dr
\end{equation*}
and by the bounds \eqref{eqn:wt:decay} and~\eqref{eqn:wt:decay:close},
\begin{align*}
\abs[bigg]{\int_{\R^\dmn_+} t\,\partial^\alpha \partial_t G(x,t)\,w_\alpha^-(x,t)\,dx\,dt}
&\leq
	\frac{C}{\kappa^{n/2} \sqrt{\abs{Q}}}\doublebar{\arr g}_{L^2(\R^n)}
,\\
\abs[bigg]{\int_{\R^\dmn_+} t\,\partial^\alpha \partial_t G(x,t)\,\widetilde w_\alpha(x,t)\,dx\,dt}
&\leq
	\frac{C\kappa}{\sqrt{\abs{Q}}} \doublebar{\arr g}_{L^2(\R^n)}
\end{align*}
as desired.
\end{proof}

We now have that $\arr b_Q^S$ is a bounded linear operator on $\dot W\!A^2_{m-1}(\R^n)$, a subspace of $L^2(\R^n)$.
We extend $\arr b_Q^S$ using a similar projection as in Section~\ref{sec:extension}; the difference in this case is that we use only two projection operators rather than countably many.

Let $W_n$ and $W_f$ be the closure in $L^2(\R^n)$ of, respectively,
\begin{align*}
\widetilde W_n &= \{\1_{(1/2)Q} \Tr_{m-1}\varphi + (1-\1_{(1/2)Q}) \arr f: \varphi \in C^\infty_0,\>\arr f\in L^2(\R^n)\}
,\\
\widetilde W_f &= \{\1_{(1/4)Q} \arr f + (1-\1_{(1/4)Q}) \Tr_{m-1}\varphi : \varphi \in C^\infty_0,\>\arr f\in L^2(\R^n)\}
.\end{align*}
Let $O_n$ and $O_f$ denote orthogonal projection from $L^2(\R^n)$ onto the subspaces $W_n$ and~$W_f$; observe that $O_n \arr f=\arr f$ outside of $(1/2)Q$ and that $O_f \arr f = \arr f$ inside $(1/4)Q$. Furthermore, $O_n (\Tr_{m-1}\varphi)=O_f (\Tr_{m-1}\varphi)=\Tr_{m-1}\varphi$ for any nice function~$\varphi$.

Let $\eta$ be smooth, supported in $(1/2)Q\times(-\ell(Q)/4,\ell(Q)/4)$ and identically equal to 1 in $(1/4)Q\times(-\ell(Q)/8,\ell(Q)/8)$, with the usual bounds on the derivatives of~$\eta$.

Define $\pi_n:W_n \mapsto \dot W\!A^2_{m-1}(\R^n)$ and $\pi_f:W_f \mapsto \dot W\!A^2_{m-1}(\R^n)$ as follows. Suppose that $\arr f=\Tr_{m-1}\varphi$ in $(1/2)Q$ or $\R^n\setminus(1/4)Q$ for some smooth function~$\varphi$.
We may renormalize $\varphi$ so that $\int_{(1/2)Q\setminus(1/4)Q} \Trace\partial^\zeta\varphi=0$ for all $\abs{\zeta}\leq m-1$. Let $\pi_n\arr f=\Tr_{m-1}(\eta\varphi)$ and let $\pi_f\arr f=\Tr_{m-1}((1-\eta)\varphi)$.
Observe that $\pi_n$ and $\pi_f$ are well-defined, that $\pi_n \arr f=0$ outside $(1/2)Q$ and that $\pi_f\arr f=0$ in $(1/4)Q$, and that by the Poincar\'e inequality $\pi_n:W_n \mapsto \dot W\!A^2_{m-1}(\R^n)$ and $\pi_f:W_f \mapsto \dot W\!A^2_{m-1}(\R^n)$ are bounded operators. Finally, notice that $\pi_n(\Tr_{m-1}\varphi)+\pi_f(\Tr_{m-1}\varphi)=\Tr_{m-1}\varphi$ for any smooth, compactly supported~$\varphi$.

We let $\arr b_Q^S$ satisfy
\begin{equation*}\langle\arr f, \arr b_Q^S \rangle_{\R^n} = \langle \pi_n O_n\arr f+ \pi_f O_f\arr f, \arr b_Q^S \rangle_{\R^n}\end{equation*}
where the right-hand side is given by formula~\eqref{eqn:b:S}. Notice that if $\arr f=\Tr_{m-1}\varphi$, then $O_n \arr f= O_f \arr f=\arr f$. Thus, this definition is consistent with formula~\eqref{eqn:b:S}.

By the bound~\eqref{eqn:b:L2:general} and boundedness of $\pi_n $ and $\pi_f $, we see that
\begin{equation*}\abs{\langle \arr f,\arr b_Q^S\rangle_{\R^n} }\leq \frac{C}{\kappa^{n/2}} \doublebar{\arr f}_{L^2(\R^n)}  \sqrt{\abs{Q}}\end{equation*}
and so the bound~\eqref{eqn:b:L2} is established with $C_0 = C \kappa^{-n}$. We are left with the bounds \eqref{eqn:b:below} and~\eqref{eqn:b:above}.

Observe that by the bound~\eqref{eqn:b:L2:close},
\begin{equation*}\abs{\langle \pi_f O_f \arr f,\arr b_Q^S\rangle_{\R^n} }\leq C \doublebar{\arr f}_{L^2(\R^n)} \kappa \sqrt{\abs{Q}}.\end{equation*}
Furthermore, if $\arr f=0$ in $(1/2)Q$ then $\pi_n O_n \arr f=0$; thus, we have that
\begin{equation}\label{eqn:b:far}\doublebar{\arr b_Q^S}_{L^2(\R^n\setminus(1/2)Q)} \leq C  \kappa \sqrt{\abs{Q}}.\end{equation}

Fix some~$\gamma$ with $\abs{\gamma}=m-1$, and let $b_Q^{\gamma}=(b_Q^S)_{\gamma}$ for some $\abs{\gamma}=m-1$. Then $b_Q^\perp=b_Q^{\gamma_\perp}$.
We seek to show that
\begin{equation*}
\re \frac{1}{\int_Q \phi_Q} \int_Q b_Q^{\gamma}(x) \phi_Q(x)\,dx\geq \sigma
\quad\text{or}\quad
\abs[bigg]{ \frac{1}{\int_Q \phi_Q} \int_Q b_Q^\gamma(x) \phi_Q(x)\,dx}
\leq\eta \sigma\end{equation*}
for some constant $\sigma$ independent of~$Q$ and some $\eta$ depending on~$C_1$. Notice that ${\int_Q \phi_Q}=c\abs{Q}$ for some constant~$c$ depending on $\phi_Q$, with $1/2\leq c \leq (1+\omega)^n$.

Let
\begin{equation*}\Phi_Q^\gamma(x,t) = \frac{1}{\gamma!} (x-y_Q,t)^\gamma \phi_Q(x)\,\rho(t)\end{equation*}
where $\rho(t)=1$ for $\abs{t}<\ell(Q)$ and $\rho(t)=0$ for $\abs{t}>2\ell(Q)$.

Notice that if $x\in (1/2)Q$ and $\abs{t}<\ell(Q)$, then
$\nabla^{m-1} \Phi_Q^\gamma(x,t)= \arr e_\gamma =\phi_Q(x) \arr e_\gamma$. In particular, $b_Q^\gamma \phi_Q = \Tr_{m-1}\Phi_Q^\gamma \cdot \arr b_Q$ in $(1/2)Q$. Furthermore, $b_Q^\gamma\, \phi_Q=0$ and $\Tr_{m-1}\Phi_Q^\gamma \cdot \arr b_Q=0$ outside of~$(1+\omega)Q$. Thus, by the bound~\eqref{eqn:b:far},
\begin{equation*}\abs[bigg]{
\frac{1}{\abs{Q}} \int_Q b_Q^\gamma(x) \phi_Q(x)\,dx
-\frac{1}{\abs{Q}} \int_{\R^n} \Tr_{m-1}\Phi_Q^\gamma\cdot \arr b_Q}
\leq
	\frac{C}{\abs{Q}} \int_{(1+\omega)Q\setminus(1/2)Q} \abs{\arr b_Q}
\leq C \kappa.\end{equation*}
Thus, to establish the bounds \eqref{eqn:b:below} and~\eqref{eqn:b:above}, it suffices to bound the quantity
\begin{equation*}\frac{1}{\abs{Q}}\int_{\R^n} \Tr_{m-1}\Phi_Q^\gamma\cdot \arr b_Q^S\end{equation*}
from above or from below.

Applying the definition of $\arr b_Q^S$, we see that
\begin{align*}
\frac{1}{\abs{Q}} \int_{\R^n} \Tr_{m-1}\Phi_Q^\gamma\cdot \arr b_Q
&=
	\int_{\R^\dmn_+} \nabla^m \Phi_Q^\gamma(x,t)\cdot \mat A(x)\nabla^m F_-(x,t)\,dx\,dt
	\\&\qquad
	+\int_{\R^\dmn_-} \nabla^m \Phi_Q^\gamma(x,t)\cdot \mat A(x)\nabla^m F_+(x,t)\,dx\,dt
.\end{align*}
Now, observe that $\nabla^m \Phi_Q^\gamma=0$ in $(1/2) Q \times (-\ell(Q),\ell(Q))$. Applying the definition of $F_\pm$ and the bounds \eqref{eqn:fundamental:slices:higher} or~\eqref{eqn:fundamental:slices:higher:2}, we see that
\begin{equation*}\abs[bigg]{\int_{\R^\dmn_-} \nabla^m \Phi_Q^\gamma(x,t)\cdot \mat A(x)
\bigl(\nabla^m F_+(x,t)-\nabla^m F_-(x,t)\bigr)\,dx\,dt}
\leq {C\kappa}
\end{equation*}
and so we may consider
\begin{equation*}
	\int_{\R^\dmn_+} \nabla^m \Phi_Q^\gamma\cdot \mat A\nabla^m F_-
	+\int_{\R^\dmn_-} \nabla^m \Phi_Q^\gamma\cdot \mat A\nabla^m F_-
=\int_{\R^\dmn} \nabla^m \Phi_Q^\gamma\cdot \mat A\nabla^m F_-
.\end{equation*}

Now, recall that
\begin{multline*}\int_{\R^\dmn} \nabla^m \Phi_Q^\gamma(x,t)\cdot \mat A\nabla^m F_-(x,t)\,dx\,dt
\\\begin{aligned}
&=
	\int_{\R^\dmn} \nabla^m \Phi_Q^\gamma(x,t)\cdot \mat A(x)\nabla_{x,t}^m
	 \partial_s^{m-1} E^L(x,t,y_Q,-\kappa\ell(Q))
	\,dx\,dt
.\end{aligned}\end{multline*}
Applying the symmetry property~\eqref{eqn:fundamental:symmetric}, we see that
\begin{multline*}\int_{\R^\dmn} \nabla^m \Phi_Q^\gamma(x,t)\cdot \mat A\nabla^m F_-(x,t)\,dx\,dt
\\\begin{aligned}
&=
	\int_{\R^\dmn} \overline{
	\partial_s^{m-1} \nabla^m_{x,t}
	 E^{L^*}(y_Q,-\kappa\ell(Q),x,t)
	\cdot
	\mat A^*(x)\nabla^m \Phi_Q^\gamma(x,t)
	 }
	\,dx\,dt
.\end{aligned}\end{multline*}
By formula~\eqref{eqn:fundamental:2},
\begin{equation*}\int_{\R^\dmn} \nabla^m \Phi_Q^\gamma(x,t)\cdot \mat A\nabla^m F_-(x,t)\,dx\,dt
=
	\overline{
	\partial_s^{m-1}  \Pi^{L^*}(\mat A^*\nabla^m \Phi_Q^\gamma)(y_Q,-\kappa\ell(Q))
	 }
.\end{equation*}
Recall (formula~\eqref{eqn:newton:identity}) that $\Pi^{L^*}(\mat A^*\nabla^m \Phi_Q^\gamma)=\Phi_Q^\gamma$. Thus
\begin{equation*}\int_{\R^\dmn} \nabla^m \Phi_Q^\gamma(x,t)\cdot \mat A\nabla^m F_-(x,t)\,dx\,dt
=
	\partial_s^{m-1}  \Phi_Q^\gamma(y_Q,-\kappa\ell(Q))
.\end{equation*}
The right-hand side is equal to one if $\gamma=\gamma_\perp$ and is zero otherwise, and so the bounds \eqref{eqn:b:below} and~\eqref{eqn:b:above} are established.

\section{Reduction to operators of higher order}
\label{sec:high}

We have now shown that $\Theta_t^D$ and $\Theta_t^S$ satisfy the bounds \eqref{eqn:Theta:decay} and~\eqref{eqn:Theta:Q}.
We have established that whenever $2m>n$,
the condition~\eqref{eqn:Theta:T1} is valid, and there exist functions $\arr b_Q$ such that the conditions \eqref{eqn:b:carleson}, \eqref{eqn:b:L2}, \eqref{eqn:b:below} and~\eqref{eqn:b:above} are valid.

Thus, if $2m>n$, then by Theorem~\ref{thm:grau:hofmann}, $\Theta_t^D$ and $\Theta_t^S$ satisfy the bound~\eqref{eqn:Theta:square}; this implies that the bounds \eqref{eqn:S:square} and~\eqref{eqn:D:tilde:square} are valid.

We now must establish these bounds for operators of order $2m\leq n$.
We use a fairly standard technique in the theory of higher-order differential equations; see \cite[Section~2.2]{AusHMT01} and \cite[Section~5.4]{Bar14p}.

Fix some operator $L$ of order $2m\leq n$, and choose some number $M$ such that $2m+4M>n$. Now, there are constants $a_\zeta$ such that
\begin{equation*}\Delta^M = \sum_{\abs{\zeta}=M} a_\zeta\, \partial^{2\zeta}.\end{equation*}
In fact, $a_\zeta = m!/\zeta!$, and so we have that $a_\zeta\geq 1$ for all $\abs{\zeta}=M$.

Define the differential operator $\widetilde L=\Delta^M \! L\Delta^M$; that is, $\langle \varphi, \widetilde L \psi\rangle = \langle \Delta^M \varphi, L \Delta^M\psi\rangle$ for all nice test functions $\varphi$ and~$\psi$. We remark that $\widetilde L$ is associated to coefficients~$\widetilde{\mat A}$ that satisfy
\begin{equation}\label{eqn:A:higher}
\widetilde A_{\delta\varepsilon}(x)
= \sum_{\substack{\alpha+2\zeta=\delta\\\beta+2\xi=\varepsilon}} a_\zeta\,a_\xi\, A_{\alpha\beta}(x)
=
	\sum_{\substack{\abs{\zeta}=M,\>2\zeta<\delta \\\abs{\xi}=M,\>2\xi<\varepsilon}} a_\zeta\,a_\xi\, A_{(\delta-2\zeta)(\varepsilon-2\xi)}(x)
.\end{equation}
Observe that $\widetilde{\mat A}$ is $t$-independent and satisfies the bounds \eqref{eqn:elliptic} and~\eqref{eqn:elliptic:bounded}. It was shown in the proof of \cite[Theorem~62]{Bar14p} that
\begin{equation*}E^L(x,t,y,s) = \sum_{\abs{\zeta}=\abs{\xi}=M} a_\zeta \, a_\xi \, \partial^{2\zeta}_{x,t} \partial^{2\xi}_{y,s} E^{\widetilde L}(x,t,y,s).\end{equation*}

%We define $\widetilde\Theta_t^D$ and $\widetilde\Theta_t^S$ analogously to~$\Theta_t^D$ and~$\Theta_t^S$, with the operator and coefficients $\widetilde L$, $\widetilde{\mat A}$ in place of the operator $L$ and coefficients~$\mat A$. By the above argument, $\widetilde\Theta_t^S$ and $\widetilde\Theta_t^D$ satisfy the bound~\eqref{eqn:Theta:square}.

Now, by formula~\eqref{eqn:S:fundamental}, if $\abs{\alpha}=m$ then
\begin{align*}
\partial^\alpha \s^{\mat A} \arr g(x,t)
&=
	\sum_{\abs{\gamma}=m-1} \int_{\R^n} \partial_{x,t}^\alpha \partial_{y,s}^\gamma E^L(x,t,y,0)\,g_{\gamma}(y)\,dy
\\&=
	\sum_{\abs{\gamma}=m-1}
	\sum_{\abs{\zeta}=M}\sum_{\abs{\xi}=M}
	a_\zeta \int_{\R^n} \partial^{\alpha+2\zeta}_{x,t} \partial^{\gamma+2\xi}_{y,s} E^{\widetilde L}(x,t,y,0)\, a_\xi  \,g_{\gamma}(y)\,dy
.\end{align*}
Let $\widetilde g_\varepsilon(x) = \sum_{\gamma+2\xi=\varepsilon} a_\xi\,g_\gamma(x)$. Notice that $\abs{\arr{\widetilde g}(x)}\leq C \abs{\arr g(x)}$. Then
\begin{align}%\label{eqn:S:high}
\partial^\alpha \s^{\mat A} \arr g(x,t)
&=
	\sum_{\abs{\zeta}=M}
	a_\zeta \partial^{\alpha+2\zeta} \s^{\widetilde{\mat A}} \arr{ \widetilde{g}}(x,t)
.\end{align}
Thus, because the bound~\eqref{eqn:S:square} is valid for operators $\widetilde L$ of order $2m+4M$ for $M$ large enough, we have that
\begin{align*}
\int_{\R^n}\int_0^\infty \abs{\nabla^m \partial_t\s^{\mat A} \arr g(x,t)}^2\,{t}\,dt\,dx
	&\leq C \int_{\R^n}\int_0^\infty \abs{\nabla^{m+2M} \partial_t\s^{\widetilde{\mat A}} \arr{\widetilde g} (x,t)}^2\,{t}\,dt\,dx
	\\&\leq C \doublebar{\arr {\widetilde g}}_{L^2(\R^n)}^2
	\leq C \doublebar{\arr g}_{L^2(\R^n)}^2
\end{align*}
and so the bound~\eqref{eqn:S:square} is valid even for operators of order $2m\leq n$.

The argument for $\D^{\mat A}$ is somewhat more involved. In this case we will use Theorem~\ref{thm:grau:hofmann:1}; observe that $\Theta_t^D$ satisfies the bounds \eqref{eqn:Theta:decay} and~\eqref{eqn:Theta:Q}, and so we need only establish the bound~\eqref{eqn:Theta:1:carleson}, that is, to bound $\Theta_t^D \arr e_\beta$ for multiindices~$\beta$ with $\abs{\beta}=m$.
%Define $\widetilde\Theta_t^D$  analogously to~$\Theta_t^D$ with the operator and coefficients $\widetilde L$, $\widetilde{\mat A}$ in place of the operator $L$ and coefficients~$\mat A$. By the above argument, $\widetilde\Theta_t^D$ satisfies the bound~\eqref{eqn:Theta:square}. Then by Lemma~\ref{lem:square:carleson}, we have that $\widetilde\Theta_t^D 1$ satisfies the Carleson-measure estimate~\eqref{eqn:Theta:1:carleson}.

Recall from formula~\eqref{eqn:D1} that
\begin{align*}
\Theta_t^D \arr e_{\beta}(x)
&=
	-\sum_{\abs{\alpha}=m} t^k\int_{\R^\dmn_-} \partial_t^{m+k} \partial_{y,s}^\alpha E^L(x,t,y,s)\,A_{\alpha\beta}(y) \,ds\,dy
\end{align*}
and so
\begin{align*}
\Theta_t^D \arr e_{\beta}(x)
&=
	-\sum_{\abs{\zeta}=M}a_\zeta\sum_{\abs{\delta}=m+2M} t^k\int_{\R^\dmn_-} \partial_t^{m+k}
	\partial^{2\zeta}_{x,t} \partial^{\delta}_{y,s} E^{\widetilde L}(x,t,y,s)\,B_{\delta\beta}(y) \,ds\,dy
\end{align*}
where
\begin{align*}
B_{\delta\beta}(y)
&=\sum_{\alpha+2\xi=\delta} a_\xi\, A_{\alpha\beta}(y)
= \sum_{\abs{\xi}=M,\>2\xi<\delta} a_\xi \,A_{(\delta-2\xi)\beta}(y).
\end{align*}
Recall our formula~\eqref{eqn:A:higher} for the coefficients $\widetilde A$ of~$\widetilde L$.
We then have that
\begin{equation*}\widetilde A_{\delta\varepsilon}(y)=\sum_{{\abs{\xi}=m,\>2\xi<\varepsilon}} a_\xi\,B_{\delta(\varepsilon-2\xi)}(y).\end{equation*}
We claim that there exist numbers $b_{\beta\varepsilon}$ such that
\begin{equation*}B_{\delta\beta}
=\sum_{\abs{\varepsilon}=m+2M} b_{\beta\varepsilon} \, \widetilde A_{\delta\varepsilon}.\end{equation*}
To see this, define the operator $\Psi$ by $\Psi(\arr F)_\varepsilon=\sum_{{\abs{\xi}=M,\>2\xi<\varepsilon }} a_\xi\, F_{\varepsilon-2\xi}$. It suffices to show that we may recover $\arr F$ from $\Psi(\arr F)$.
Begin with indices $\alpha$ such that $\alpha_j\leq 1$ for all but one value $j_0$ of~$j$. Let $\xi=M\vec e_{j_0}$ and let $\varepsilon=\alpha+2\xi$. Then $\xi$ is the only multiindex with $\abs{\xi}=M$ and with $2\xi<\varepsilon$. Thus $\Psi(\arr F)_\varepsilon = a_\xi \, F_\alpha = F_\alpha$; in other words, we can recover $F_\alpha$ from $\Psi(\arr F)$. Next, consider indices $\alpha$ with $\alpha_{j_0}$ arbitrary, $2\leq\alpha_{j_1}\leq 3$, and $\alpha_j\leq 1$ for all other values of~$j$. Let $\xi$ be as before and let $\varepsilon=\alpha+2\xi$. Then $\Psi(\arr F)_\varepsilon = a_\xi \, F_\alpha + a_\zeta \, F_{\varepsilon-2\zeta}$, where $\zeta=\vec e_{j_1}+(M-1)\vec e_{j_0}$. Since $F_{\varepsilon-2\zeta}$ is known we may recover $F_\alpha$. Continuing in this fashion, we may recover $F_\alpha$ from $\Psi(\arr F)$ for all multiindices~$\alpha$.

Thus
\begin{equation*}
\Theta_t^D \arr e_{\beta}(x)
=
	-\sum_{\abs{\zeta}=M}a_\zeta
	\sum_{\substack{\abs{\delta}=m+2M \\ \abs{\varepsilon}=m+2M}} b_{\beta\varepsilon}\,
	t^k\int_{\R^\dmn_-} \partial_t^{m+k}
	\partial^{2\zeta}_{x,t} \partial^{\delta}_{y,s} E^{\widetilde L}(x,t,y,s)\, \widetilde A_{\delta\varepsilon}(y) \,ds\,dy
.\end{equation*}
By formula~\eqref{eqn:D:fundamental} for $\D$ and~\eqref{dfn:D:tilde} for $\widetilde\D$, this equals
\begin{equation*}
\Theta_t^D \arr e_{\beta}(x)
=
	\sum_{\abs{\zeta}=M}a_\zeta\sum_{\abs{\varepsilon}=m+2M} b_{\beta\varepsilon}\,
	t^k\partial_t^{m+k} \,
	\partial^{2\zeta}_{x,t}
	\widetilde D^{\mat{\widetilde A}}\arr e_\varepsilon(x,t).
\end{equation*}
Recall from formula~\eqref{dfn:Theta:D} that
$\Theta_t^D \arr f(x) = t^k \partial_t^{m+k} \widetilde D^{\mat A}\arr f(x,t)$. Define
\begin{equation*}
\widetilde\Theta_t^D \arr f(x) = t^{k'} \partial_t^{m+2M+k'} \widetilde D^{\widetilde {\mat A}}\arr f(x,t)
\end{equation*}
for some $k'$ to be chosen momentarily.

Because $m+2M>n$, if $k'$ is large enough then we have that $\widetilde\Theta_t^D$ satisfies the estimates \eqref{eqn:Theta:square} and~\eqref{eqn:Theta:decay}, and so by Lemma~\ref{lem:square:carleson},
\begin{equation*}\sup_Q\frac{1}{\abs{Q}}\int_Q \int_0^{\ell(Q)} \abs{t^{k'} \partial_t^{m+2M+k'} \widetilde D^{\widetilde {\mat A}}\arr e_\varepsilon(x,t)}^2\frac{dt\,dx}{t}\leq C.\end{equation*}

Fix some cube~$Q$ and observe that
\begin{equation*}\int_0^{\ell(Q)}\int_Q \abs{\Theta_t^D \arr e_\beta(x)}^2 \frac{dx\,dt}{t}
\leq
C\sum_{\abs{\varepsilon}=m+2M}
\int_0^{\ell(Q)}\int_Q \abs{t^k\partial_t^{m+k}
	\nabla^{2M} \widetilde D^{\widetilde{\mat A}} \arr e_\varepsilon(x,t)}^2 \frac{dx\,dt}{t}
.\end{equation*}
Applying the Caccioppoli inequality in Whitney boxes, we see that
\begin{equation*}\int_0^{\ell(Q)}\int_Q \abs{\Theta_t^D \arr e_\beta(x)}^2 \frac{dx\,dt}{t}
\leq
C\sum_{\abs{\varepsilon}=m+2M}
\int_0^{2\ell(Q)}\!\int_{2Q} \abs{t^{k-2M}\partial_t^{m+k}
	\widetilde D^{\widetilde{\mat A}} \arr e_\varepsilon(x,t)}^2 \frac{dx\,dt}{t}
.\end{equation*}
If we let $k=2M+k'$, we see that
\begin{equation*}\int_0^{\ell(Q)}\int_Q \abs{\Theta_t^D \arr e_\beta(x)}^2 \frac{dx\,dt}{t}\leq C\abs{Q}\end{equation*}
and so $\Theta_t^D$ satisfies the bound~\eqref{eqn:Theta:1:carleson}. Thus, by Theorem~\ref{thm:grau:hofmann:1} we have that $\Theta_t^D$ satisfies the bound~\eqref{eqn:Theta:1:square}. Thus, by Lemma~\ref{lem:square:Theta} we have that $\widetilde D^{\mat A}$ satisfies the bound~\eqref{eqn:D:tilde:square}, as desired.

% Change to \bibliographystyle{amsplain} for numbers
%\bibliographystyle{amsalpha}
%\bibliography{bibli}

\newcommand{\etalchar}[1]{$^{#1}$}
\providecommand{\bysame}{\leavevmode\hbox to3em{\hrulefill}\thinspace}
\providecommand{\MR}{\relax\ifhmode\unskip\space\fi MR }
\providecommand{\MRhref}[2]{%
  \href{http://www.ams.org/mathscinet-getitem?mr=#1}{#2}
}
\providecommand{\href}[2]{#2}

\end{document}